\documentclass{article}
\usepackage[utf8]{inputenc}
\usepackage[english]{babel}
\usepackage[a4paper,left=1in,top=1in,right=1in,bottom=1in,nohead]{geometry}
\usepackage{amsmath,amsfonts,amsthm,amssymb,mathtools,bbm,bm}
\usepackage{appendix}
\usepackage{bigints}
\usepackage[nocompress]{cite}
\usepackage{graphicx}
\usepackage{blkarray}
\usepackage{tikz}
\usepackage{circuitikz}
\usepackage{cleveref}
\usepackage{placeins}
\usepackage[margin=0.9375cm,format=hang, font={small,it} , labelfont=normalfont,bf]{caption}
\usetikzlibrary{automata, positioning, calc, arrows.meta}
\usepackage{algorithm2e}
\RestyleAlgo{ruled}
\makeatletter
\renewcommand{\@algocf@capt@plain}{above}
\makeatother

\makeatletter

\let\c@table\c@figure
\makeatother 

\Crefname{algocf}{Algorithm}{Algorithms}

\numberwithin{equation}{section}
\numberwithin{figure}{section}
\numberwithin{table}{section}

\theoremstyle{plain}
\newtheorem{thm}{Theorem}[section]
\crefname{thm}{theorem}{theorems}
\theoremstyle{definition}
\newtheorem{defn}[thm]{Definition}
\theoremstyle{definition}

\theoremstyle{plain}
\newtheorem{prop}[thm]{Proposition}
\theoremstyle{plain}

\theoremstyle{plain}
\newtheorem{lemma}[thm]{Lemma}
\theoremstyle{plain}
\newtheorem{ass}[thm]{Assumption}
\crefname{ass}{assumption}{assumptions}
\theoremstyle{plain}
\newtheorem{rem}[thm]{Remark}
\theoremstyle{definition}

\theoremstyle{definition}

\newcommand{\E}{\mathbb{E}}
\newcommand{\R}{\mathbb{R}}
\newcommand{\N}{\mathbb{N}}
\newcommand{\PP}{\mathbb{P}}
\newcommand{\Pcal}{\mathcal{P}}
\newcommand{\de}{\mathrm{d}}
\newcommand{\pr}[1]{{#1}^{\prime}}

\newcommand{\JTstar}{J^{T,*}}
\newcommand{\QTstar}{Q^{T,*}}

\newcommand{\esssup}[1]{\mathrm{ess\,sup}_{#1}}
\newcommand{\essinf}[1]{\mathrm{ess\,inf}_{#1}}

\title{Mean-field games of speedy information access\\ with observation costs}
\author{Dirk Becherer\thanks{Institute of Mathematics, Humboldt Universit\"at zu Berlin \newline\hspace*{1.8em}(\texttt{becherer@math.hu-berlin.de})} \and Christoph Reisinger\thanks{Mathematical Institute, University of Oxford \newline\hspace*{1.8em}(\texttt{reisinge@maths.ox.ac.uk}, \texttt{tam@maths.ox.ac.uk})} \and Jonathan Tam\footnotemark[2]}
\date{July 5, 2026}

\begin{document}

\maketitle

\abstract{We investigate mean-field games (MFG) in which agents actively control their speed of access to information. Specifically, the agents can dynamically decide to obtain observations with reduced delay by accepting higher observation costs. Agents seek to exploit their active information acquisition by making further decisions to influence their state dynamics so as to maximise rewards. In a mean-field equilibrium, each generic agent solves individually a partially observed Markov decision problem in which the way partial observations are obtained is itself subject to dynamic control actions, while no agent can improve unilaterally given the actions of all others. We formulate the mean-field game with controlled costly information access as an equivalent standard mean-field game on an augmented space, by utilizing a parameterisation of the belief state by a finite number of variables. With sufficient  entropy regularisation, a fixed point iteration converges to the unique MFG equilibrium. Moreover, we derive an approximate $\varepsilon$-Nash equilibrium for a large but finite population size and small regularisation parameter. We illustrate our (extended) MFG of information access and of controls by an example from epidemiology, where medical testing results can be procured at different speeds and costs. 
}

\section{Introduction} 

In dynamic decision making, one often has the opportunity to improve the quality of one’s observations by expending extra resources. Balancing information acquisition against its associated costs is therefore central to optimising rewards. In many applications, actions affect not only outcomes but also the timing and quality of future observations. Medical laboratories may invest in resources to reduce test turnaround times; investors may delay trades while gathering or processing information; costs and waiting times associated with different health service options may affect individual choices
on personal and societal health. Decisions can moreover be affected by perceived collective behaviour. While some decisions are more static in nature (e.g. investing in infrastructure), some allow for dynamic choices (e.g. paying for faster test results). We focus on the latter, mathematically and computationally more challenging situation.\\

With likewise examples in mind, we introduce a novel mean-field game (MFG) model in discrete time, in which agents actively control their speed of access to information. During the game, agents can adjust their information access speed by respective costly efforts, and exploit their dynamic information flow to control their state dynamics to maximise expected rewards. This gives rise to a problem of partial observations, in which the flow of information is not exogenously given but instead dynamically controlled by the agents. Consequently, the information flow is endogenously obtained as part of a mean-field equilibrium. We model the speed of information access in the MFG through observations which are delayed to varying degrees, and utilise the information structure to construct a suitable augmentation of the state space. The augmented space includes past actions taken within the dynamic delay period, and serves a finite-dimensional parametrisation of an equivalent MFG in the belief state. Thereby, numerical schemes for discrete MFGs can be employed to compute approximate mean-field Nash equilibria (MFNE).\\

This paper covers three themes: (1) the dynamic control of information access by observation delay, (2) observation costs, and (3) the analysis of associated MFGs incorporating the said two features. Standard Markov decision processes (MDP) assume that state observations are received instantaneously. This limits the applicability of MDPs in many real-life situations, where it is often the case that observation delay arises due to inherent features or practical limitations of a system.\\

\emph{Information delays.} There has been a large amount of literature involving the modelling of observation delays, with applications in (but not limited to) network communications \cite{adlakha2008information, adlakha2011networked}, quantitative finance \cite{cartea2023optimal,abel2013optimal,bruder2009impulse, oksendal2008optimal}, medical diagnosis \cite{rong_covid19} and reinforcement learning \cite{chen2021delay, schuitema2010control, stochastic_delays_RL}. Most models involve an MDP framework with either a constant or random observation delay, both of which are exogenously given by the system. Both constant and random observation delay MDPs can be modelled as a partially observable MDP (POMDP) via state augmentation \cite{bander1999markov, altman_nain, asynchronous_collection, stochastic_delays_RL}. It has also been shown that action delays can be considered as a form of observation delay, under a suitable transformation of the MDP problem \cite{asynchronous_collection}. The continuous-time counterpart with an associated HJB-type master equation has been studied in \cite{saporito_zhang}.\\

\emph{Information costs.} In the MDP framework, the information source upon which control decision can depend on is fixed \textit{a priori} and thus exogenously given. As already pointed out by Bismut \cite{bismut1978introductory}, however, in many relevant applications the observations available at a time would naturally depend on the controls exercised previously, whereas formulations of standard theory “\textit{assume $\sigma$-fields $\mathcal{F}_t$ do not depend on the control, i.e., the information available at time t does not vary with the control u. This is clearly counterintuitive}''. Indeed, further applications have since emerged over the years that feature control-dependent observations, particularly in resource-constrained environments, where frequent measurements or sampling are either too expensive, laborious or impractical. Applications include efficient medical treatment allocation \cite{winkelmann_markov_2014},
environmental management \cite{YOSHIOKA_1_impulse_control, Yoshioka_2_river_impulse, Yoshioka_3_monitoring_deviation, YOSHIOKA_4_biological_stopping},
optimal inattention in investment behaviour of individual consumers \cite{abel2013optimal},
communications sampling \cite{paging_registration, sampling_guo_2021}, optimal sensing \cite{sensor_energy_harvesting, wu2008optimal,tzoumas2020lqg}, reinforcement learning \cite{bellinger2020active, krueger2020active, bellinger2022}, and more. Recently, \cite{gu2026information} proposed a theory for information control in which the filtration itself is the (directly) controlled object.

We shall refer to problems where the user can opt to receive some observation about the current state of the process, at the price of an observation cost which is included in the reward functional to be optimised, as observation cost models (OCMs). An OCM can equivalently be characterised as a POMDP, by including the time elapsed together with the last states observed and actions applied to form an augmented Markov system. In many cases, a reasonable simplification is to assume constant actions between observations \cite{huang2021self}, or open-loop controls parameterised by finite number of variables \cite{reisinger2022markov}. This leads to a finite dimensional characterisation of the augmented state, and allows efficient computation of the resulting system of quasi-variational inequalities by a penalty scheme \cite{reisinger2022markov}. Analysis for the more general case of non-constant actions has generally been restricted to linear-quadratic Gaussian problems \cite{wu2008optimal, tzoumas2020lqg, cooper1971optimal}.\\

\emph{Mean-field games.} In stochastic games, the computation of Nash equilibria is often intractable for a large number of players. Mean-field games (MFGs), first introduced in \cite{lasry2007mean} and \cite{caines2006large}, provide a way of seeking approximate Nash equilibria, by assuming symmetry between agents that can be modelled by a mean-field term, in the form of a measure flow. MFGs can be regarded as asymptotic approximations for games with a large number of interacting players. Finding a mean-field Nash equilibrium (MFNE) amounts to a search for an optimal policy for a representative player, while ensuring that the state distribution of this player under such a policy is consistent with the postulated  evolution of distributions of the other players, described by the measure flow. It is also possible to model interaction of the players through the empirical distribution of both their states and their controls \cite{lauriere_tangpi2022,cardaliaguet2018mean}. Such formulations in the mean-field regimes are 
sometimes 
also referred to as extended MFGs or as MFG of controls \cite{gomes2016extended,carmona2018probabilistic,cardaliaguet2018mean}.

In discrete time, the existence of MFNE has been established in \cite{saldi_mfg_discrete}. Analysis has also appeared for several problem variants involving risk-sensitive criteria \cite{saldi2022partially}, partially observable systems \cite{saldi2019approximate,saldi2022partially} or unknown reward and transition dynamics \cite{guo2019learning}. However, MFGs may suffer from non-uniqueness of MFNE and non-contractivity of a naively iterated fixed point algorithm \cite{cui2021approximately}. Several algorithms have emerged to address efficient computations of MFNEs. Entropy regularisation exploits a duality between convexity and smoothness to achieve contractivity, either by incorporating entropic penalties directly into the reward functional, or by imposing softmax policies during the optimisation step \cite{geist2019theory, saldi_regularisation, cui2021approximately}. Fictitious play schemes aim to smooth mean-field updating by averaging for new iterates over the past mean-field terms, effectively damping  updates to aid numerical convergence  \cite{perrin2020fictitious}. Online mirror descent further decreases computational complexity by replacing best response updates with direct $Q$-function computations \cite{perolat2021scaling}. In contrast, \cite{mf-omo} reformulates the problem to find MFNE to an equivalent optimisation problem, allowing a possible search for multiple MFNE with standard gradient descent algorithms. We refer to \cite{lauriere2022learning} for a comprehensive overview of the above algorithms.\\

\noindent \textbf{Our work.} We model agents' strategic choices about their speed of information access in the game by a novel MFG in which the dynamic information access 
in itself is subject to costly controls. Throughout the paper, we assume that both the state and action spaces are finite. The agents participating in the game exercise control over two aspects: actions that directly influence their state transition dynamics and rewards (as usual in MFG), and the temporal duration of their access speed to information, modelled as an observation delay. Agents dynamically choose from a given finite set of delay periods, each value of which corresponds to some observation cost. To acquire more timely observations, higher information costs are incurred, and vice versa.\\

Our framework here differs from existing works in that the observation delay is not exogenously given, as in the constant case \cite{altman_nain, bander1999markov}, nor is it a given random variable as in the stochastic case \cite{stochastic_delays_RL, cartea2023optimal, chen2021delay}. Instead, the length of the delay is dynamically and actively decided by the agent, based on a trade-off between the extra cost versus the accuracy of more speedy observations, the latter of which can be exploited though better informed control of the dynamics and hence higher rewards. The choice of the delay period becomes an extra part of the control in the optimisation problem in tandem with the agent's actions. When considering this as a single agent problem, as it occurs during the optimisation step where the measure flow is given, we refer to it as a Markov controllable delay model (MCDM). The MCDM can be reformulated in terms of a POMDP, by augmenting the state with the most recent observation and the actions taken thereafter, to constitute a Markovian system in a fully observable form. This permits for dynamic programming in the respective augmented state space to obtain the Bellman equation.\\

When viewed as part of the overall MFG, the partial information structure of the problem implies that the measure flow should be specified on the augmented space for the fixed point characterisation of the MFNE. However, the underlying transition dynamics and reward structure depend on the distribution of the (non-observed) states at the present time. In the models of \cite{saldi2019approximate, saldi2022partially} with partial information being of quite different structure, given by noisy observations, the barycenter map is used to map measures on a suitably augmented space to the respective measures on the underlying states. However, our model here differs in two aspects. Firstly, our belief state (posterior distribution of non-observed states) admits a description by finitely many variables, so that the equivalent fully-observed MFG is established on a finite-dimensional augmented space, which is motivated by and helps for numerical computation. Secondly, due to the structure from actively controlled information delay, the observation kernel depends on the distribution of the states throughout each moment in time across the entire delay period. Thus, taking an average of a distribution over the augmented space of parameters, as a barycenter map would do, is not applicable here. Instead, by using the delay structure, we explicitly 
calculate the required map from a measure flow on the augmented space to a sequence of measures on the underlying (non-augmented) states. Intuitively, this corresponds to an agent estimating the distribution of the current states of the population, given the observations they possess. We detail the construction of the MCDM in \Cref{sec:single_agent} and the corresponding MFG formulation, which we will also refer to as the MFG-MCDM, with its MFNE definition in \Cref{sec:MFG}.\\

The second part of this paper, from Section \ref{sec:regularise}, focuses on the computation of an MFNE for the MFG-MCDM. We employ the seminal (relative) entropy regularisation technique, which aids convergence of the classical iterative scheme: computing an optimal policy (of a single player) as a best response to a given measure flow, followed by computing the distribution (in augmented state) resulting from said policy. In the standard MFG model, it 
has been shown that the fixed point operator for the regularised problem is contractive under mild conditions \cite{saldi_regularisation, cui2021approximately}. This forms the basis of the prior descent algorithm, which is one of the current state-of-the-art algorithms for the computation of approximate Nash equilibria for MFGs \cite{cui2021approximately,mfglib}. We prove that for our MFG-MCDM, the corresponding fixed point operator also forms a contraction, provided it is sufficiently regularised by an entropic penalty term.\\

Apart from actively controlled information access, which is the main contribution, our analysis with compact Polish state and action spaces also extends some results due to
\cite{cui2021approximately, saldi_regularisation} about entropy regularized MFGs in finite spaces. Our infinite horizon discounted cost problem with time-dependent measure flows also complements the treatment in \cite{cui2021approximately} for finite horizon problems, and the result in \cite{saldi_regularisation} for infinite horizon problems where the measure flows are taken to be stationary. As the MFG-MCDM is a problem with partial observations, the proof also requires a crucial extra step to demonstrate that the aforementioned mapping of the measure flow on the augmented space to that on the underlying space is Lipschitz continuous, in order to conclude the desired fixed point contraction.\\

The contributions of this article can be summarised as follows.

\begin{enumerate}
    \item We show dynamical programming for a Markov controllable delay model (MCDM), an MDP model where an individual agent can exercise dynamic control over the latency of their observations, with less information delay being more costly. The problem is cast in terms of a partially observed MDP (POMDP) with controlled but costly partial observations, for which the belief state can be parameterised by a finite number of variables. Solving this POMDP is shown to be equivalent to solving an MDP on an augmented state space, whose extension also involves past actions taken during the (non-constant but dynamically controlled) delay period.
    \item We introduce a corresponding mean-field game (MFG) where speedy information access is subject to the agents' strategic control decisions. For a fixed measure flow, which describes the statistical population evolution, the ensuing single agent control problem becomes an MCDM. Although a mean-field Nash equilibrium (MFNE) is defined in terms of the augmented space, the underlying dynamics and rewards still depend on the underlying state distribution. We show how a measure flow on the underlying space is determined and computed from that of the augmented space. This construction exploits our parameterisation of the belief state; whereas the barycenter approach in \cite{saldi2019approximate} for a measure-valued belief state does not apply here.
    \item   
    By using a sufficiently strong (relative) entropy regularisation in the reward functional, we prove that the regularised MFG-MCDM has a unique MFNE, which is described by a fixed point and can serve as an approximate Nash equilibrium for a large but finite population size. Our setting with the regularised MFG in Polish spaces contributes to a line of work by \cite{saldi_regularisation, cui2021approximately}, whose articles are using finite state and action spaces. Results extend to a MFG formulated on infinite horizon with non-stationary but time-dependent measure flows. We also outline a straightforward extension of the results above to the case of MFG-MCDMs with extended mean field interaction through both the controls and the states. 
    \item We illustrate our model by an example from epidemiology, in which we compute both qualitative effects from actively controlled information access and respective costs to the equilibrium, and also the quantitative properties of convergence relating to the entropy regularisation. For the computation, we employ the Prior Descent algorithm \cite{cui2021approximately}, using the \texttt{mfglib} Python package \cite{mfglib} as a basis for the MFG-MCDM. This example also demonstrates the extension of MFG of information access to (extended) MFG of controls \cite{gomes2016extended,cardaliaguet2018mean}, as it is derived in \Cref{sec:extended}.
\end{enumerate}
~\\
The remainder of the paper is organised as follows. \Cref{sec:single_agent} formulates the single agent problem, and establishes the MCDM in terms of a POMDP on a finitely augmented space. \Cref{sec:MFG} sets up the corresponding MFG-MCDM, as well as the fixed point characterisation for an MFNE. \Cref{sec:regularise,sec:approximate} establishes uniqueness of MFNE for a sufficiently regularised MFG-MCDM, and that a regularised MFNE yields an approximate equilibrium for the finite player game. \Cref{sec:extended} briefly outlines how the aforementioned results extend also to so-called extended MFG, where interaction occurs also through the controls. Finally, \Cref{sec_numerics} demonstrates a numerical example based on epidemiology, which illustrates the effects of the cost on information access speed on the population behavior at equilibrium.

\subsection{Notation and preliminaries} \label{sec:notation}

Let $E$ be a Polish space, that is $E$ is a separable and complete metric space (we take a suitable metric under which completeness holds to be already chosen). Being equipped with the Borel sigma algebra $\mathcal{B}(E)$, then $(E, \mathcal{B}(E))$ is a Borel space. Let $C_b(E)$ denote the set of continuous and bounded functions on $E$, and let $\Pcal(E)$ denote the space of probability measures on $E$. We say a sequence of measures $\{\mu_n\}$ converges to $\mu$ weakly if $\int f\de \mu_n \to \int f \de \mu$ for all $ f \in C_b(E)$. If $E$ is separable, the weak convergence of measures is equivalent to the convergence in the L\'evy--Prokhorov metric
\begin{align*}
    \delta_{\rm Pr}(\mu, \nu) = \inf\{ \epsilon > 0: \nu(B) \leq \mu(B^{\varepsilon})+ \varepsilon, \ \mu(B) \leq \nu(B^{\varepsilon})+ \varepsilon,\ \forall B \in \mathcal{B}(E) \},
\end{align*}
where $B^{\varepsilon}$ denotes the $\varepsilon$--neighbourhood of the set $B$. Furthermore, by Prokhorov's theorem (see, e.g., \cite[Section 8.6]{bogachev2007measure}), if $E$ is compact then the space $(\Pcal(E), \delta_{\rm Pr})$ is compact. We shall also consider the total variation metric $\delta_{TV}$, induced by the total variation norm on the space of signed measures. That is,
\begin{align*}
    \delta_{TV}(\mu, \nu) = \lVert \mu - \nu \rVert_{TV} = \sup_{A\in \mathcal{B}(E)} \lvert \mu(A) - \nu(A) \rvert.
\end{align*}
For any measurable space $E$, the space $(\Pcal(E), \delta_{TV})$ is a closed subset of the Banach space of finite signed measures on $E$, hence it itself is also complete. If both $\mu$ and $\nu$ are absolutely continuous with respect to a measure $m$, with densities $p$ and $q$ respectively, then we also have
\begin{align*}
    \delta_{TV}(\mu, \nu) = \frac{1}{2}\int\lvert p(x) - q(x) \rvert m(\de x).
\end{align*}

Given two Borel spaces $E$ and $\hat{E}$, a Markov kernel on $\hat{E}$ given $E$ is a function $q: E \times \mathcal{B}(\hat{E}) \to [0,1]$ with the properties that for each $e \in E$, $q(e, \cdot)$ is a probability measure on $(\hat{E}, \mathcal{B}(\hat{E}))$, and for each $\mathcal{E} \in  \mathcal{B}(\hat{E})$, $q(\cdot, \mathcal{E})$ is measurable.
We will often consider a Markov kernel $q$ as a function $\tilde{q}$ from $E$ to $\Pcal(\hat{E})$ via $\tilde{q}(e) = q(e, \cdot)$. Since we are concerned with the state transitions of MDPs, we will also refer to a Markov kernel as a transition kernel.

\begin{defn}
    A Markov decision process (MDP) is a tuple $\langle \mathcal{X}, A, p, r \rangle$, where
    \begin{itemize}
        \item the \textit{state space} $\mathcal{X}$ and the \textit{action space} $A$ are non-empty Polish spaces, each equipped with its Borel sigma field;
        \item the \textit{transition kernel} is a map $p: \mathcal{X} \times A \to \Pcal(\mathcal{X})$;
        \item the \textit{reward function} $r: \mathcal{X} \times A \to [0, \infty)$ is bounded and measurable.
    \end{itemize}
\end{defn}

Since we will often consider sequences of states and actions taken, we will use a shorthand notation with the bracket, e.g. $(x)^n_1 \coloneqq (x_1, \ldots, x_n)$. The $n$-step transition kernels $(p^{(n)})_{n \in \N}$ can then be recursively defined by the relation
\begin{align*}
    p^{(n)}(\de x_n \mid  x_0,(a)_0^{n-1}) =p(\de x_n\mid x_{n-1}, a_{n-1})\ p^{(n-1)}(\{x_{n-1}\} \mid x_0,(a)_0^{n-2}).
\end{align*}

In the context of mean-field games, the transition kernel and reward functions will also depend on an input measure, so will be considered as functions $p: \mathcal{X} \times A \times \Pcal(\mathcal{X}) \to \Pcal(\mathcal{X})$ and  $r: \mathcal{X} \times A \times \Pcal(\mathcal{X})\to [0, \infty)$ respectively. We denote by $(\Pcal(E))^{T}$ the space of measure flows on $E$, with $T \in \N \cup \{\infty\}$. If $T$ is finite, we equip $(\Pcal(E))^{T}$ with the sup metric
\begin{align*}
    \delta_{\mathrm{max}}(\bm{\mu}, \bm{\hat{\mu}}) = \max_{0 \leq t \leq T} \delta_{TV}(\mu_t ,\hat{\mu}_t),\quad \bm{\mu}, \bm{\hat{\mu}} \in (\Pcal(E))^{T}.
\end{align*}
For $T = \infty$, we instead use the metric, for $\zeta >0$,
\begin{align}\label{eq_flow_metric}
    \delta_{\infty}(\bm{\mu}, \bm{\hat{\mu}}) =
    \sum^{\infty}_{t=0} \zeta^{-t} \delta_{TV}(\mu_t ,\hat{\mu}_t),\quad \bm{\mu}, \bm{\hat{\mu}} \in (\Pcal(E))^{T}.
\end{align}
While that the choice of $\zeta$ is not canonical, note that $\delta_{\infty}$ induces the product topology on $\Pcal(E)^{\infty}$ as long as $\zeta > 1$. We will consider randomised Markovian policies in this paper, such a policy $\pi = (\pi_t)_t$ is described by a sequence of maps $\pi_t : E \to \Pcal(\pr{E})$, $t \geq 0$, which we equip with the norm
\begin{align}\label{metric:policy}
    \delta_{\mathcal{A}}(\pi, \hat{\pi}) = \sum^{\infty}_{t=0} \zeta^{-t} \sup_{e \in E} \delta_{TV}(\pi_t(\cdot \mid e), \hat{\pi}_t(\cdot \mid e)).
\end{align}

We will frequently make use of the following basic inequality in our analysis. For $P ,Q \in \Pcal(E)$, we write $P \ll Q$ to say that $P$ is absolutely continuous with respect to $Q$.
\begin{prop}\label{prop: maxmin_ineq}
Let $\mu, \hat{\mu} \in \Pcal(E)$, and let $m \in \Pcal(E)$ such that $\mu, \hat{\mu} \ll m$. Then
    for any measurable function $f$ on $E$, we have the inequality
    \begin{align}\label{eq:minmax_ineq}
        \left\lvert \int_{e \in E} f(e) \mu (\de e) - \int_{e \in E} f(e) \hat{\mu}(\de e) \right \rvert \leq \lambda(f)\ \delta_{TV}(\mu, \mu^{\prime}),
    \end{align}
    for $\lambda(f) \coloneqq {\esssup{e\in E}}f(e) - {\essinf{e \in E}}f(e)$, and the essential supremum and essential infimum are taken with respect to the measure $m$. Moreover, if $E$ is compact and $f$ is continuous, then \eqref{eq:minmax_ineq} holds for $\lambda(f) = \sup_{e \in E} f(e) - \inf_{e \in E} f(e) = \max_{e\in E}f(e) - \min_{e \in E}f(e)$.
\end{prop}
\begin{proof}
    See \cite[p.141]{georgii2011gibbs}.
\end{proof}

Next, let us recall basic notions of differentiability in Banach spaces (see e.g.\  \cite[Chapter VIII]{dieudonne2011foundations}).

\begin{defn}
     Let $E$ be a Banach space and $D$ be an open subset of $E$. For $x \in E$ and $\varepsilon > 0$, let $B_{\varepsilon}(x)$ denote the open ball of radius $\varepsilon$ around $x$. A function $f: D \to \R$ is Fr\'echet differentiable at a point $x \in D$, if there exists $\varepsilon > 0$, a continuous linear map $L: E \to \R$ and a function $\theta: B_{\varepsilon}(0) \to \R$ such that
     \begin{align*}
         f(x +h) - f(x) = L(x) + \theta(h)
     \end{align*}
    for any $x + h \in B_{\varepsilon}(x) \cap D$ and $\lim_{h\to 0} \vert \theta(h)\vert/ \lVert h \rVert_E = 0$. $f$ is differentiable on $D$ if it is differentiable at every point in $D$.
\end{defn}

If $f$ is Fr\'echet differentiable at $x$, the Fr\'echet derivative $L(x)$ of $f$ at $x$ is unique and denoted by $df_x$. For later use, we recall a mean value theorem for functions with domains in Banach spaces. 

\begin{thm}\label{mean_value_theorem}
    Let $E$ and $F$ be Banach spaces and $D$ be an open subset of $E$. Let $x, y \in D$ and suppose that the closed line segment joining $x$ and $y$ is contained in $D$. Let $f: D \to F$ be continuous on the closed line segment joining $x$ and $y$, and differentiable on the open segment joining $x$ and $y$. Then, there exists $t^{*} \in (0,1)$ such that for $ z:= (1-t^{*})x  + t^{*}y$,
    \begin{align*}
        \lVert f(y) - f(x) \rVert \leq \lVert df_z (y-x) \rVert \leq \lVert df_z \rVert \lVert (y-x) \rVert.
    \end{align*}
    Furthermore, in the case $F = \R$, there exists $t^*$, defining $z$, such that we have the equality 
    \begin{align*}
     f(y) - f(x) =  df_z (y-x).
    \end{align*}
\end{thm}

Finally, we recall here the notion on distances between sets.
\begin{defn}\label{hausdorff}
    Let $(E,d)$ be a metric space. For any set $X \subseteq E$, define 
    \begin{align*}
        X_{\varepsilon} \coloneqq \bigcup_{x \in X} \{ e \in E: d(e,x) \leq \varepsilon\}.
    \end{align*}
    For any sets $X, Y \subseteq E$, the Hausdorff distance between $X$ and $Y$ is defined by
    \begin{align*}
        d_H(X,Y) \coloneqq \inf\{\varepsilon \geq 0 : X \subseteq Y_{\varepsilon}\ \rm{and}\ Y \subseteq X_{\varepsilon}\}.
    \end{align*}
\end{defn}

\subsection{Brief outline of the framework}
\label{subsec:outline}

We present in this section an outline of the framework and main results of the paper. We emphasise that the focus here is on the intuition and general ideas, and for this purpose we frequently refer here to specific places in later sections for the precise definitions.\\

We refer to our model in consideration as a Markov controllable delay model (MCDM), which consists of an MDP $\langle \mathcal{X}, A, p, r \rangle$, together with a set of \textit{delay values} $\mathcal{D} = \{ d_0, \ldots, d_K\}$, with $ 0 \leq d_K < \ldots < d_0$, and a set of \textit{cost values} $\mathcal{C} = \{ c_0, c_1, \ldots, c_K\}$, with $0 = c_0 < c_1 \ldots < c_K$. In an MCDM, a player can control both their actions (with values in $A$), as well as their observation delay (with values in $\mathcal{D}$), by paying a cost (with values in $\mathcal{C}$). Each $d_k \in \mathcal{D}$ is associated with the corresponding $c_k \in \mathcal{C}$, so that a short delay comes at a high cost, and vice versa. The observation delay determines the set of admissible policies for the player --  given an observation delay of $d$, an admissible policy at time $t$ can only depend on the state variables up to time $t-d$, and the actions taken up to time $t-1$ (cf. \Cref{defn:mcdm_admissible}). \Cref{fig:MCDM_evol} illustrates an example of the evolution of the MCDM over time. This restricted set of admissible policies is denoted by $\mathcal{A}_{\rm DM}$. The cost is incorporated as a negative reward, which leads to the objective
\begin{align*}
   J(\pi) \coloneqq \E^{\pi}_{q_0}\left[ \sum^{\infty}_{t=0} \gamma^{t} \left( r(x_t, a_t) -  \sum^K_{k=1} c_k \mathbf{1}_{\{i_t=k\}} \right)\right],
\end{align*}
where $a_t$ represents the choice of action, $i_t$ represents the choice of delay, and $q_0$ is a distribution over the initial observations, see \eqref{eq:objective_single_agent}. The precise construction of the MCDM problem is detailed in \Cref{sec:single_agent}. When considering a mean-field game setting for the MCDM, the transition kernel $p$ and reward $r$ depend in addition on the distribution of the population (cf. \Cref{defn:mcdm_mf_ver}). Given a measure flow $\bm{\mu}= (\mu_t)_t \in \Pcal(\mathcal{X})^{\infty} $ representing the distribution of the population, we consider instead the objective 
\begin{align*}
    J_{\bm{\mu}}(\pi) \coloneqq \E^{\pi}\left[ \sum^{\infty}_{t=0} \gamma^{t} \left( r(x_t, a_t, 
    \mu_t) -  \sum^K_{k=1} c_k \mathbf{1}_{\{i_t=k\}} \right)\right].
\end{align*}
The notion for a solution for a mean-field game is a mean-field Nash equilibrium (MFNE). This is characterised as a fixed point of the composition of a best-response map and a measure flow map. The best-response map outputs an optimal policy $\pi \in \mathcal{A}_{\rm DM}$ that maximises $J_{\bm{\mu}}(\pi)$ for a fixed measure flow $\bm{\mu} \in \Pcal(\mathcal{X})^{\infty}$. Conversely, the measure flow map outputs a measure flow $\bm{\mu} \in \Pcal(\mathcal{X})^{\infty}$ corresponding to a given policy $\pi \in \mathcal{A}_{\rm DM}$. The fixed point condition acts as a consistency condition, requiring that the distribution of players under an optimal policy computed from a measure flow at Nash equilibrium, to be exactly the measure flow itself.\\

As the MCDM is a model of partial information, it is more convenient to formulate the MFNE in an augmented problem, in the space $\mathcal{Y} \coloneqq \bigcup_{d=d_K}^{d_0} (\{d\} \times \mathcal{X} \times A^d)$. The equivalence of the two formulations for the single agent problem is established in \Cref{prop:equivalence_aug}. However, in the augmented formulation of the mean-field game, the mean-field flow will be defined in the augmented space $\mathcal{Y}$, whilst the transition kernel and reward still depend on measures on the underlying space $\mathcal{X}$. We therefore construct a map that takes
a measure flow $\bm{\nu}$ on the augmented space, to a corresponding $\bm{\mu}^{\nu}$ on the underlying space, through exploiting the finite parameterisation of the belief state and integrating $\bm{\nu}$ repeatedly against the transition kernel $p$ (cf. \Cref{defn:aug_to_underlying}). This allows us to define the best response map $\Phi^{\rm aug} : \Pcal(\mathcal{Y})^{\infty} \to \mathcal{A}_{\rm DM} $ and the measure flow map $\Psi^{\rm aug}: \mathcal{A}_{\rm DM} \to \Pcal(\mathcal{Y})^{\infty}$ as
\begin{align*}
\Phi^{\mathrm{aug}}(\bm{\nu}) &= \left\{ \hat{\pi} \in \mathcal{A}_{\rm DM}:  J_{\bm{\mu}^{\nu}}( \hat{\pi}) = \sup_{\pi \in \mathcal{A}_{\rm DM}} J_{\bm{\mu}^{\nu}}(\pi) \right\},\\
\Psi^{\mathrm{aug}}(\pi)_{t+1}(\cdot) &=\int_{\mathcal{Y}\times U} p_y\Big(\cdot \mid y, u, \bm{\mu}_t^{\Psi^{\mathrm{aug}}(\pi)} \Big) \pi_t(\de u \mid y) \Psi^{\mathrm{aug}}(\pi)_t(\de y),
\end{align*}
where $U \coloneqq A \times \mathcal{I}$ and $p_y$ denote the transition kernel in the augmented space (cf. \eqref{aug_transition}), so that for the MCDM, a MFNE is defined to be a fixed point of the map $\Psi^{\rm aug} \circ \Phi^{\rm aug}$ (cf. \Cref{defn:mfne_mcdm}).\\

Our result concerns a regularised MFG of the MCDM, where we show that the fixed point iteration for the regularised game is contractive, therefore establishing the uniqueness of a regularised MFNE. This requires the addition of a Kullback–Leibler (KL) divergence term in the value function. For a reference measure $q \in \Pcal(U)$, measure flow $\bm{\nu} \in \Pcal(\mathcal{Y})^{\infty}$ and regularisation parameter $\eta >0$ we consider the regularised objective
\begin{align*}
    J_{\eta, \bm{\nu}}(\pi) =  \E^{\pi}\left[\sum^{\infty}_{t=0} \gamma^{t}\left( r_y \big(y_t, u_t, \bm{\mu}^{\nu}_t\big) - \eta D_{KL}(\pi_t \Vert q) \right) \right],
\end{align*}
where $r_y$ represents the reward formulated in the augmented space (cf.\ \eqref{aug_reward}), and in particular includes the observation costs. In the regularised game, for a fixed measure flow $\bm{\nu} \in \Pcal(\mathcal{Y})^{\infty}$, the optimal policy is given by the softmax policy
 \begin{align*}
        \pi^{\rm soft}_t(\de u \mid y) \coloneqq \frac{\exp\left(  Q^{*}_{\eta,\bm{\nu}}(t,y,u)/\eta\right)}{\int_U \exp\left( Q^{*}_{\eta,\bm{\nu}}(t,y,\pr{u})/\eta\right) q(\de \pr{u})} q(\de u),
    \end{align*}
where $ Q^{*}_{\eta,\bm{\nu}}$ is the associated $Q$-function for $J_{\eta, \bm{\nu}}$ (see \eqref{regularised_Q_func}). Thus, by defining the regularised best-response map to output the softmax policy associated to the measure flow $\bm{\nu}$ (cf. \Cref{defn:regularised MFNE-MCDM}), we have the following result. 
\begin{thm}
    The fixed point operator $\Psi^{\mathrm{aug}} \circ \Phi^{\mathrm{reg}}$ is a contraction mapping on the Banach space $(\Delta^{\infty}_{\tilde{\mathcal{Y}}}, \delta_{\infty})$, provided that $\zeta> C_{p,r}$ in the metric \eqref{eq_flow_metric}, and $\eta > K_{\zeta, p, r}$. Here, the constant $C_{p,r}$ depends only on the bounds and Lipschitz constants of the transition kernel $p$ and reward function $r$, and $K_{\zeta, p, r}$ depends in addition on the chosen value of $\zeta$. Hence, by Banach's fixed point theorem, there exists a unique fixed point for $\Psi^{\mathrm{aug}} \circ \Phi^{\mathrm{reg}}$, which is a regularised MFNE.
\end{thm}

In particular, the threshold for $\eta$ is lower when the discount factor $\gamma$ is high or the transition kernel's dependence on the population distribution is low, i.e., for small values of $L_p$. The precise statement of the theorem above with more explicit thresholds for $\zeta$ and $\eta$ is given in \Cref{thm:contraction_reg_mcdm}. Finally, we show that a regularised MFNE does indeed give an approximate Nash equilibrium for the finite player game. In the finite player game, each player's policy depends also on the empirical distribution of the other players. In a game with $N$ players, the objective for player $n$ is given by
\begin{align*}
    J^N_n(\pi) \coloneqq \E^{\pi}_{q_0}\left[ \sum^{\infty}_{t=0} \gamma^{t} \big(r(x^n_t, a^n_t, e^N_t) - \sum^K_{k=1} c_k \mathbf{1}_{\{i^n_t = k\}} \big) \right],\ \pi = (\pi^1, \ldots, \pi^N) \in \mathcal{A},
\end{align*}
where $\mathcal{A} = \prod^N_{n=1} \mathcal{A}^n$ denotes the space of admissible policies for the $N$ players (cf. \Cref{defn:policy_nplayer}). A regularised MFNE with sufficiently small regularisation is then an approximate Nash equilibrium for a finite player game with sufficiently larger number of players. We state this more precisely in the theorem below.

\begin{thm}
    Let $(\eta_j)_j$ be a sequence with $\eta_j \downarrow 0$. For each $j$, let $(\pi^{(j)}_*,  \bm{\nu}^{(j)}_*) \in \mathcal{A}_{\rm DM} \times \Pcal(\mathcal{Y})^{\infty}$ be the associated regularised MFNE, defined via \Cref{defn:regularised MFNE-MCDM}, for the MCDM-MFG. Then for any $\varepsilon > 0$, there exists $\pr{j}, \pr{N} \in \N$ such that for all $j\geq\pr{j}$ and $N \geq \pr{N}$, the policy $(\pi^{(j)}_*, \ldots, \pi^{(j)}_*)$ is $\varepsilon$-Nash for the $N$-player game. That is,
    \begin{align*}
        J^N_n(\pi^{(j)}_*, \ldots, \pi^{(j)}_*) \geq \sup_{\pi \in \mathcal{A}_{\rm DM}} J^N_n(\pi, \pi^{(j), -n}_*) - \varepsilon \quad \mbox{for all $n \in \{1, \ldots, N\}$},
    \end{align*}
    where $\pi^{(j), -n}_*$ represents the policy $\pi^{(j)}_*$ being applied to all players except player $n$.
\end{thm}

\section{Control of information speed: single agent case}\label{sec:single_agent}

We first state the definition of a Markov controllable delay model (MCDM) below, which characterises the scenarios where agents can control their information delay.

\begin{defn}\label{defn:mcdm}
    A Markov controllable delay model (MCDM) is a tuple $\langle \mathcal{X}, A, \mathcal{D}, \mathcal{C}, p, r \rangle$, where
    \begin{itemize}
        \item the \textit{state space} $\mathcal{X}$ and the \textit{action space} $A$ are non-emoty Polish spaces, equipped with their respective Borel sigma fields;
        \item the set of \textit{delay values} is $\mathcal{D} = \{ d_0, \ldots, d_K\} \subset \mathbb Z_+$, with $ 0 \leq d_K < \ldots < d_0$;
        \item the set of \textit{cost values} is $\mathcal{C} = \{ c_0, c_1, \ldots, c_K\} \subset \R_+$, with $0 = c_0 < c_1 \ldots < c_K $;
        \item the \textit{transition kernel} is a map $p: \mathcal{X} \times A \to \Pcal(\mathcal{X})$;
        \item the \textit{reward function} $r: \mathcal{X} \times A \to [0, \infty)$ is bounded and measurable.
    \end{itemize}
\end{defn}

We will make the following additional assumption for the rest of the paper. 

\begin{ass}\label{ass_compacct}
    The state space $\mathcal{X}$ and the action space $A$ are compact, the reward function $r$ is continuous, and the transition kernel $p$ is weakly continuous, i.e. if $(x_n, a_n) \to (x,a) \in \mathcal{X} \times A$, then ${p(\cdot\mid x_n, a_n) \to p(\cdot \mid x,a)}$ weakly.
\end{ass}

\Cref{defn:mcdm} mimics that of a standard MDP, but with the addition of delay and cost values. In the following, we assume that the problem initiates at time $t=0$, and denote prior observations with negative indices.   

\begin{defn}\label{defn:history sets}
    Denote by $\mathcal{I} \coloneqq \{0,1,\ldots, K\}$ the \textit{intervention set}. Let $H_{0} \coloneqq (\mathcal{X} \times A)^{d_0} \times \mathcal{X}$, and define recursively $H_t \coloneqq H_{t-1} \times A \times \mathcal{I} \times \mathcal{X}$, $t\geq 1$. A \textit{policy} $\pi = (\pi_t)_{t \geq 0}$ is a sequence of (Markov) kernels $\pi_t: H_t \to \Pcal(A \times \mathcal{I})$.
    \end{defn}

The intervention set $\mathcal{I}$ represents an agent's control over their delay in an MCDM. We will restrict our admissible policies to those that depend exclusively on the delayed observations, and the history of actions and interventions taken.

\begin{defn}\label{defn:mcdm_admissible}
Let $h_t = (x_{-d_0}, a_{-d_0}, \ldots, x_0, a_0, i_0, \ldots, a_{t-1}, i_{t-1}, x_t) \in H_t$. Define the sequence $(\tilde{\iota}_t)_t$ recursively by $\tilde{\iota}_{t+1} = \min\{\tilde{\iota}_{t} + 1, d_{i_{t}}\}$, $\tilde{\iota}_0 = d_0$. A policy $\pi$ is \textit{admissible for an MCDM} if at each time $t$, there exists kernels $\phi_t: \mathcal{X}^{t +1}  \times A^{t+d_0} \times \mathcal{I}^t \to \Pcal(A \times \mathcal{I})$, such that for each $h_t \in H_t$,
    \begin{align*}
       \pi_t(\cdot \mid h_t) = \phi_t(\cdot \mid x_{0-d_0}, x_{1 - \tilde \iota _1},\ldots, x_{t-{\tilde{\iota}_t}}, a_{-d_0}, \ldots, a_{t-1}, \tilde\iota_0,\ldots, \tilde\iota_{t-1}).
    \end{align*}
   The set of admissible policies for the MCDM is denoted by $\mathcal{A}_{\rm DM}$.
\end{defn}
	With an admissible policy $\pi \in \mathcal{A}_{\rm DM}$ and an initial distribution $q_0 \in \Pcal(H_0)$, we can construct a probability measure $\PP^{\pi}_{q_0}$ on $H_{\infty} \coloneqq H_0 \times (A \times \mathcal{I} \times \mathcal{X})^{\infty}$ with the Ionescu--Tulcea theorem \cite[Appendix C]{hernandez1989adaptive}. This is the unique probability measure such that  for $\omega = (x_{-d_0}, a_{-d_0}, \ldots, x_{0}, a_{0}, i_{0}, \ldots) \in H_{\infty}$,
\begin{align*}
    \PP^{\pi}_{q_0}(\de \omega)  & = q_0(\de h_0)\ p(\de x_{-d_0+1} \mid x_{-d_0}, a_{-d_0})\cdots p(\de x_{0} \mid x_{-1}, a_{-1})\\
    & \qquad \pi_{0}(\de a_{0}, \de i_{0} \mid h_0)\ p(\de x_{1} \mid x_{0}, a_{0})\ \pi(\de a_1, \de i_1 \mid h_1)\cdots
\end{align*}
The objective function for the infinite horizon problem with discounted cost is then
\begin{align}\label{eq:objective_single_agent}
    J(\pi) \coloneqq \E^{\pi}_{q_0}\left[ \sum^{\infty}_{t=0} \gamma^{t} \left( r(x_t, a_t) -  \sum^K_{k=1} c_k \mathbf{1}_{\{i_t=k\}} \right)\right],
\end{align}
where $\E^{\pi}_{q_0}$ is the expectation over the measure $\PP^{\pi}_{q_0}$, and $\gamma \in (0,1)$ is the discount factor. If the problem starts at some time $t>0$ instead, we write $\E^{\pi}_{t,q_0}$ for the expectation. If $q_0 = \delta_{h_0}$, then we write $\E^{\pi}_{h_0}$ (resp. $\E^{\pi}_{t,h_0}$) in place of $\E^{\pi}_{\delta_{h_0}}$ (resp. $\E^{\pi}_{t,\delta_{h_0}}$).

The above setup provides the following interpretation of the sequential evolution of an MCDM. At time $t$, the controller observes the underlying state $x_{t-d} \in \mathcal{X}$ for some $d \in [d_K, d_0]$, with knowledge of their applied actions $a_{t-d}, \ldots, a_{t-1} \in A$. Based on this information, the controller applies an action $a_t$ and receives a reward $r(x_t, a_t)$, which we assume not to be observable until $x_t$ becomes observable. The controller then chooses a value for $i_t$, which determines their cost $c_{i_t}$, as well as their next delay period of $d_{i_t}$ units. This process then repeats at the next time. To ensure that the setup is sensible in its interpretation, if $i_t = k$ and the current delay $d$ is shorter than $d_k$, then the delay at time $t+1$ will simply be extended to $d+1$ units (in reality, paying a higher cost for a longer delay is clearly sub-optimal, so such a choice of $i_t$ would not practically occur). \Cref{fig:MCDM_evol} depicts a typical evolution of an MCDM.

\begin{figure}[t!]
\centering
    \includegraphics[scale=0.225]{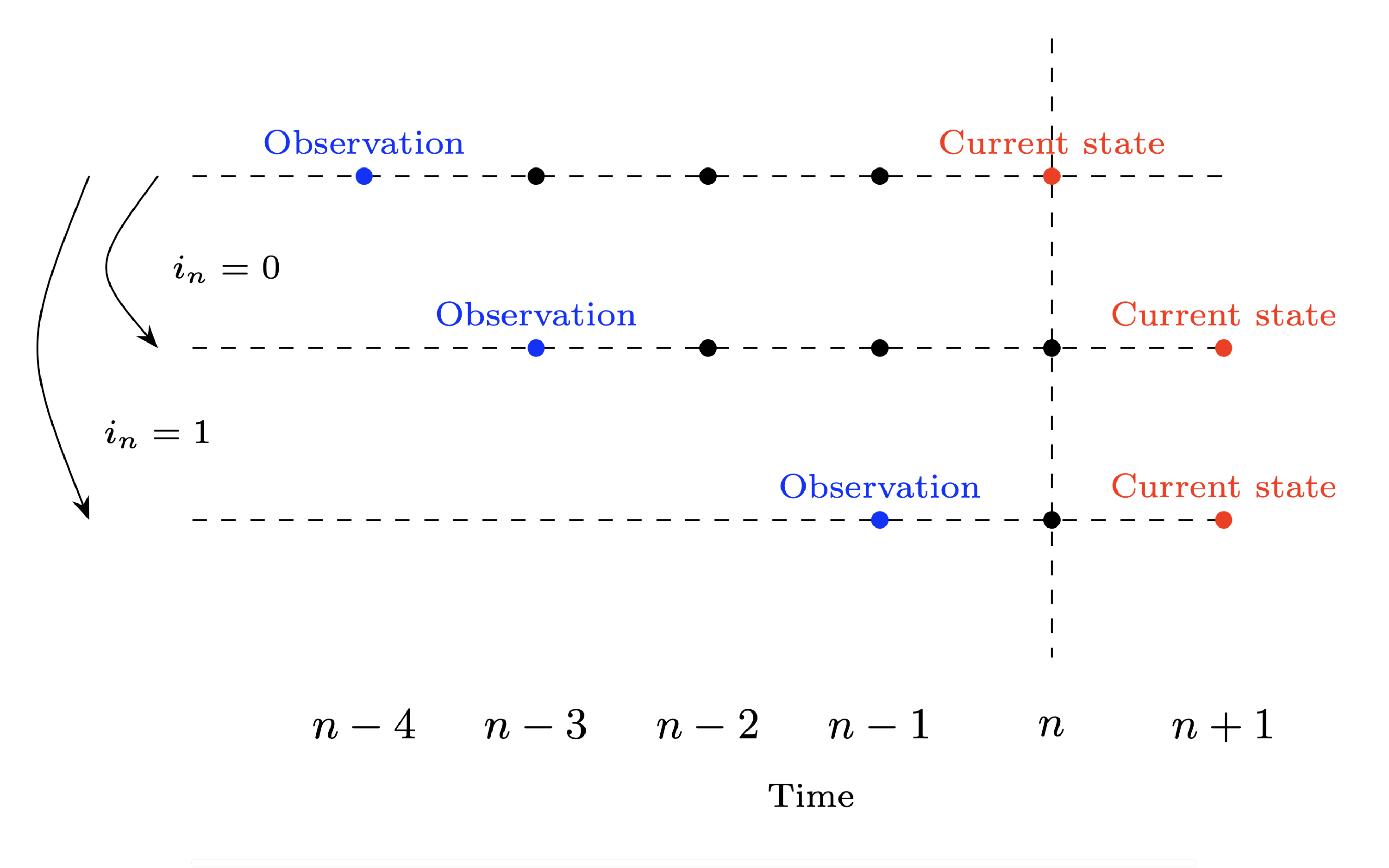}
    \caption{An example of an MCDM with possible observation delays being $d_0 = 4$ and $d_1 =2$. Top: the unobservable current state (red) at time $n$, and the delayed observable state (blue) that occured at time $n-4$. Middle and bottom: depictions of the observation delay at time $n+1$ for two alternative choices of $i_n$. \vspace{1em}
    \label{fig:MCDM_evol}}
    {}
\end{figure}

Since the control problem arising from an MCDM is non-Markovian in its current form, we shall establish an equivalence with a fully observable MDP, so that we can apply dynamic programming. As noted in \cite{altman_nain}, this is akin to solving a POMDP by constructing an equivalent MDP on the belief state. Yet, here the information basis for our decision problem is not described by observations being noisy functions of the underlying state being obtained concurrently. Instead, it is given by the last delayed state we had actively chosen to observe in the past, at some cost,  
complemented by our actions taken since then. 

\begin{defn}\label{defn:augmented_space}
    For $d_K \leq d \leq d_0$, let $\mathcal{Y}_{d} \coloneqq \mathcal{X} \times A^{d}$, equipped with the product topology. Define the \textbf{augmented space} $\mathcal{Y}$ by
    \begin{align*}
        \mathcal{Y} \coloneqq \bigcup_{d=d_K}^{d_0} \left( \{d\} \times \mathcal{X} \times  A^d \right) = \bigsqcup^{d_0}_{d=d_K}\mathcal{Y}_d,
    \end{align*}
    where $\bigsqcup$ denotes the disjoint union of sets. The augmented space $\mathcal{Y}$ is equipped with the disjoint union topology, which is metrizable by
    \begin{align*}
        \delta_Y(y, \pr{y}) = \begin{cases}
            \delta_d(y, \pr{y}) \wedge 1 , \quad &y, \pr{y} \in  \mathcal{Y}_{d},\ d_K \leq d \leq d_0,\\
            1 , & y\in \mathcal{Y}_d,\ \pr{y} \in \mathcal{Y}_{\pr{d}},\ d \neq \pr{d},
        \end{cases}
    \end{align*}
    where $\delta_d$ is any metric that induces the product topology on $\mathcal{Y}_d$.
\end{defn}
Thus the augmented variable includes the delay of the system at the current time, the underlying state that is observed with that delay, and the actions applied from that moment up to the present. We shall write an element $y \in \mathcal{Y}$ in the form
$(d, x, a_{-d}, \ldots, a_{-1}),\ \mbox{or } (d, x, (a)^{-1}_{-d})$, where negative indices are used to indicate that the actions had occurred in the past. If specific indices are not required, we will also use the notation $y=( d,x, \bm{a})$. The augmented MDP is given by the tuple $\langle \mathcal{Y}, A \times \mathcal{I}, p_y, r_y \rangle$, where the kernel ${p_y: \mathcal{Y} \times (A \times \mathcal{I}) \to \Pcal(\mathcal{Y}})$ is defined by
\begin{align}\label{eq:augmented_kernel}
    p_y(\de \hat{y} \mid y, \tilde{a}, i ) & \coloneqq \begin{cases} p^{(d-d_i+1)}\big(\de \hat{x} \mid x, (a)_{-d}^{-d_i}\big)\otimes \delta_{\{ (a_{-d_i+1},\ldots, a_{-1}, \tilde{a}) \}}(\de  \hat{\bm{a}}) \otimes \delta_{d_i}(\hat{d}) , & \ d_i \leq d \leq d_0;\\
     \delta_{x}(\de \hat{x}) \otimes 
    \delta_{(a_{-d}, \ldots, a_{-1}, \tilde{a})}(\de \hat{\bm{a}}) \otimes \delta_{d+1}(\hat{d}) ,& \ d_K \leq d< d_i, \end{cases}
\end{align}
where $\de \hat{y} =  \sqcup_{\hat{d}}\ \de\hat{x} \de \hat{\bm{a}}$, $y = (d, x, \bm{a}) \in \mathcal{Y}$, $\tilde{a} \in A$, $i \in \mathcal{I}$, and $p^{(d-d_i+1)}$ is the $(d-d_i+1)$--step transition kernel. Similarly, the augmented reward $r_y: \mathcal{Y} \times (A \times \mathcal{I}) \to [0,\infty)$ is defined by
\begin{align*}
    r_y(y, \tilde{a}, i) \coloneqq \int_{ \mathcal{X}} r(\pr{x}, \tilde{a}) p^{(d)}(\de \pr{x} \mid x, \bm{a})  - \sum^K_{k=1} c_k \mathbf{1}_{\{i=k\}}.
\end{align*}

\begin{defn}\label{def:aug_pol}
    Let $\pr{H}_0 \coloneqq \mathcal{Y}$ and $\pr{H}_t \coloneqq \pr{H}_{t-1} \times A \times \mathcal{I} \times \mathcal{Y}$ for $t \geq 1$. An element $\pr{h}_t = (y_0, a_0, i_0, \ldots, y_t) \in \pr{H}_t$ is \textit{consistent} if for any $y_s = (d, x_{s-d}, (a)^{s-1}_{s-d})$ and $y_t = (\pr{d}, x_{t-\pr{d}}, (a)^{t-1}_{t-\pr{d}})$, it holds that $a_{s-k} = a_{t-l}$ if $s-k=t-l$, and $x_{s-d} = x_{t-\pr{d}}$ if $s-d = t- \pr{d}$. We denote the set of consistent elements of $\pr{H}_t$ by $\mathcal{H}_t$. A policy $\pr{\pi} = (\pr{\pi}_t)$ for the augmented MDP is then a sequence of kernels $\pr{\pi}_t: \mathcal{H}_t \to \Pcal(A \times \mathcal{I})$, with the set of policies for the augmented MDP denoted by $\mathcal{A}_{\mathrm{aug}}$.
\end{defn}

Given a policy $\pr{\pi} \in \mathcal{A}_{\mathrm{aug}}$ and initial distribution $\tilde{q}_0 \in \Pcal(\mathcal{Y})$, we can obtain a unique probability measure $\PP^{\pr{\pi}}_{\tilde{q}_0}$ via the Ionescu--Tulcea theorem again. Note that by construction all the history sequences induced by $p_y$ and $\pr{\pi}_t$ are consistent (cf. \Cref{def:aug_pol}), so that the support of $\PP^{\pr{\pi}}_{\tilde{q}_0}$ is contained within the consistent elements of $\pr{H}_{\infty} = \pr{H}_0 \times (A \times \mathcal{I} \times \mathcal{Y} )^{\infty}$. We then define the objective for the augmented MDP by 
\begin{align*}
     J_{\mathrm{aug}}(\pr{\pi}) \coloneqq \E^{\pr{\pi}}_{\tilde{q}_0}\left[ \sum^{\infty}_{t=0} \gamma^{t} r_y(y_t, a_t, i_t) \right],
\end{align*}
where $\E^{\pr{\pi}}_{\tilde{q}_0}$ is the expectation over $\PP^{\pr{\pi}}_{\tilde{q}_0}$.

\begin{prop}\label{prop:equivalence_aug}
    Given an initial distribution $\tilde{q}_0 \in \Pcal(\mathcal{Y})$ on the augmented MDP, let $q_0 \in \Pcal(\mathcal{X})$ be an initial distribution for the MCDM given by
    \begin{align*}
        q_0(\cdot) \coloneqq \int_{\mathcal{Y}} p^{(d)}(\cdot \mid x, \bm{a}) \tilde{q}_0(\de y),\quad y = (d,x ,\bm{a}).
    \end{align*}
    Then solving the augmented MDP is equivalent to solving the MCDM problem, and it holds that
\begin{align*}
    \sup_{\pi \in \mathcal{A}_{\rm DM}} J(\pi) = \sup_{\pr{\pi} \in \mathcal{A}_{\mathrm{aug}}} J_{\mathrm{aug}}(\pr{\pi}).
\end{align*}
\end{prop}

\begin{proof}
    We follow a similar argument to that in \cite[p.89]{hernandez1989adaptive}. Each $h_t \in H_t$ can be mapped to a corresponding $\pr{h}_t \in \mathcal{H}_t$ via
\begin{align*}
    H_t \ni (x_{-d_0}, a_{-d_0}, \ldots, x_{0}, a_{0}, i_{0}, \ldots, x_t) \rightarrow (y_0, a_0, i_0, \ldots, y_{t-1}, a_{t-1}, i_{t-1}, y_t) \in \mathcal{H}_t,
\end{align*}
where $y_t = (\tilde{\iota}_t, x_{t-\tilde{\iota}_t}, (a)_{t-\tilde{\iota}_t}^{t-1})$ for $t \geq 0$ ($\tilde{\iota}_t$ is as defined in \Cref{defn:mcdm_admissible}). Then, given a policy $\pi \in \mathcal{A}_{\rm DM}$, one can define a policy $\pr{\pi}\in\mathcal{A}_{\mathrm{aug}}$ via
\begin{align}\label{eq:equiv_pol}
     \pr{\pi}_{t}(\cdot \mid \pr{h}_t) \coloneqq \pi_t(\cdot \mid h_t), \quad t \geq 0.
\end{align}
Conversely, each $\pr{h}_t \in \mathcal{H}_t$ corresponds to a sequence $(x_{0-d_0}, \ldots, x_{t-{\tilde{\iota}_t}}, a_{-d_0}, \ldots, a_{t-1}, \tilde\iota_0,\ldots, \tilde\iota_{t-1})$. Therefore, given a policy $\pr{\pi}\in\mathcal{A}_{\mathrm{aug}}$, one can define an admissible policy $\pi \in \mathcal{A}_{\rm DM}$ by constructing the corresponding kernels $\phi_t$ via
 \begin{align*}
        \phi_t(\cdot \mid x_{0-d_0}, \ldots, x_{t-{\tilde{\iota}_t}}, a_{-d_0}, \ldots, a_{t-1}, \tilde\iota_0,\ldots, \tilde\iota_{t-1}) \coloneqq \pr{\pi}_t(\cdot \mid \pr{h}_t).
\end{align*}
Moreover, the policies $\pi$ and $\pr{\pi}$ assign the same conditional probability on $A \times \mathcal{I}$, given any $h_t \in H_t$ alongside the corresponding $\pr{h}_t \in \mathcal{H}_t$. Therefore, by the definition of the objective functions and the relation \eqref{eq:equiv_pol}, we have $J(\pi) = J_{\rm aug}(\pr{\pi})$, and by the equivalence of the set of policies, we have 
\begin{align*}
    \sup_{\pi \in \mathcal{A}_{\rm DM}} J(\pi) = \sup_{\pr{\pi} \in \mathcal{A}_{\mathrm{aug}}} J_{\mathrm{aug}}(\pr{\pi})
\end{align*}
as required.
\end{proof}

Given the equivalence of policies, we shall write $\mathcal{A}_{\rm DM}$ to represent the set of admissible policies without loss of generality. Since the augmented MDP is fully observable and Markovian, we can establish dynamic programming to solve for the MCDM.

\begin{prop}
Let $v : \mathcal{Y} \to \R$ be the value function
\begin{align*}
    J^{*}(y) \coloneqq \sup_{\pi \in \mathcal{A}_{\rm DM}} \E^{\pi}_{y}\left[ \sum^{\infty}_{t=0} \gamma^{t} r_y(y_t, a_t, i_t) \right],
\end{align*}
then $J^{*}$ satisfies the dynamic programming equation
    \begin{align*}
        J^{*}(y) = \max_{(\tilde{a},i) \in A \times \mathcal{I}} \left\{ \int_{ \mathcal{X}} r(\pr{x}, \tilde{a}) p^{(d)}(\de \pr{x} \mid x, \bm{a}) - \sum^{K}_{k=1} c_k \mathbf{1}_{\{i = k\}} + \gamma \int_{ \mathcal{Y}}J^{*}(\pr{y}) p_y( \de \pr{y}\mid y, \tilde{a}, i)  \right\},
    \end{align*}
    where $p_y$ is the augmented kernel as in \eqref{eq:augmented_kernel}. Moreover, the optimal policy is of feedback form (with respect to the augmented space).
\end{prop}

\begin{proof}
    As the equivalent augmented formulation is now a fully observable problem, the dynamic programming equation follows from a standard argument for the case of a fully observable MDP, see e.g. \cite[Theorem 4.2.3]{hernandezlerma_1}.
\end{proof}

\begin{rem}\label{rem_homo}
    When considering the MFG in the next section, a deterministic measure flow representing the population distribution introduces an implicit time dependence within the transition kernel and reward. The generic single agent problem in the definition of the MFNE then becomes time-inhomogeneous. The time-homogeneous setup in this section readily generalises directly to a setup with time-inhomogeneous transition kernels, rewards and dynamic programming equations. However, for ease of exposition we choose to present the MCDM under the time-homogeneous setting here.
\end{rem}

\section{Mean-field games with control of information speed}\label{sec:MFG}

{In the remainder of the paper, we let $U$ be the product space $A \times \mathcal{I}$ and we write $u = (a,i) \in U$.} We adopt the following definition of the MCDM, to take into account of the interacting affects of agents in the game.
\begin{defn}\label{defn:mcdm_mf_ver}
    An MCDM with mean-field interaction is a tuple $\langle \mathcal{X}, A, \mathcal{D}, \mathcal{C}, p, r \rangle$, where
    \begin{itemize}
        \item the \textit{state space} $\mathcal{X}$ and the \textit{action space} $A$ are a Polish spaces, equipped with their respective Borel sigma fields (and suitable complete metrics);
        \item the set of \textit{delay values} is $\mathcal{D} = \{ d_0, \ldots, d_K\} \subset \mathbb Z_+$, with $ 0 \leq d_K < \ldots < d_0$;
        \item the set of \textit{cost values} is $\mathcal{C} = \{ c_0, c_1, \ldots, c_K\}\subset \R_+$, with $0 = c_0 < c_1 \ldots < c_K$;
        \item the \textit{transition kernel} is a map $p: \mathcal{X} \times A \times \Pcal(\mathcal{X}) \to \Pcal(\mathcal{X})$;
        \item the \textit{reward function} $r: \mathcal{X} \times A \times \Pcal(\mathcal{X})\to [0, \infty)$ is bounded and measurable.
    \end{itemize}
\end{defn}

For the Polish spaces $\mathcal{X}$ and $A$ which are separable, we fix compatible metrics $\delta_{\mathcal{X}}$ and $\delta_A$ respectively, such that $\mathcal{X}$ and $A$ are complete when equipped with $\delta_{\mathcal{X}}$ and $\delta_A$ respectively.
For any $t > 1$, the $t$--step transition kernel is a function $p^{(t)}: \mathcal{X} \times A^t \times \Pcal(\mathcal{X})^t \to \Pcal(\mathcal{X})$, defined recursively by the relation: for $x_0 \in \mathcal{X}$, $a_0,\ldots, a_{t-1} \in A$, $\mu_0, \ldots, \mu_{t-1} \in \Pcal(\mathcal{X})$,
\begin{align}\label{eq:n_step_kernel_w/measure}
    &p^{(t)}(\de x_t \mid x_0,(a)_0^{t-1}, (\mu)_0^{t-1}) \coloneqq p(\de x_t\mid x_{t-1}, a_{t-1}, \mu_{n-1})\ p^{(t-1)}(\{x_{t-1}\} \mid x_0,(a)_0^{t-2}, (\mu)_0^{t-2}).
\end{align}

We shall also frequently make use of the following lemma in subsequent proofs.

\begin{lemma}\label{lem:dominating_measure}
    For any fixed $\mu \in \Pcal(\mathcal{X})$, there exists a measure $q \in \Pcal(\mathcal{X})$ such that $p(\cdot\mid x, a,\mu)$ is absolutely continuous with respect to $q$ for all $x \in \mathcal{X}$ and $a \in A$. 
\end{lemma}

\begin{proof}
    Since $\mathcal{X}$ and $A$ is compact, and the map $(x,a) \mapsto p(\cdot \mid x,a,\mu)$ is continuous with respect to $\delta_{TV}$, the set $(p(\cdot \mid x,a,\mu))_{x\in\mathcal{X},a\in A}$ is compact in $\Pcal(\mathcal{X})$ under $\delta_{TV}$, and hence separable. Let $\{\rho_n\}_n$ be a corresponding dense set of measures, and define
    \begin{align*}
        q(C) = \sum^{\infty}_{n=1} \frac{1}{2^n} \rho_n(C),\quad C \in \mathcal{B}(\mathcal{X}).
    \end{align*}
    Clearly $q(C)= 0$ implies $\rho_n(C) = 0$ for all $n$. Since $\{\rho_n\}$ is dense, this in turn implies ${p(C \mid x,a,\mu) = 0}$ for all $x \in \mathcal{X}$ and $a \in A$, as required.
\end{proof}

\subsection{Finite agent game with observation delay}\label{subsec:nplayer}

We consider an $N$-player game with mean-field interaction, where each agent can control their observation delay. Take a tuple $\langle \mathcal{X}, A, \mathcal{D}, \mathcal{C}, p, r \rangle$ as in \Cref{defn:mcdm_mf_ver}. Let $x^n_t$ and $a^n_t$ be the state and action taken by the $n$\textsuperscript{th} player at time $t$ respectively. Then the $n$\textsuperscript{th} player receives a reward of $r(x^n_t, a^n_t, e^N_t)$ at time $t$, where $e^N_{t} = \frac{1}{N} \sum^N_{k=1} \delta_{x^k_{t}}$ is the empirical distribution of the states of the agents. The player's next transition from state $x^n_t$ is then determined by the transition kernel
$p( \de x^n_{t+1}  \mid x^n_{t}, a^n_{t}, e^N_{t})$.

\begin{defn}\label{defn:policy_nplayer}
    A policy $\pi = (\pi^1, \ldots, \pi^N)$ for the $N$-player MCDM is a sequence of maps $\pi^n_t: (H_t)^N \to \Pcal(U)$ for $1 \leq n \leq N$, $t \geq 0$ (cf. \Cref{defn:history sets}). $\pi$ is admissible if at each time $t$, there exists kernels $\phi^{j}_t: \mathcal{X}^{t +1}  \times A^{t+d_0} \times \mathcal{I}^t \times \Pcal(\mathcal{Y})^{t+1}\to \Pcal(U)$, such that for each $h_t \in (H_t)^N$,
    \begin{align*}
         \pi^n_t(\cdot \mid h_t) = \phi^{n}_t(\cdot \mid x^n_{-d_0}, \ldots, x^n_{t-{\tilde{\iota}^n_t}}, a^n_{-d_0}, \ldots, a^n_{t-1}, \tilde \iota^n_0, \ldots, \tilde\iota^n_t,  \tilde{e}^N_0, \ldots, \tilde{e}^N_t),
    \end{align*}
    where $\tilde{\iota}^n_t$ is the delay for the $n$\textsuperscript{th} player at time $t$ (cf. \Cref{defn:mcdm_admissible}), and $\tilde{e}^N_t$ is the empirical distribution of the augmented state, i.e.
    \begin{align*}
        \tilde{e}^N_t = \frac{1}{N}\sum^N_{k=1} \delta_{y^k_t}(\cdot), \quad y^k_t = (\tilde{\iota}^k_t, x^k_{t-\tilde{\iota}^k_t}, a^k_{t-\tilde{\iota}^k_t}, \ldots, a^k_{t-1}).
    \end{align*}
    The set of admissible policies for $N$-players is denoted by $\mathcal{A}$, and the set of admissible policies for player $n$ is denoted by $\mathcal{A}^n$, so that $\mathcal{A} = \prod^N_{n=1} \mathcal{A}^n$.
\end{defn}
Given $q_0 \in \Pcal(\mathcal{Y})^N$, player $n$'s objective is given by
\begin{align*}
    J^N_n(\pi) \coloneqq \E^{\pi}_{q_0}\left[ \sum^{\infty}_{t=0} \gamma^{t} \left(r(x^n_t, a^n_t, e^N_t) - \sum^K_{k=1} c_k \mathbf{1}_{\{i^n_t = k\}} \right) \right],\ \pi = (\pi^1, \ldots, \pi^N) \in \mathcal{A}.
\end{align*}
If the other players' strategies are considered fixed, then we can consider $J^N_n$ as a function of the policy of player $n$ only, and we write $J^N_n(\pi^n, \pi^{-n})$, where $\pi^{-n} = (\pi^1, \ldots, \pi^{n-1}, \pi^{n+1}, \ldots, \pi^N)$. The notion of optimality in the $N$-player game can be captured by the concept of a Nash equilibrium, which intuitively says that, at equilibrium, no player can make gains by deviating from their current strategy, provided that all other players remain at their strategy.

\begin{defn}
$\pi_{*} \in \mathcal{A}$ is a Nash equilibrium for the $N$-player MCDM if for each $n \in \{1, \ldots, N\}$,
\begin{align*}
    J^N_n(\pi_{*}) = \sup_{\pi^n \in \mathcal{A}^n} J^N_n(\pi^n, \pi_{*}^{-n}),
\end{align*}
For $\varepsilon>0$, a policy $\pi \in \mathcal{A}$ is an $\varepsilon$-Nash equilibrium for the $N$-player MCDM if for each $n \in \{1, \ldots, N \}$,
\begin{align*}
    J^N_n(\pi) \geq \sup_{\pi^n \in \mathcal{A}^n} J^N_n(\pi^n, \pi^{-n}) - \varepsilon.
\end{align*}
\end{defn}

In general, the Nash equilibrium is hard to characterise and computationally intractable. It is also impractical to search over policies that depend on the distribution of all players. Therefore, it is more useful to consider a search over Markovian policies for each player, and formulate the equilibrium condition with respect to such policies. As such, the common approach for modelling partially observable games is to consider Markovian policies as above \cite{saldi2019approximate}. This is a reasonable assumption as in practice it will be hard for each agent to keep track of the movement of all other players when the number of players is large. A policy is \textit{Markovian} if $\pi = (\pi_t)_{t \geq 0}$ is such that $\pi_t: \mathcal{Y}^N \to (\Pcal(U))^N$. Let $\mathcal{A}_{\mathrm{mrkv}}$ denote the set of Markovian policies, and $ \mathcal{A}_{\mathrm{mrkv}}^n$ for the set of Markovian policies for player $n$, so that $\mathcal{A}_{\mathrm{mrkv}} = \prod^N_{n=1} \mathcal{A}_{\mathrm{mrkv}}^n$. 

\begin{defn}
$\pi_{*} \in \mathcal{A}_{\mathrm{mrkv}}$ is a Markov--Nash equilibrium for the $N$-player MCDM if for each $n \in \{1, \ldots, N\}$,
\begin{align*}
    J^N_n(\pi_{*}) = \sup_{\pi^n \in \mathcal{A}_{\mathrm{mrkv}}^n} J^N_n(\pi^n, \pi_{*}^{-n}),
\end{align*}
For $\varepsilon>0$, a policy $\pi \in \mathcal{A}_{\mathrm{mrkv}}$ is an $\varepsilon$-Markov--Nash equilibrium for the MCDM if for each $n \in \{1, \ldots, N \}$,
\begin{align*}
    J^N_n(\pi) \geq \sup_{\pi \in \mathcal{A}_{\mathrm{mrkv}}^n} J^N_n(\pi^n, \pi^{-n}) - \varepsilon.
\end{align*}
\end{defn}

\subsection{Mean-field game with observation delay}

The computation and characterisation of (Markov--)Nash equilibria is typically intractable due to the curse of dimensionality and the coupled dynamics across the different agents. Therefore, as an approximation, we consider the infinite population limit by sending the number of players ${N \to \infty}$, and replace the empirical distribution of the agents with a measure flow ${\bm{\mu} = (\mu_t)_t \in \Pcal(\mathcal{X})^{\infty}}$. We consider the viewpoint of one representative agent, and assume that its interactions with members of the population, modelled by the measure flow $\bm{\mu}$, are symmetric.\\

Let $\langle \mathcal{X}, A, \mathcal{D}, \mathcal{C}, p, r \rangle$ be an MCDM with mean-field interaction. Let $\bm{\mu} \in \Pcal(\mathcal{X})^{\infty}$ be a fixed measure flow. At time $t$, a representative agent transitions from the state $x_t$ under the transition kernel $p(\de x_{t+1} \mid x_t, a_t, \mu_t)$, and collects a reward of $r(x_t, a_t, \mu_t)$. We impose the following additional assumptions on the MCDM-MFG for the rest of the paper.
\begin{ass}\label{assumption_lipschitz_p_r}
\begin{itemize}
    \item[(a)]There exists a constant $L_p$ such that for all $x, \hat{x} \in \mathcal{X}$, $a, \hat{a} \in A$, ${\mu, \hat{\mu} \in \Pcal(\mathcal{X}})$,
    \begin{align*}
       \delta_{TV}\left( p(\cdot \mid x, a, \mu),\  p(\cdot \mid \hat{x}, \hat{a},\hat{\mu}) \right) \leq L_p ( \delta_{\mathcal{X}}(x, \hat{x}) + \delta_A(a,\hat{a}) + \delta_{TV}( \mu , \hat{\mu}) ).
    \end{align*}
    \item[(b)] There exists a constant $L_r$ such that for all $x, \hat{x} \in \mathcal{X}$, $a, \hat{a} \in A$, ${\mu, \hat{\mu} \in \Pcal(\mathcal{X}})$,
    \begin{align*}
        \lvert r(x, a, \mu) - r(\hat{x}, \hat{a}, \hat{\mu}) \rvert \leq L_r ( \delta_{\mathcal{X}}(x, \hat{x}) + \delta_A(a,\hat{a}) + \delta_{TV}( \mu , \hat{\mu}) ).
    \end{align*}
    Moreover, $r$ is bounded by a constant $M_r > 0$.
\end{itemize}
\end{ass}

Given a fixed measure flow $\bm{\mu}\in\Pcal(\mathcal{X})^{\infty}$ and initial condition $q_0 \in \Pcal(H_0)$, a representative player faces a similar single agent problem as in \Cref{sec:single_agent}: maximise the objection function $J_{\bm{\mu}} (\pi)$ over the set of admissible policies $\mathcal{A}_{\rm DM}$ (cf. \Cref{defn:mcdm_admissible}), where
\begin{align*}
    J_{\bm{\mu}} (\pi) = \E^{\pi}_{q_0}\left[ \sum^{\infty}_{t=0} \gamma^{t} \left( r(x_t, a_t, \mu_t) -  \sum^K_{k=1} c_k \mathbf{1}_{\{i_t=k\}} \right)\right].
\end{align*}

By \Cref{prop:equivalence_aug}, we can solve the representative player MCDM by solving an equivalent full information MDP on the augmented space $\mathcal{Y} \coloneqq \bigcup_{d=d_K}^{d_0} ( \{ d \} \times \mathcal{X} \times A^d )$. In particular, given $y=(d, x, (a)^{t-1}_{t-d})$, $u = (a_t, i)$, and ${\bm{\mu} = (\mu)^{t-d_K}_{t-d_0}}$, the augmented kernel ${p_y: \mathcal{Y} \times U \times \Pcal(\mathcal{X})^{d_0-d_K+1} \to \Pcal(\mathcal{Y}})$ is defined by
\begin{align}\label{aug_transition}
    &\ p_y(\de \hat{y} \mid y, \tilde{a}, i ) \nonumber\\
     \coloneqq & \begin{cases} p^{(d-d_i+1)}\big(\de \hat{x} \mid x, (a)_{t-d}^{t-d_i}, (\mu)^{t-d_i}_{t-d}\big)\otimes \delta_{\{ (a_{t-d_i+1},\ldots, a_{t},) \}}(\de  \hat{\bm{a}}) \otimes \delta_{d_i}(\hat{d}) , & \ d_i \leq d \leq d_0;\\
     \delta_{x}(\de \hat{x}) \otimes 
    \delta_{(a_{t-d}, \ldots, a_{t})}(\de \hat{\bm{a}}) \otimes \delta_{d+1}(\hat{d}) ,& \ d_K \leq d< d_i, \end{cases}
\end{align}
where $\de \hat{y} = \sqcup_{\hat{d}}\ \de\hat{x} \de \hat{\bm{a}}$, and the augmented reward $r_y: \mathcal{Y} \times U \times \Pcal(\mathcal{X})^{d_0+1} \to [0,\infty)$ is defined by
\begin{align}\label{aug_reward}
    r_y(y, u, \bm{\mu}) \coloneqq \int_{\mathcal{X}} r(\pr{x}, a_t, \mu_t)\ p^{(d)}(\de \pr{x} \mid x, \bm{a}, (\mu)^{t-1}_{t-d}) - \sum^K_{k=1} c_k \mathbf{1}_{\{i=k\}}.
\end{align}
Given \Cref{assumption_lipschitz_p_r}, we have the following bounds for the $t$-step transition kernel $p^{(t)}$ and augmented reward $r_y$.
\begin{prop}\label{prop:lip_cont}
    Under \Cref{assumption_lipschitz_p_r}:
    \begin{itemize}
    \item[(a)] The $t$--step transition kernels are continuous with respect to $\delta_{TV}$. Moreover, it is Lipschitz in the $\mu$ argument: for all $x \in \mathcal{X}$, $(a)^{t-1}_0 \in A^t$ and ${(\mu)^{t-1}_0, (\hat{\mu})^{t-1}_0 \in \Pcal(\mathcal{X}})^t$,
    \begin{align*}
        \delta_{TV}\left( p^{(t)}(\cdot \mid x, (a)^{t-1}_0, (\mu)^{t-1}_0),\ p^{(t)}(\cdot \mid x, (a)^{t-1}_0, (\hat{\mu})^{t-1}_0) \right) \leq  t L_p\ \delta_{\mathrm{max}}((\mu)^{t-1}_0,  (\hat{\mu})^{t-1}_0).
    \end{align*}
    \item[(b)] Let $L_P = d_0 L_p$ and $L_R= L_r +2 M_r L_P$. For all $y \in \mathcal{Y}$, $u \in U$ and ${\bm{\mu}, \hat{\bm{\mu}}\in \Pcal(\mathcal{X}})^{d_0 + 1}$,
    \begin{align*}
        \lvert r_y(y, u,\bm{\mu}) - r_y(y, u,  \hat{\bm{\mu}}) \rvert \leq L_R\ \delta_{\mathrm{max}}(\bm{\mu},  \hat{\bm{\mu}}).
    \end{align*}
    Moreover, $r_y$ is bounded by $M_R \coloneqq M_r + c_K$.
    \end{itemize}
\end{prop}
\begin{proof}~~
\begin{itemize}
\item[(a)] The continuity of the $t$-step transition kernels follow directly from \eqref{eq:n_step_kernel_w/measure}. The Lipschitz property holds by assumption for the base case. Suppose now the Lipschitz property holds for some $t >1$. By \Cref{lem:dominating_measure}, we can find a dominating measure $P \in \Pcal(\mathcal{X})$, such that for any $s \in \{1, ...,t+1\}$, the $s$--step kernels are absolutely continuous with respect to $P$. Let $x_0 \in \mathcal{X}$, $(a)^t_0 \in A^{t+1}$, and $(\mu)^t_0, (\hat{\mu})^t_0\in \Pcal(\mathcal{X})^{t+1}$. For the rest of the proof, we will use the shorthand $p^{(s)}$ and $\hat{p}^{(s)}$ to represent the dependence of the $s$--step kernels on $(\mu)^{s-1}_0$ and $(\hat{\mu})^{s-1}_0$ respectively. Similarly, we denote by $q^{(s)}$ and $\hat{q}^{(s)}$ respectively, the densities of $p^{(s)}$ and $\hat{p}^{(s)}$ with respect to $P$. Then, the analogous to \eqref{eq:n_step_kernel_w/measure} holds for the densities: for $\tilde{q} \in \{q ,\hat{q}\}$,
\begin{align*}
    \tilde{q}^{(t+1)}(x_{t+1}) = \int_{\mathcal{X}} \tilde{q}(x_{t+1} \mid x_t, a_t) \tilde{q}^{(t)}(x_t \mid x_0, (a)^{t-1}_0) P(\de x_t).
\end{align*}
We then have
    \begin{align*}
        & \delta_{TV}(p^{(t+1)}(\cdot\mid x_0, (a)^t_0),\ \hat{p}^{(t+1)}(\cdot\mid x_0, (a)^t_0 ) )\\
        =\ &\frac{1}{2} \int_{\mathcal{X}} \left\lvert q^{(t+1)}(x_{t+1}\mid x_0, (a)^t_0) -  \hat{q}^{(t+1)}(x_{t+1}\mid x_0, (a)^t_0) \right \rvert P(\de x_{t+1}) \\
        \leq\ & \frac{1}{2} \int_{\mathcal{X}^2} \big\lvert q(x_{t+1}\mid x_t, a_t) (q^{(t)}( x_t\mid x_0, (a)^{t-1}_0) - \hat{q}^{(t)}(x_t\mid x_0, (a)^{t-1}_0))  \\
         & \qquad\quad  - \hat{q}^{(t)}(x_{t}\mid x_0, (a)^{t-1}_0) (q(x_{t+1}\mid x_t, a_t) - \hat{q}(x_{t+1} \mid x_t, a_t)) \big\rvert P(\de x_t) P(\de x_{t+1})\\
        \leq\ & \delta_{TV}(p^{(t)}(\cdot\mid x_0, (a)^{t-1}_0),\ \hat{p}^{(t)}(\cdot\mid x_0, (a)^{t-1}_0))\\
        & \quad + \frac{1}{2} \int_{\mathcal{X}} \sup_{x \in \mathcal{X}} \lvert q(x_{t+1}\mid x, a_t) - \hat{q}(x_{t+1} \mid x,a_t) \rvert P(\de x_{t+1}) \\
        \leq\ &  tL_p\ \delta_{\max}((\mu)^{t-1}_0,(\hat{\mu})^{-1}_0) + L_p\ \delta_{TV}(\mu_{t}, \hat{\mu}_t))  \\
        \leq\ & (t+1)L_p\ \delta_{\max}\left((\mu)^t_0, (\hat{\mu})^t_0\right),
    \end{align*}
    where in the penultimate inequality we use the assumption that $\mathcal{X}$ is compact and the transition kernels (and therefore also the densities) are continuous with respect to $x$. The result then holds by induction.
\item[(b)] Let $y=\left(d, x, \bm{a}\right)$, and assume $\bm{\mu}$ and $\bm{\hat{\mu}}$ are indexed over time $t-d_0$ to $t$, then
\begin{align*}
    \ & \lvert r_y(y, u, \bm{\mu}) - r_y(y, u, \hat{\bm{\mu}}) \rvert \\
 =\ & \left\lvert  \int_{ \mathcal{X}} r(\pr{x}, a_t, \mu_t)\ p^{(d)}(\de \pr{x} \mid x, \bm{a}, (\mu)^{t-1}_{t-d}) -  \int_{\mathcal{X}} r(\pr{x}, a_t, \hat{\mu}_t)\ p^{({d})}(\de \pr{x} \mid {x}, {\bm{a}}, (\hat{\mu})^{t-1}_{t-d})\right\rvert \\
 \leq\ &  \left\lvert  \int_{\mathcal{X}} r(\pr{x}, a_t, \mu_t) \left(p^{(d)}(\de \pr{x} \mid x, \bm{a}, (\mu)^{t-1}_{t-d}) - \ p^{({d})}(\de \pr{x} \mid {x}, {\bm{a}}, (\hat{\mu})^{t-1}_{t-d}\right) \right\rvert\\
 & \quad +  \left\lvert  \int_{\mathcal{X}} \left( r(\pr{x}, a_t, \mu_t)-   r(\pr{x}, a_t, \hat{\mu}_t)\right) p^{({d})}(\de \pr{x} \mid {x}, {\bm{a}}, (\hat{\mu})^{t-1}_{t-d}) \right\rvert \\
 \leq\ & \lambda(r)\ \delta_{TV}(  p^{(d)}\big(\cdot  \mid x ,\bm{a}, (\mu)^{t-1}_{t-d}\big),\ p^{({d})}\big(\cdot \mid {x}, {\bm{a}}, (\hat{\mu})^{t-1}_{t-d}\big) ) \\
 & \quad + \sup_{x \in \mathcal{X}} \lvert r(x,a, \mu_t) - r(x,{a}, \hat{\mu}_t) \rvert \\
 \leq\ & 2M_r  L_P\ \delta_{\mathrm{max}}(\bm{\mu}, \hat{\bm{\mu}}) + L_r \delta_{\mathrm{max}}(\bm{\mu}, \hat{\bm{\mu}}),
\end{align*}
where \Cref{prop: maxmin_ineq} is used for the second inequality. The second part on boundedness is immediate from the definition of $r_y$.
    \end{itemize}
\end{proof}

\subsection{Mean-field Nash equilibrium}\label{sub:mfne} 

We proceed to define the mean-field Nash equilibrium (MFNE) for the mean-field MCDM problem. As the presence of observation delays leads to a non-Markovian problem, it is more convenient to formulate the MFNE fixed point in terms of the augmented space. We shall give the MFNE definition in terms of both the underlying and augmented problem, and show that the two notions are equal. We note that even in the augmented formulation, the objective function $J_{\bm{\mu}}$ depends on the measure flow $\bm{\mu} \in \Pcal(\mathcal{X})^{\infty}$ on the underlying state space. Therefore, in order to define the MFNE in the augmented problem, we construct a map that takes for each time $t$, a measure $\nu_t$ on the augmented space, to a sequence of measures $(\mu_{t,d})^{d_0}_{d=0} \in \Pcal(\mathcal{X})^{d_0+1}$. We will use the following operation which integrates a transition kernel against a measure.

\begin{defn}\label{def:star_operator}
    Let $E,F,G$ be Borel spaces. Let $\nu \in \Pcal(E \times F)$ and $p: E \times \mathcal{B}(G) \to [0,1]$ be a transition kernel. We define the measure $v \star p \in \Pcal(G \times F)$ as follows. For any open set $O = O_1 \times O_2 \in G \times F$,
    \begin{align*}
          (v \star p)(O_1 \times O_2) = \int_E p(O_1 \mid e) \mathbf{1}_{O_2}\ \nu (\de e \times \de f).
    \end{align*}
\end{defn}

We will typically take $E = \mathcal{X} \times A$, $F = A^{d-1}$ and $G = \mathcal{X}$ for some $d \in \{d_K,\ldots, d_0\}$. In this setting, $v \star p \in \Pcal(\mathcal{X} \times A^{d-1})$, and we can define $v \star_k p \coloneqq (v \star_{k-1} p) \star p  \in \Pcal(\mathcal{X} \times A^{d-k})$ for $k \in \{0, \ldots, d\}$. If $\bar{p} = (p_1, \ldots ,p_k)$ is a sequence of transition kernels, write $v \star_k \bar{p} \coloneqq ((v \star p_1) \star \ldots) \star p_k$. We will also write $p^{\mu}$ when we want to indicate the dependence of the transition kernel $p$ on the measure $\mu$.\\

To construct the aforementioned map, we enlarge $\mathcal{Y}$ to 
\begin{align*}
    \tilde{\mathcal{Y}} \coloneqq \bigcup_{d=d_K}^{d_0} \left( \{d\} \times \mathcal{X}^{d_0-d+1} \times A^d \right) = \bigsqcup^{d_0}_{d=d_K} \left( \mathcal{X}^{d_0-d+1} \times A^d \right) \eqqcolon \bigsqcup^{d_0}_{d=d_K}\tilde{\mathcal{Y}}_d.
\end{align*}
In this instance, an element $\tilde{y}_n \in \tilde{\mathcal{Y}}$ can be viewed as
$\tilde{y}_n = (d, x_{n-d_0}, \ldots, x_{n-d}, a_{n-d}, \ldots, a_{n-1})$. This extra enlargement of $\mathcal{Y}$ to $\tilde{\mathcal{Y}}$ 
%
%
enables 
the measure flow construction in \Cref{defn:aug_to_underlying} below, as well as the
formulation
of the MFNE fixed-point condition as 
in \Cref{defn:mfne_mcdm}.
%
%



\begin{defn}\label{defn:aug_to_underlying}
    Let $\bm{\nu} = (\nu_t)_t \in \Pcal(\tilde{\mathcal{Y}})^{\infty}$. For each $t \geq 0$, we write $\nu_t = \sum_d w^d_t \nu^d_t$, where $w^d_t \in [0,1]$ and $\nu^d_t \in \Pcal(\tilde{\mathcal{Y}}_d)$. For $d \leq \pr{d} \in \{0,\ldots, d_0\}$, let $\nu^{x_{t-\pr{d}}}_t$ and $\nu^{x_{t-\pr{d}}\vert d}_t$ be the marginals of $\nu_t$ and $\nu^{d}_t$ in the $x_{t-\pr{d}}$ coordinate respectively. Starting with $\pr{d} = d_0$, we define $\mu_{t,d_0} = \nu^{x_{t-d_0}}_t \in \Pcal(\mathcal{X})$. Then, for each $0 \leq   \pr{d} < d_0$, we define
\begin{align*}
    \mu_{t,\pr{d}}&= \sum^{d_0}_{d=d_K} w^d_t \xi^{d}_{t,\pr{d}},\quad \xi^{d}_{t,\pr{d}}\in \Pcal(\mathcal{X}), \\
    \xi^{d}_{t,\pr{d}}(C) &= \begin{cases}
        \nu^{x_{t-\pr{d}}\mid d}_t(C), & d \leq \pr{d}; \\
       (\nu^{d,x,a}_t \star_{d - \pr{d}} p^{d, \pr{d}}_{\mu}) (C \times A^{\pr{d}}), & d > \pr{d},
    \end{cases}\\
     p_{d, \pr{d}}^{\mu} &= (p^{\mu_{t,d}},\ldots, p^{\mu_{t, \pr{d}-1}}),\ C \in \mathcal{B}(\mathcal{X}),
\end{align*}
where $\nu^{d,x,a}_t$ is the marginal of $\nu^d_t$ on the $(x_{t-d}, a_{t-d},\ldots , a_{t-1})$--coordinates. Finally, we define $\bm{\mu}^{\nu}_t \coloneqq (\mu_{t,d})^{d_0}_{d=0}$ and $\mathcal{M}(\nu_0) \coloneqq \mu_{t,0}$.
\end{defn}

We interpret $\bm{\mu}^{\nu}_t$ in \Cref{defn:aug_to_underlying} as the distribution of the underlying states from time $t-d_0$ to time $t$, computed from $\nu_t$ -- the distribution of the observable past states at time $t$. 
The following lemma shows that the mapping $\nu_t \mapsto \bm{\mu}^{\nu}_t$ is Lipschitz, and will be useful in the later analysis.

\begin{prop}\label{lem:nu_to_mu} 
    The mapping $\nu_t \mapsto \bm{\mu}^{\nu}_t$ is Lipschitz with constant $L_M = \sum^{d_0}_{d=0} (d L_p)^d$.
\end{prop}
\begin{proof}
Let $\nu_t, \hat{\nu_t} \in \Pcal(\tilde{\mathcal{Y}})$, with respective images $\bm{\mu}^{\nu}_t = (\mu_{t,d})^{d_0}_{d=0}$ and $\bm{\mu}^{\hat{\nu}}_t = (\hat{\mu}_{t,d})^{d_0}_{d=0}$. As in \Cref{defn:aug_to_underlying}, we can write $\nu_t = \sum_d w^d_t \nu^d_t$ and $\hat{\nu}_t = \sum_d \hat{w}^d_t \hat{\nu}^d_t$. We use the shorthand
\begin{align*}
    \delta(d) = \delta_{TV}(\mu_{t,d}, \hat{\mu}_{t,d}),\quad \delta(\Bar{d}, \pr{d}) = \delta_{\max}((\mu_{t,d})^{\Bar{d}}_{\pr{d}}, (\hat{\mu}_{t,d})^{\Bar{d}}_{\pr{d}}).
\end{align*}
By a similar argument to \Cref{lem:dominating_measure}, we can find a measure $q \in \Pcal(\mathcal{X})$, against which $\xi^{\Bar{d}}_{t,\pr{d}}$ is absolutely continuous for all $0 \leq \Bar{d}, \pr{d} < d_0$. Therefore, with slight abuse of notation, we shall use $\xi^{\Bar{d}}_{t,\pr{d}}$ to denote both the measure and its corresponding density. By definition we have
\begin{align*}
    \delta(d_0) = \delta_{TV}(\nu^{x_{t-d_0}}_t,\hat{\nu}^{x_{t-d_0}}_t)\leq \delta_{TV}(\nu_t, \hat{\nu}_t).
\end{align*}
Suppose now that $\pr{d} < d_0$. Then
\begin{align}\label{eq:aug_to_underlying_map_proof}
    \delta(\pr{d}) &= \frac{1}{2} \int_{\mathcal{X}} \Bigg\lvert \underbrace{\sum^{d_0}_{d=d_K} w^d_t \xi^{d}_{t, \pr{d}}(\pr{x})}_{I_1(\pr{x})} - \underbrace{\sum^{d_0}_{d=d_K} \hat{w}^d_t \hat{\xi}^{d}_{t, \pr{d}}(\pr{x})}_{I_2(\pr{x})} \Bigg\rvert q(\de \pr{x})
\end{align}
We can write $I_1$ and $I_2$ as
\begin{align*}
    I_1 &= \sum^{\pr{d}}_{d =d_K} w^d_t \nu^{x_{t-\pr{d}}\vert d}_t + \underbrace{\sum^{d_0}_{d = \pr{d}+1} w^d_t\cdot [\nu^{d,x,a}_t \star_{d - \pr{d}} p_{d, \pr{d}}^{\mu} ]_\mathcal{X}}_{J_1}\ ,\\
    I_2 &=\sum^{\pr{d}}_{d =d_K } \hat{w}^d_t \hat{\nu}^{x_{t-\pr{d}}\vert d}_t + \underbrace{\sum^{d_0}_{d = \pr{d}+1} \hat{w}^d_t\cdot [\hat{\nu}^{d,x,a}_t \star_{d - \pr{d}} p_{d, \pr{d}}^{\hat{\mu}} ]_\mathcal{X}}_{J_2}\ ,
\end{align*}
where $[\cdot]_{\mathcal{X}}$ denotes the marginal of the measure on $\mathcal{X}$ (cf. \Cref{defn:aug_to_underlying}). Then
\begin{align*}
    &\int_{\mathcal{X}} \lvert J_1 (\pr{x})- J_2 (\pr{x}) \rvert q(\de \pr{x}) \\
    \leq &\ \sum^{d_0}_{d = \pr{d}+1} \int_{\mathcal{X}} \Bigg\lvert  w^d_t\cdot [v^{d,x,a}_t \star_{d - \pr{d}} p_{d,\pr{d}}^{\mu} ]_\mathcal{X}\ (\pr{x}) - \hat{w}^d_t\cdot  [\hat{v}^{d,x,a}_t \star_{d - \pr{d}}  p_{d,\pr{d}}^{\hat{\mu}} ]_\mathcal{X}\ (\pr{x}) \Bigg\rvert q_x(\de \pr{x})\\
    \leq &\ \sum^{d_0}_{d = \pr{d}+1} \bigg(\int_{\mathcal{X}} \Bigg\lvert  w^d_t\cdot \Big([v^{d,x,a}_t \star_{d - \pr{d}}p_{d,\pr{d}}^{\mu} ]_\mathcal{X} -[\hat{v}^{d,x,a}_t \star_{d - \pr{d}} p_{d,\pr{d}}^{\hat{\mu}}]_\mathcal{X} \Big)\ (\pr{x}) \Bigg\rvert q(\de \pr{x})\\
   &\ \qquad \qquad + \int_{\mathcal{X}} \Bigg\lvert  w^d_t\cdot [v^{d,x,a}_t \star_{d - \pr{d}} p_{d,\pr{d}}^{\hat{\mu}}]_\mathcal{X}\ (\pr{x}) - \hat{w}^d_t\cdot  [\hat{v}^{d,x,a}_t \star_{d - \pr{d}} p_{d,\pr{d}}^{\hat{\mu}}]_\mathcal{X}\ (\pr{x}) \Bigg\rvert q(\de \pr{x}) \bigg)\\
    \leq&\ 2 \sum^{d_0}_{d = \pr{d}+1} \left((d-\pr{d})(L_p)^{d-\pr{d}} \cdot \delta(d, \pr{d}+1) + \lvert w^d_t\nu^{d,x,a}_t - \hat{w}^d_t\hat{\nu}^{d,x,a}_t \rvert \right).
\end{align*}
Returning to \eqref{eq:aug_to_underlying_map_proof}, we have
\begin{align*}
    \delta(\pr{d}) & \leq \frac{1}{2}   \int_{\mathcal{X}} \bigg(\lvert J_1(\pr{x}) - J_2(\pr{x}) \rvert + \sum^{\pr{d}}_{d = d_K}  \left\lvert w^d_t \nu^{x_{t-\pr{d}}\vert d}_t(\pr{x}) - \hat{w}^d_t\hat{\nu}^{x_{t-\pr{d}}\vert d}_t(\pr{x}) \right\rvert \bigg) q(\de \pr{x}) \\
    &\leq\ \sum^{d_0}_{d = \pr{d}+1}(d-\pr{d})(L_p)^{d-\pr{d}}\cdot \delta(d, \pr{d}+1) + \delta_{TV}\left(\nu^{(x)^{t-\pr{d}}_{t-d_0}, (a)^{t-\pr{d}-1}_{t-d_0}}_t,\ \hat{\nu}^{(x)^{t-\pr{d}}_{t-d_0}, (a)^{t-\pr{d}-1}_{t-d_0}}_t \right)\\
    &\leq\ \sum^{d_0}_{d = \pr{d}+1}(d-\pr{d})(L_p)^{d-\pr{d}}\cdot \delta(d_0) + \delta_{TV}\left(\nu^{(x)^{t-\pr{d}}_{t-d_0}, (a)^{t-\pr{d}-1}_{t-d_0}}_t,\ \hat{\nu}^{(x)^{t-\pr{d}}_{t-d_0}, (a)^{t-\pr{d}-1}_{t-d_0}}_t \right),
\end{align*}
where the superscripts above $\nu_t$ and $\hat{\nu}_t$ denote their marginals on the corresponding coordinates. Therefore, we have
\begin{align*}
    \delta_{\mathrm{max}}(\bm{\mu}^{\nu}_t, \bm{\mu}^{\hat{\nu}}_t) = \max_{0\leq d \leq d_0} \delta_{TV}(\mu_{t,d}, \hat{\mu}_{t,d}) \leq \sum^{d_0}_{d=0} (d L_p)^d\ \delta_{TV}(\nu_t, \hat{\nu}_t),
\end{align*}
as required.
\end{proof}

\begin{rem}
    A superficially similar task is solved using an elegant approach in \cite{saldi2019approximate,saldi2022partially} for a partial observation problem of a different kind, where the respective $\bm{\mu}^{\nu} = (\bm{\mu}^{\nu}_t)_t$ can be obtained as barycenters of measures ${\nu}_t$ describing the fully observable belief state (posterior state distribution). The barycenter approach makes use of the different structure of their partial observation setting, where observations of states are perturbed by unbiased noise. In our case, the belief state is parameterised by a finite set given by past observations, so the notion of taking the barycenter does not apply here. Instead, exploiting the fact that our state and action spaces are assumed to be finite, we will construct the required $\bm{\mu}^{\nu}_t$ explicitly by repeatedly applying the transition kernel \eqref{aug_transition}.
\end{rem}

We are now ready to define the mean-field Nash equilibrium (MFNE) through a fixed-point characterisation. Note that since policies can only rely on past information, we require an additional measure $\nu_0$ to act as an initial condition.

\begin{defn}\label{defn:mfne}
Let $\bm{\mu}=(\mu_t)_t \in \Pcal(\mathcal{X})^{\infty}$ and $\nu_0 \in \Pcal(\mathcal{Y})$ such that $\mu_ 0 = \mathcal{M}(\nu_0)$. Define:
\begin{enumerate}
    \item [(i)] The best-response map $\Phi_{\nu_0}: \Pcal(\mathcal{X})^{\infty} \to \mathcal{A}_{\rm DM}$, given by
        \begin{align*}
            \Phi_{\nu_0}(\bm{\mu}) 
            \coloneqq \left\{ \hat{\pi} \in \mathcal{A}_{\rm DM}:  J_{\bm{\mu}}( \hat{\pi}) = \sup_{\pi \in \mathcal{A}_{\rm DM}} J_{\bm{\mu}}(\pi) \right\}.
        \end{align*}
    \item[(ii)] The measure flow map $\Psi: \mathcal{A}_{\rm DM} \to \Pcal(\mathcal{X})^{\infty}$, defined recursively by $\Psi(\pi)_0 \coloneqq \mu_0$ and
    \begin{align*}
        \Psi(\pi)_{t+1}(\cdot) \coloneqq \E^{\pi}_{\mu_0}\left[ \int_{\mathcal{X}} p(\cdot \mid x ,\pi^a_t, \mu_t)\Psi(\pi)_t(\de x) \right]
    \end{align*}
    where $\pi^a_t$ is the marginal of $\pi_t$ on $A$.
    \item[(iii)] ${(\pi^{*}, \bm{\mu}^{*}) \in \mathcal{A}_{\rm DM} \times \Pcal(\mathcal{X})^{\infty}}$ is a mean-field Nash equilibrium (MFNE), if it is a fixed point of the map, i.e. $\pi^{*} \in \Phi_{\nu_0}(\bm{\mu}^{*})$ and $\bm{\mu}^{*} = \Psi(\pi^{*})$.
\end{enumerate}
\end{defn}

To define an MFNE in the augmented space, for any measure flow $\bm{\nu} \in \Pcal(\tilde{\mathcal{Y}})^{\infty}$, consider the augmented objective
\begin{align*}
    J_{\bm{\nu}}(\pi) \coloneqq \E^{\pi}_{\nu_0}\left[ \sum^{\infty}_{t=0} \gamma^{t} r_y(y_t, u_t, \bm{\mu}^{\nu}_t) \right],
\end{align*}
where $\E^{\pi}_{\nu_0}$ is the expectation induced by initial measure $\nu_0$, policy $\pi$, and transition kernel $p_y$, noting that the induced measures $\bm{\mu}^{\nu}$ are used when evaluating $p_y$ in \eqref{aug_transition}.
\begin{defn}\label{defn:mfne_mcdm}
Let $\bm{\nu}=(\nu_t)_t \in \Pcal(\tilde{\mathcal{Y}})^{\infty}$. Define:
\begin{enumerate}
    \item [(i)] The best-response map $\Phi^{\mathrm{aug}}: \Pcal(\tilde{\mathcal{Y}})^{\infty} \to \mathcal{A}_{\rm DM}$, given by
        \begin{align*}
            \Phi^{\mathrm{aug}}(\bm{\nu}) \coloneqq \left\{ \hat{\pi} \in \mathcal{A}_{\rm DM}:  J_{\bm{\nu}}( \hat{\pi}) = \sup_{\pi \in \mathcal{A}_{\rm DM}} J_{\bm{\nu}}(\pi) \right\}.
        \end{align*}
    \item[(ii)] The measure flow map $\Psi^{\mathrm{aug}}: \mathcal{A}_{\rm DM} \to \Pcal(\tilde{\mathcal{Y}})^{\infty}$, defined recursively by $\Psi^{\mathrm{aug}}(\pi)_0 \coloneqq \nu_0$ and
    \begin{align*}
        \Psi^{\mathrm{aug}}(\pi)_{t+1}(\cdot) \coloneqq\int_{\tilde{\mathcal{Y}}\times U} p_y\Big(\cdot \mid y, u, \bm{\mu}_t^{\Psi^{\mathrm{aug}}(\pi)} \Big) \pi_t(\de u \mid y) \Psi^{\mathrm{aug}}(\pi)_t(\de y).
    \end{align*}
    \item[(iii)] ${(\pi^{*}, \bm{\nu}^{*}) \in \mathcal{A}_{\rm DM} \times \Pcal(\tilde{\mathcal{Y}})^{\infty}}$ is a mean-field Nash equilibrium (MFNE) for the augmented MCDM problem, if it is a fixed point of the map $\Psi^{\mathrm{aug}} \circ \Phi^{\mathrm{aug}}$, i.e., $\pi^{*} \in \Phi^{\mathrm{aug}}(\bm{\nu}^{*})$ and $\bm{\nu}^{*} = \Psi^{\mathrm{aug}}(\pi^{*})$.
\end{enumerate}
\end{defn}

Note that thanks to the definition of $\Psi^{\rm aug}$ in (ii), the support of any measure $\bm{\nu}^{*}$ that satisfies the MFNE condition is consistent (cf. \Cref{defn:mcdm_admissible}). The following proposition demonstrates the equivalence of the two formulations of the MFNE.

\begin{prop}
    Let ${(\pi^{*}, \bm{\nu}^{*}) \in \mathcal{A}_{\rm DM} \times \Pcal(\tilde{\mathcal{Y}})^{\infty}}$ be an MFNE for the augmented problem. Then the induced pair  ${(\pi^{*}, \bm{\mu}^{\nu,*}) \in \mathcal{A}_{\rm DM} \times \Pcal(\mathcal{X})^{\infty}}$ is an MFNE for the underlying problem. Conversely, given ${(\pi^{*}, \bm{\mu}^{*}) \in \mathcal{A}_{\rm DM} \times \Pcal(\mathcal{X})^{\infty}}$, the induced pair ${(\pi^{*}, \bm{\nu}^{\mu,*}) \in \mathcal{A}_{\rm DM} \times \Pcal(\tilde{\mathcal{Y}})^{\infty}}$ is an MFNE for the augmented problem.
\end{prop}
\begin{proof}
    Assume that $(\pi^{*}, \bm{\nu}^{*})$ is an MFNE in the augmented space. By definition we have $\pi^{*} \in \Phi(\bm{\mu}^\nu)$ and also $\mu^{\nu,*}_0 = \mathcal{M}(\nu_0) = \Psi(\pi^{*})_0$. Since $\bm{\nu}^{*}$ is the image of $\Psi^{\rm aug}$, it holds that $\mu^{\nu,*}_{t, 0} = \mu^{\nu,*}_{t+1, 1}$ for any $t \geq 0$. Next, assume $\mu^{\nu,*}_t = \Psi(\pi^{*})_t$ holds true, then
    \begin{align*}
     \Psi(\pi)_{t+1}(\cdot) & = \E^{\pi}_{\mu_0}\left[ \int_{\mathcal{X}} p(\cdot \mid x ,\pi^{a,*}_t, \mu^{\nu,*}_t)\mu^{\nu,*}_t(\de x) \right] \\
     &= \int_A \int_{\mathcal{X}} p(\cdot \mid x ,a, \mu^{\nu}_{t+1, 1})\mu^{\nu,*}_{t+1, 1}(\de x) \pi^{a,*}_t(\de a)\\
     & = \mu^{\nu,*}_{t+1, 0}(\cdot) = \mu^{\nu,*}_{t+1} (\cdot),
    \end{align*}
    noting that the distribution of  $\pi^{a,*}_t$ is the same as the marginal of $\nu^{*}_{t+1}$ on the variable $a_t$. Therefore the claim that $(\pi^{*}, \bm{\mu}^{\nu,*})$ is an MFNE for the underlying problem holds by induction. 
    
    Conversely, let $(\pi^{*}, \bm{\mu}^{*})$ be an MFNE for the underlying problem. With the joint distribution of all states and actions given, we can directly construct a measure flow $\bm{\nu}^{\mu,*} \in \Pcal(\tilde{\mathcal{Y}})^{\infty}$ such that $\bm{\mu}^{\bm{\nu}^{\mu,*}} = \bm{\mu}^{*}$, so that $\pi^{*} \in \Phi^{\rm aug}(\bm{\nu}^{\mu,*})$. The condition that $\bm{\nu}^{\mu,*} = \Psi^{\rm aug}(\pi^{*})$ is trivially satisfied.
\end{proof}

The existence of MFNE for discrete MFGs in the fully observable case is demonstrated in \cite{saldi_mfg_discrete}, and further extended to the partial information case in \cite{saldi2019approximate}. For completeness sake, we include the result here, citing the relevant result.

\begin{thm} \label{thm:existence}
    Under the assumptions of the MCDM in \Cref{defn:mcdm}, \Cref{ass_compacct}, and \Cref{assumption_lipschitz_p_r}(a), there exists a mean-field Nash equilibrium for the mean-field game under the MCDM, in the sense of \Cref{defn:mfne} (or \Cref{defn:mfne_mcdm}).
\end{thm}

\begin{proof}
    By \Cref{prop:lip_cont}, the augmented transition kernel $p_y$ is weakly continuous. Together with the aforementioned assumptions, we can apply \cite[Theorem 3.1]{saldi2019approximate}, which derives the existence of MFNE via an application of the Kakutani fixed point theorem. 
\end{proof}

\begin{rem}
When computing the best-response for a representative player, the measure flow $\bm{\nu}$ is considered as fixed. Any best-response policy is therefore Markovian in $\mathcal{Y}$: the extended augmentation to  $\tilde{\mathcal{Y}}$ is only required for the measure flow map.
\end{rem}

\section{Regularised MFG-MCDM} \label{sec:regularise}

It is well-known that in MFGs, the MFNE need not be unique, and that in general the fixed point operator given by ${\Psi^{\mathrm{aug}} \circ \Phi^{\mathrm{aug}}}$ is not contractive \cite{cui2021approximately}. In order to compute an approximate MFNE, we mirror the approaches of \cite{cui2021approximately, saldi_regularisation} and consider an entropy regularised game. For $P, Q$ in $\Pcal(U)$, the Kullback–Leibler (KL) divergence of $P$ from $Q$ is defined to be
\begin{align*}
    D_{KL}(P \Vert Q) \coloneqq \begin{cases}\int_U \log \left(\frac{\de P}{\de Q}(u) \right) Q(\de u), & P \ll Q, \\
    \infty, & \mbox{otherwise},
    \end{cases}
\end{align*}
where $\frac{\de P}{\de Q}$ is the Radon--Nikodym derivative of $P$.
\begin{ass}
    Throughout the rest of the paper, we assume that any reference measure $q \in \Pcal(U)$ has full support on $U$.
\end{ass} 
With $\eta > 0$ and $\bm{\nu} \in \Pcal(\mathcal{\tilde{Y}})^{\infty}$, and $q \in \Pcal(U)$ acting as a reference measure, we consider the regularised objective function $J_{\eta, \bm{\nu}}: \mathcal{A}_{\rm DM} \to \R$, defined by
\begin{align*}
    J_{\eta, \bm{\nu}}(\pi) \coloneqq \E^{\pi}_{q_0}\left[\sum^{\infty}_{t=0} \gamma^{t}\left( r_y \big(y_t, u_t, \bm{\mu}^{\nu}_t\big) - \eta D_{KL} (\pi_t \Vert q) \right) \right],
\end{align*}
and the corresponding regularised optimisation problem with the value function $J^{*}_{\eta,\bm{\nu}}: \N \times \mathcal{Y} \to \R$, defined by
\begin{align*}
    J^{*}_{\eta,\bm{\nu}}(t,y) \coloneqq  \sup_{\pi \in \mathcal{A}_{\rm DM}} \E^{\pi}_{t,y}\left[ \sum^{\infty}_{s=t} \gamma^{s-t} \left( r_y \big(y_s, u_s, \bm{\mu}^{\nu}_s\big) - \eta D_{KL} (\pi_s \Vert q) \right)\right].
\end{align*}

In this regularised game, optimal policies are of softmax form as described next.

\begin{thm}\label{thm_dpp_regularise}
    The value function $J^{*}_{\eta,\bm{\nu}}$ is the unique solution to the Bellman equation
    \begin{align}\label{eq:regularised_dpp}
        J^{*}_{\eta,\bm{\nu}}(t, y) = \sup_{\hat{q} \in \Pcal(U)} \int_{U} \bigg( r_y \big(y_t, u, \bm{\mu}^{\nu}_t\big) &- \eta \log \frac{\de \hat{q}}{\de q}(u)  \nonumber \\
         &  + \gamma \int_{\tilde{\mathcal{Y}}} J^{*}_{\eta,\bm{\nu}}(t+1, y^{\prime}) p_y\big(\de y^{\prime}\mid y, u, \bm{\mu}^{\nu}_t\big) \bigg) \hat{q}(\de u),
    \end{align}
    and satisfies
    \begin{align*}
        J^{*}_{\eta,\bm{\nu}}(t,y) = \eta \log \left(\int_U\exp\left( \frac{ Q^{*}_{\eta,\bm{\nu}}(t,y,u)}{\eta} \right) q(\de u) \right),
    \end{align*}
    where the function $Q^{*}_{\eta,\bm{\nu}}: \N \times \mathcal{Y} \times U$, defined by
     \begin{align}\label{regularised_Q_func}
     Q^{*}_{\eta,\bm{\nu}}(t,y,u) \coloneqq r_y\big(y, u,  \bm{\mu}^{\nu}_t\big) + \gamma \int_{\tilde{\mathcal{Y}}} J^{*}_{\eta,\bm{\nu}}(t+1, y^{\prime}) p_y\big(\de y^{\prime}\mid y, u, \bm{\mu}^{\nu}_t\big),
    \end{align}
    is bounded and measurable. Moreover, the softmax policy $\pi^{\rm soft}$, defined by
    \begin{align*}
        \pi^{\rm soft}_t(\de u \mid y) \coloneqq \frac{\exp\left(  Q^{*}_{\eta,\bm{\nu}}(t,y,u)/\eta\right)}{\int_U \exp\left( Q^{*}_{\eta,\bm{\nu}}(t,y,\pr{u})/\eta\right) q(\de \pr{u})} q(\de u)
    \end{align*}
    is the unique optimal policy that attains the supremum in \eqref{eq:regularised_dpp}.
\end{thm} 

\begin{proof}
This is a specific case of the dynamic programming principle for a full information problem. We refer to \cite[Theorem B.1]{kerimkulov2023fisher} for a self-contained exposition of the idea of the proof for the stationary problem, the extension of which to the non-stationary case is straightforward.
\end{proof}

We next define the regularised MCDM-MFNE by an analogous fixed point criterion.
\begin{defn}\label{defn:regularised MFNE-MCDM}
Let $\bm{\nu}=(\nu_t)_t \in \Pcal(\tilde{\mathcal{Y}})^{\infty}$ and $\eta > 0$.
\begin{enumerate}
    \item[(i)] We define the map $\Phi_{\eta}: \Pcal(\tilde{\mathcal{Y}})^{\infty} \to \mathcal{A}_{\rm DM}$ to be the best-response map from \Cref{thm_dpp_regularise}, given by
        \begin{align*}
            \Phi_{\eta}(\bm{\nu})_t(\de u\mid y) \coloneqq \frac{\exp\left(  Q^{*}_{\eta,\bm{\nu}}(t,y,u)/\eta\right)}{\int_U \exp\left( Q^{*}_{\eta,\bm{\nu}}(t,y,u)/\eta\right) q(\de u)} q(\de u),\quad t \geq 0.
        \end{align*}
    \item[(ii)] We define the map $\Psi^{\mathrm{aug}}: \mathcal{A}_{\rm DM} \to \Pcal(\tilde{\mathcal{Y}})^{\infty}$ to be the measure flow map, where $\Psi^{\mathrm{aug}}(\pi)_0 = \nu_0$ and
        \begin{align*}
            \Psi^{\mathrm{aug}}(\pi)_{t+1}(\cdot) =\int_{\tilde{\mathcal{Y}}} \int_U p_y\Big(\cdot \mid y, u, \bm{\mu}_t^{\Psi^{\mathrm{aug}}(\pi)} \Big) \pi_t(\de u \mid y) \Psi^{\mathrm{aug}}(\pi)_t(\de y), \quad t \geq 0.
        \end{align*}
    \item[(iii)] A regularised MFNE for the MCDM problem $(\pi^{*}, \bm{\nu}^{*})\in \mathcal{A}_{\rm DM} \times \Pcal(\tilde{\mathcal{Y}})^{\infty}$ is defined by a fixed point $\bm{\nu}^{*}$ of $\Psi^{\mathrm{aug}} \circ \Phi_{\eta}$, for which $\pi^{*}  = \Phi_{\eta}(\bm{\nu}^{*})$ (best policy response to a given measure flow) and $\bm{\nu}^{*} = \Psi^{\mathrm{aug}}(\pi^{*})$ (measure flow induced by policy) holds.
\end{enumerate}
\end{defn}

Our next step is to show that the fixed point operator $\Psi^{\mathrm{aug}} \circ \Phi_{\eta}$ forms a contraction mapping under a suitable choice of metric and regularisation parameter $\eta$, such that the iteration of these maps converges towards the unique fixed point, which is the regularised MFNE. We combine the approaches of \cite{cui2021approximately, saldi_regularisation}, extending their proof to the case of a non-stationary infinite horizon problem on general Polish spaces with time-varying measure flows, as well as the inclusion of the map $\nu_t \mapsto \bm{\mu}^{\nu}_t$ in the definitions of $\Phi_{\eta}$ and $\Psi^{\mathrm{aug}}$.

In order to demonstrate contraction of the regularised iterations, we show next a series of propositions regarding the Lipschitz continuity of the individual mappings. We consider a truncation at some finite time horizon $T$, defining the functions $\JTstar_{\eta, \bm{\nu}}$ by
\begin{align*}
    \JTstar_{\eta, \bm{\nu}}(t,y) \coloneqq  \begin{cases}\sup_{\pi \in \mathcal{A}_{\rm DM}} \E^{\pi}_{t,y}\left[ \sum^{T}_{s=t} \gamma^{s-t} \left( r_y \big(y_s, u_s, \bm{\mu}^{\nu}_s\big) - \eta D_{KL} (\pi_s \Vert q) \right)\right], &  t \leq T, \\
    0, & t > T.
    \end{cases}
\end{align*}
Similarly, let us define the truncated $Q$-functions by
\begin{align*}
    \QTstar_{\eta,\bm{\nu}}(t,y,u) \coloneqq \begin{cases} r_y\big(y, u,  \bm{\mu}^{\nu}_t\big) + \gamma \int_{\tilde{\mathcal{Y}}} \JTstar_{\eta,\bm{\nu}}(t+1, y^{\prime}) p_y\big(\de y^{\prime}\mid y, u, \bm{\mu}^{\nu}_t\big), & t \leq T, \\
    0, & t >T.
    \end{cases}
\end{align*}
\begin{prop}\label{prop:dpp_truncated_q}
    For $0 \leq t <T$, the truncated $Q$-functions satisfy
\begin{align*}
     \QTstar_{\eta,\bm{\nu}}(t,y,u) &= r_y\big(y, u,  \bm{\mu}^{\nu}_t\big) \nonumber \\
     &\quad + \gamma \eta \int_{\tilde{\mathcal{Y}}} \log \left( \int_U \exp \left(\frac{\QTstar_{\eta,\bm{\nu}}(t+1, \pr{y}, \pr{u})}{\eta} \right) q(\de \pr{u})\right) p_y(\de \pr{y} \mid y, u,  \bm{\mu}^{\nu}_t), \\
     \QTstar_{\eta,\bm{\nu}}(t,y,u) &= r_y\big(y, u,  \bm{\mu}^{\nu}_t\big) \nonumber \\
     &\quad+ \gamma \eta  \int_{\tilde{\mathcal{Y}}}   \log \left(\int_U\exp \left(\frac{Q^{T-1,*}_{\eta,\bm{\nu}}(t, \pr{y}, \pr{u})}{\eta} \right) q(\de \pr{u})\right) p_y(\de \pr{y} \mid y, u,  \bm{\mu}^{\nu}_t).
\end{align*}
\end{prop}

\begin{proof}
    This follows directly from \Cref{thm_dpp_regularise} and bythe definition of $\QTstar_{\eta, \bm{\nu}}$.
\end{proof}

It is a standard result that the value function and $Q$-functions in finite horizon converges pointwise to their infinite horizon counterparts \cite[Section 4.2]{hernandezlerma_1}. In particular, we will use the fact that $\JTstar_{\eta,\bm{\nu}}$ and 
$\QTstar_{\eta,\bm{\nu}}$ converges pointwise towards $J^{*}_{\eta,\bm{\nu}}$ and $Q^{*}_{\eta,\bm{\nu}}$ respectively in the analysis of the remainder of the section.
\begin{lemma}\label{lemma:q_bound}
For any $\bm{\nu}, \hat{\bm{\nu}} \in \Pcal(\tilde{\mathcal{Y}})^{\infty}$, and for all $(t,y,u) \in \N \times \mathcal{Y} \times U$,  $\lvert \QTstar_{\eta,\nu}(t,y,u) \rvert \leq q^{*}$ and $\lvert Q^{*}_{\eta,\bm{\nu}}(t,y,u) \rvert \leq q^{*}$, where $q^{*} \coloneqq M_R / ( 1- \gamma)$.
\end{lemma}
\begin{proof}
First, we have $\QTstar_{\eta,\bm{\nu}}(t,\cdot,\cdot) = 0$ for $t >T$, and 
\begin{align*}
    \lvert \QTstar_{\eta,\bm{\nu}}(T,y,u) \rvert = \lvert r_y\big(y, u,  \bm{\mu}^{\nu}_T\big) \rvert \leq M_R \eqqcolon q_{T,T}.
\end{align*}
Then, for each $t < T$,
\begin{align*}
    \lvert \QTstar_{\eta,\bm{\nu}}(t,y,u) \rvert & \leq M_R + \gamma \eta \sup_{\pr{y} \in \tilde{\mathcal{Y}}} \left\lvert  \log \left( \int_U \exp \left(\frac{\QTstar_{\eta,\bm{\nu}}(t+1, \pr{y}, \pr{u})}{\eta}\right) q(\de \pr{u})\right) \right\rvert \nonumber \\
    &\leq M_R + \gamma \eta \left(\frac{q_{T,t+1}}{\eta} \right) \nonumber \\
    & = M_R + \gamma q_{T,t+1} \eqqcolon q_{T,t}
\end{align*}
where $q_{T,t+1}$ is the bound for $\QTstar_{\eta,\bm{\nu}}(t+1,y,u)$. As $T \to \infty$, $q_{T,0}$ converges to the fixed point $q^{*}$ of the map $x \mapsto \gamma x + M_R$. Moreover, for all $T > 0$ and $ 0 \leq t \leq T$, $q_{T,t} \leq q^{*}$. By pointwise convergence, the other bound for $Q^{*}_{\eta, \bm{\nu}}$ follows by sending $T \to \infty$.
\end{proof}

Next, by induction we conclude the following statement:

\begin{lemma}
     Let $\bm{\nu}, \hat{\bm{\nu}} \in \Pcal(\tilde{\mathcal{Y}})^{\infty}$. Then 
     for each $T \geq 0$, the truncated $Q$-functions satisfies
     \begin{align}\label{eq_qn_lipschitz}
     \left\lvert \QTstar_{\eta, \bm{\nu}}(t,y,u) - \QTstar_{\eta, \hat{\bm{\nu}}}(t,y,u) \right \rvert \leq l_{T,t}\ \delta_{TV}(\nu_t, \hat{\nu}_t),\quad 0 \leq t \leq T,\ y \in \tilde{\mathcal{Y}},\ u \in U,\ \eta \geq 0,
    \end{align}
    where the doubly-indexed sequence $(l_{T,t})_{T,t}$ , with $0 \leq t \leq T$, is given by
    \begin{align*}
        l_{T,T} \coloneqq L_R L_M\quad \mathrm{and} \quad l_{T,t} \coloneqq \left(L_RL_M + \gamma \exp\left(\frac{2q^{*}}{\eta}\right) l_{T-1,t} + 2\gamma q^{*} L_PL_M \right).
    \end{align*}
\end{lemma}
   
\begin{proof}
    For the base case $T=0$, we have
    \begin{align*}
    \lvert Q^{0,*}_{\eta, \bm{\nu}}(0, y, u) - Q^{0,*}_{\eta, \hat{\bm{\nu}}}(0, y, u) \rvert = \lvert r_y\big(y, u,  \bm{\mu}^{\nu}_0\big) - r_y\big(y, u,  \bm{\mu}^{\hat{\nu}}_0\big) \rvert \leq L_R L_M \delta_{TV}(\nu_0, \hat{\nu}_0)
\end{align*}
Now assume the hypothesis holds for some $T\geq0$, and consider the truncated $Q$-functions with terminal time $T+1$. For the case $ t = T+1$, we have the desired inequality by the same argument as above, with $l_{T+1, T+1} = L_R L_M$. Next, for $ 0 \leq t \leq T$, we get
\begin{align*}
    & \left\lvert Q^{T+1,*}_{\eta, \bm{\nu}}(t,y,u) - Q^{T+1,*}_{\eta, \hat{\bm{\nu}}}(t,y,u)  \right\rvert\\
    \leq\ & \left\lvert r_y(y,u,  \bm{\mu}^{\nu}_t) -  r_y(y,u,  \bm{\mu}^{\hat{\nu}}_t)\right\rvert \\
    & + \gamma\eta \Bigg\lvert \int_{\tilde{\mathcal{Y}}} \log \left( \int_U \exp \left(\frac{Q^{T,*}_{\eta,\bm{\nu}}(t, \pr{y}, \pr{u})}{\eta} \right) q(\de \pr{u})\right) p_y(\de \pr{y} \mid y, u,  \bm{\mu}^{\nu}_t) \\
    & \qquad \quad - \int_{\tilde{\mathcal{Y}}}    \log \left(\int_U \exp \left(\frac{Q^{T,*}_{\eta,\hat{\bm{\nu}}}(t, \pr{y}, \pr{u})}{\eta} \right) q(\de \pr{u})\right) p_y(\de \pr{y} \mid y, u,  \bm{\mu}^{\hat{\nu}}_t) \Bigg\rvert  \\
    \leq\ & L_R\cdot \delta_{\mathrm{max}}( \bm{\mu}^{\nu}_t,\bm{\mu}^{\hat{\nu}}_t)\\
    & + \gamma\eta \sup_{\pr{y} \in \tilde{\mathcal{Y}}}\left\lvert \log \left( \int_U \exp \left(\frac{Q^{T,*}_{\eta,\bm{\nu}}(t, \pr{y}, \pr{u})}{\eta}\right) q(\de \pr{u})\right) -  \log \left(\int_U \exp \left(\frac{Q^{T,*}_{\eta,\hat{\bm{\nu}}}(t, \pr{y}, \pr{u})}{\eta} \right) q(\de \pr{u})\right)  \right\rvert \\
    & + \gamma \eta\ \delta_{TV}\left( p_y( \cdot \mid y, u, \bm{\mu}^{\nu}_t),\ p_y( \cdot \mid y, u,  \bm{\mu}^{\hat{\nu}}_t) \right) \cdot \\
    & \qquad  \left\lvert \log \left( \int_U\exp\left(\frac{Q^{T,*}_{\eta,\bm{\hat{\nu}}}(t, y_{\mathrm{max}}, \pr{u})}{\eta}\right) q(\de\pr{u})\right) - \log \left(\int_U\exp\left(\frac{Q^{T,*}_{\eta,\bm{\hat{\nu}}}(t, y_{\mathrm{min}}, \pr{u})}{\eta}\right) q(\de \pr{u})\right) \right\rvert\\
    \eqqcolon\ & L_R\cdot \delta_{\mathrm{max}}( \bm{\mu}^{\nu}_t,\bm{\mu}^{\hat{\nu}}_t) + \gamma \eta ( I_1 + I_2)
\end{align*}
where we used the fact that $\tilde{\mathcal{Y}}$ is compact and \Cref{prop: maxmin_ineq} to obtain the the last term in the last inequality for some $y_{\mathrm{max}}, y_{\mathrm{min}} \in \tilde{\mathcal{Y}}$. Then, by \Cref{lemma:q_bound} that $\lVert Q^{T,*}_{\eta,\bm{\nu}} \rVert_{\infty} \leq q^{*}$, we have
\begin{align*}
    I_2 \leq \frac{2}{\eta} q^{*} L_P\ \delta_{\mathrm{max}}( \bm{\mu}^{\nu}_t,\bm{\mu}^{\hat{\nu}}_t) \leq  \frac{2}{\eta} q^{*} L_P L_M \delta_{TV}(\nu_t, \hat{\nu}_t).
\end{align*}

To bound $I_1$, we use \Cref{mean_value_theorem} as follows. Let $C(U)$ be the Banach space of continuous functions from the compact set $U$ to $\R$, equipped with the sup norm. Let $L: C(U) \to \R$ be the functional defined by
$L(f) = \log \left( \int_U \exp(f(u)) q(\de u) \right)$. It is easy to verify that $L$ is Fr\'echet differentiable, and for $f \in C(U)$ the Fr\'echet derivative $dL_f$ is given by
\begin{align*}
    dL_f(h) = \frac{\int_U h(u) \exp (f(u)) q(\de u)}{\int_U \exp (f(u)) q(\de u)},\quad h \in C(U).
\end{align*}
By \Cref{mean_value_theorem}, there exists $t^{*} \in (0,1)$ such that for $f^{*}(\cdot) = (t^{*}Q^{T,*}_{\eta,\bm{\nu}}(t, \pr{y}, \cdot) - (1-t^{*})Q^{T,*}_{\eta,\hat{\bm{\nu}}}(t, \pr{y}, \cdot))/\eta$, we have
\begin{align*}
    I_1 &\coloneqq  \sup_{\pr{y} \in \tilde{\mathcal{Y}}} \lvert dL_{f^{*}} \{(Q^{T,*}_{\eta,\bm{\nu}}(t, \pr{y}, \cdot) - Q^{T,*}_{\eta,\hat{\bm{\nu}}}(t, \pr{y}, \cdot))/\eta \}\rvert\\
    &= \sup_{\pr{y} \in \tilde{\mathcal{Y}}} \left \lvert\frac{1}{\eta}\frac{\int_U (Q^{T,*}_{\eta,\bm{\nu}}(t, \pr{y}, u)- Q^{T,*}_{\eta,\hat{\bm{\nu}}}(t, \pr{y}, u)) \exp (f^{*}(u)) q(\de u)}{\int_U \exp (f^{*}(u)) q(\de u)}\right\rvert \\
    & \leq \frac{1}{\eta} \exp \left(\frac{2q^{*}}{\eta}\right) l_{T,t}\ \delta_{TV}(\nu_t, \hat{\nu}_t),
\end{align*}
where the last line follows from the induction assumption and the fact that $\lVert f^{*} \rVert_{\infty} \leq q^{*}/\eta$. Combining all the above, we complete the induction step by observing that
\begin{align*}
     \left\lvert Q^{T+1,*}_{\eta, \bm{\nu}}(t,y,u) - Q^{T+1,*}_{\eta, \hat{\bm{\nu}}}(t,y,u) \right \rvert \leq \left(L_RL_M + \gamma \exp\left(\frac{2q^{*}}{\eta}\right) l_{T,t} + 2\gamma q^{*} L_PL_M \right) \delta_{TV}(\nu_t, \hat{\nu}_t).
\end{align*}
\end{proof}

\begin{prop}\label{lem:qregstar_lip}
    Consider the metric $\delta_Q$ given by
    \begin{align*}
        \delta_Q(Q^{*}_{\eta, \bm{\nu}}, Q^{*}_{\eta, \hat{\bm{\nu}}}) &\coloneqq \sum^{\infty}_{t=0} \zeta^{-t} \sup_{\substack{y \in \tilde{\mathcal{Y}}\\ u \in U}} \left\lvert Q^{*}_{\eta, \bm{\nu}}(t,y,u) - Q^{*}_{\eta, \hat{\bm{\nu}}}(t,y,u) \right \rvert.
    \end{align*}
    Let $l_{\eta} \coloneqq (L_M(L_R + 2 \gamma q^{*}L_P))/(1- \gamma \exp(\frac{2q^{*}}{\eta}))$. Then for any $\eta > \frac{2M_R}{-(1-\gamma)\log \gamma}$, we have that
    \begin{align*}
     \delta_Q(Q^{*}_{\eta, \bm{\nu}}, Q^{*}_{\eta, \hat{\bm{\nu}}})  \leq l_{\eta} \delta_{\infty}(\bm{\nu}, \hat{\bm{\nu}}).
    \end{align*}
\end{prop}

\begin{proof}
Since $\frac{2M_R}{-(1-\gamma)\log \gamma} = -\frac{2q^{*}}{\log \gamma}$, $\eta > \frac{2M_R}{-(1-\gamma)\log \gamma}$ implies that $\gamma \exp\left(\frac{2q^{*}}{\eta}\right) < 1$. Now fix $t>0$ in \eqref{eq_qn_lipschitz} and send $T \to \infty$. The sequence $l_{T,t}$ converges towards to the unique fixed point of the map
\begin{align*}
    x \mapsto L_R L_M + \gamma \exp\left(\frac{2q^{*}}{\eta}\right) x +  2\gamma q^{*} L_P L_M,
\end{align*}
which is given by $l_{\eta}$, so that for each $t$,
\begin{align*}
     \left\lvert Q^{*}_{\eta, \bm{\nu}}(t,y,u) - Q^{*}_{\eta, \hat{\bm{\nu}}}(t,y,u) \right \rvert \leq l_{\eta}\ \delta_{TV}(\nu_t, \hat{\nu}_t).
\end{align*}
Since $l_{\eta}$ does not depend on $t$, we obtain the claimed upper bound
\begin{align*}
    \delta_Q(Q^{*}_{\eta, \bm{\nu}}, Q^{*}_{\eta, \hat{\bm{\nu}}}) &= \sum^{\infty}_{t=0} \zeta^{-t} \sup_{\substack{y \in \tilde{\mathcal{Y}}\\ u \in U}} \left\lvert Q^{*}_{\eta, \bm{\nu}}(t,y,u) - Q^{*}_{\eta, \hat{\bm{\nu}}}(t,y,u) \right \rvert \\
    & \leq \sum^{\infty}_{t=0} \zeta^{-t} l_{\eta}\ \delta_{TV}(\nu_t, \hat{\nu}_t)\\
    & = l_{\eta}\ \delta_{\infty}({\bm{\nu}, \hat{\bm{\nu}}}).
\end{align*}
\end{proof}

\begin{lemma}\label{prop_softmax_derivative}
    On the Banach space $C(U)$ (equipped with the sup norm $\lVert \cdot \rVert_{\infty}$), the map $I:C(U) \to C(U)$, given by
    \begin{align*}
        I(f)(\cdot) \coloneqq \frac{\exp(f(\cdot))}{\int_U \exp(f(u)) q(\de u)},
    \end{align*}
    is Fr\'echet differentiable. At each $f \in C(U)$, its Fr\`echet derivative $dI_f$ is given by
    \begin{align*}
        dI_f(h)(\cdot) = I(f)(\cdot)\left[h(\cdot) - \int_U h(u) I(f)(u) q(\de u) \right],\quad h \in C(U).
    \end{align*}
    The Fr\`echet derivative $dI_f \in L(C(U), C(U))$ is a bounded linear operator with operator norm $\lVert dI_f \rVert \leq 2 \lVert I(f)\rVert_{\infty}$.
\end{lemma}

\begin{proof}
    We use the capital $O$ for the Landau notation in the proof as follows. For functions  $g: \R \to \R$ and $G: \R_+ \to \R$, we say that $g(x) \coloneqq O(G(x))$ if there exists $M>0$ such that $\lvert g(x) \rvert \leq M \lvert G(x) \rvert$ as $x \to 0$. Now let $f \in C(U)$ and let $Z \coloneqq \int_U \exp(f(u)) q(\de u)$. Let $h \in C(U)$ be such that $\frac{1}{Z}\int_U h(u) \exp(f(u)) q(\de u) < 1$. By Taylor's approximation for $x \mapsto \exp(x)$ at the point $h(u)$, we get
    \begin{align*}
        \left[\int_U\exp(f(u) + h(u)) q(\de u)\right]^{-1} &= \left[\int_U\exp(f(u)) (1 + h(u) + O(h(u)^2)) q(\de u)\right]^{-1} \\
        & = \frac{1}{Z} \left[ 1 + \frac{\int_U h(u) \exp(f(u)) q(\de u)}{Z} + O(\lVert h \rVert_{\infty} ^2)\right]^{-1} \\
        & = \frac{1}{Z} \left(1 -  \frac{\int_U h(u) \exp(f(u)) q(\de u)}{Z} + O(\lVert h \rVert_{\infty} ^2)\right),
    \end{align*}
    using $\frac{1}{1+x} =  1 -x + O(x^2)$ for the last equality. Then we have 
    \begin{align*}
        I(f+h)(x) & = \frac{\exp(f(x))( 1 + h(x) + O(h(x)^2))}{Z}\left(1 -  \frac{\int_U h(u) \exp(f(u)) q(\de u)}{Z} + O(\lVert h \rVert_{\infty} ^2)\right) \\
        & = I(f)(x) + I(f)(x) h(x) - I(f)(x) \int_U h(u) I(f)(u) q(\de u) + O(\lVert h \rVert_{\infty}^2).
    \end{align*}
    Therefore, we obtain
    \begin{align*}
        \lim_{\lVert h \rVert_{\infty} \to 0} \frac{\lVert I(f+h) - I(f) - I(f)[h - \int_Uh(u)I(f)(u)q(\de u)] \rVert }{\lVert h \rVert_{\infty}} = 0,
    \end{align*}
    to conclude $dI_f(h)(\cdot) = I(f)(\cdot)\left[h(\cdot) - \int_U h(u) I(f)(u) q(\de u) \right]$, as claimed. Now for any $g \in C(U)$ and $ u \in U$, we have
    \begin{align*}
        \lvert dI_f(g)(u) \rvert & = \left \lvert I(f)(u)\left[g(u) - \int_U g(\pr{u})I(f)(\pr{u}) q(\de \pr{u}) \right] \right\rvert \\
        & \leq \lvert I(f)(u) g(u) \rvert + \lvert I(f)(u) \rvert \int_U\lvert g(\pr{u})I(f)(\pr{u})\rvert q(\de \pr{u}) \\
        & \leq 2 \lVert I(f) \rVert_{\infty} \lVert g \rVert_{\infty},
    \end{align*}
    where the last inequality uses the fact that $\int_U I(f)(u) q(\de u) = 1$, as the integrand is continuous on the compact set $U$, hence integrable. Therefore,
    \begin{align*}
        \lVert dI_f \rVert = \inf\big\{ c \geq 0 : \lVert dI_f(g) \rVert_{\infty} \leq c \lVert g \rVert_{\infty},\ g \in C(U) \big\} \leq 2 \lVert I(f) \rVert_{\infty}.
    \end{align*}
\end{proof}

\begin{prop}\label{cor:lip_best_response}
Let $\eta^{*}$ be a constant with $\eta^{*} > \frac{2M_R}{-(1-\gamma)\log \gamma}$. Then for any $\eta \geq \eta^{*}$, the map $\Phi_{\eta}$ is $K_{\eta}$-Lipschitz continuous with $K_{\eta} = \frac{2q^{*} l_{\eta^{*}}}{\eta^2} $.
\end{prop}
\begin{proof}
    Given any $\bm{\nu} \in \Pcal(\tilde{\mathcal{Y}})^{\infty}$, $\Phi_{\eta}$ maps $\bm{\nu}$ to the softmax policy $ \pi^{\rm soft}$, given by
\begin{align*}
    \pi^{\rm soft}_{\bm{\nu},t}(\de u \mid y) \coloneqq \frac{\exp\left(  Q^{*}_{\eta,\bm{\nu}}(t,y,u)/\eta\right)}{\int_U \exp\left( Q^{*}_{\eta,\bm{\nu}}(t,y,\pr{u})/\eta\right) q(\de \pr{u})} q(\de u), \quad t \geq 0,
\end{align*}
where the denominator of $\pi^{\rm soft}_t(\de u \mid y)$ is finite, since $Q^{*}_{\eta,\bm{\nu}}$ is continuous and $U$ is compact. Recall from \eqref{metric:policy} that 
\begin{align*}
    \delta_{\mathcal{A}}(\Phi_{\eta}(\bm{\nu}), \Phi_{\eta}(\hat{\bm{\nu}}) ) = \sum^{\infty}_{t=0} \zeta^{-t} \sup_{y \in \tilde{\mathcal{Y}}} \delta_{TV}(\pi^{\rm soft}_{\bm{\nu},t}(\cdot \mid y), \pi^{\rm soft}_{\bm{\nu},t}(\cdot \mid y)), \quad \zeta >1.
\end{align*}
As $\tilde{\mathcal{Y}}$ is compact and $\pi^{\rm soft}_{\bm{\nu},t}$ is continuous in $y$, for any $\bm{\nu}, \hat{\bm{\nu}} \in \Pcal(\tilde{\mathcal{Y}})^{\infty}$ there exists some $\hat{y}$ such that
\begin{align*}
    & \delta_{\mathcal{A}}(\Phi_{\eta}(\bm{\nu}), \Phi_{\eta}(\hat{\bm{\nu}}) ) \\
    =\ & \sum^{\infty}_{t=0} \zeta^{-t}\ \delta_{TV}\left(  \pi^{\mathrm{soft}}_{\bm{\nu},t}(\cdot \mid \hat{y}),\ \pi^{\mathrm{soft}}_{\hat{\bm{\nu}},t}(\cdot\mid \hat{y}) \right) \\
    =\ & \sum^{\infty}_{t=0} \zeta^{-t} \frac{1}{2} \int_U \left\lvert \frac{\exp\left(  Q^{*}_{\eta,\bm{\nu}}(t,\hat{y},u)/\eta\right)}{\int_U \exp\left( Q^{*}_{\eta,\bm{\nu}}(t,\hat{y},\pr{u})/\eta\right) q(\de \pr{u})} - \frac{\exp\left(  Q^{*}_{\eta,\hat{\bm{\nu}}}(t,\hat{y},u)/\eta\right)}{\int_U \exp\left( Q^{*}_{\eta,\hat{\bm{\nu}}}(t,\hat{y},\pr{u})/\eta\right) q(\de \pr{u})}\right\rvert q(\de u) .
\end{align*}

We apply \Cref{prop_softmax_derivative} to the integrand, with the functions $f(u) \coloneqq Q^{*}_{\eta,\bm{\nu}}(t,y,u)/\eta$ and $\hat{f}(u) \coloneqq Q^{*}_{\eta,\hat{\bm{\nu}}}(t,y,u)/\eta$, to obtain
\begin{align*}
    & \delta_{\mathcal{A}}(\Phi_{\eta}(\bm{\nu}), \Phi_{\eta}(\hat{\bm{\nu}}) ) \\
    \leq\ & \sum^{\infty}_{t=0} \zeta^{-t}  \frac{\lVert I(f) - I(\hat{f}) \rVert_{\infty}}{2} \\
    \leq\ & \sum^{\infty}_{t=0} \zeta^{-t} \frac{1}{\eta} \sup_{u\in U} \lvert tQ^{*}_{\eta,\bm{\nu}}(t,\hat{y},u) + (1-t) Q^{*}_{\eta,\hat{\bm{\nu}}}(t,\hat{y},u)\rvert \lVert f - \hat{f}\rVert_{\infty} \\
    \leq\ & \sum^{\infty}_{t=0} \zeta^{-t} \frac{2q^{*} l_{\eta^{*}}}{\eta^2} \delta_{TV}(\nu_t, \hat{\nu}_t)\\
    =\ &\frac{2q^{*} l_{\eta^{*}}}{\eta^2} \delta_{\infty}(\bm{\nu}, \hat{\bm{\nu}}).
\end{align*}
\end{proof}

Next, we show that the measure flow map $\Psi^{\mathrm{aug}}$ from \Cref{defn:regularised MFNE-MCDM} is Lipschitz, under a suitable choice of the constant $\zeta$ in the metrics $\delta_{\infty}$ and $\delta_{\mathcal{A}}$ (cf. \eqref{eq_flow_metric} and \eqref{metric:policy}). Intuitively, given two policies, which are close in the sense of the metric $\delta_{\mathcal{A}}$, the corresponding measure flows will gradually drift apart at most by a bounded rate. The choice of $\zeta$ amounts to the weighting one gives to the current time over the distant future.

\begin{prop}\label{prop:lip_measure_flow}
    For $\zeta \in \N$ such that $\zeta > 2L_P L_M +2$ in the metric from $\delta_{\infty}$ \eqref{eq_flow_metric}, the map $\Psi^{\mathrm{aug}}: \mathcal{A}_{\rm DM} \to \Pcal(\tilde{\mathcal{Y}})^{\infty}$ (cf. \Cref{defn:regularised MFNE-MCDM}) is Lipschitz: for $\pi, \hat{\pi} \in \mathcal{A}_{\rm DM}$,
    \begin{align*}
        \delta_{\infty}(\Psi^{\mathrm{aug}}(\pi), \Psi^{\mathrm{aug}}(\hat{\pi})) \leq L_{\Psi} \delta_{\mathcal{A}}(\pi, \hat{\pi}),
    \end{align*}
    where
    \begin{align*}
        L_{\Psi} \coloneqq \frac{2L_P}{2L_P L_M +1} \left( \frac{\zeta}{\zeta - 2L_P L_M -2} + \frac{1}{\zeta-1} \right).
    \end{align*}
\end{prop}

\begin{proof}
    We show inductively that
\begin{align}\label{eq:psi_aug_induction}
    \delta_{TV}( \Psi^{\mathrm{aug}}(\pi)_{t}, \Psi^{\mathrm{aug}}(\hat{\pi})_{t}) \leq S_t\ \delta_{\mathcal{A}}(\pi, \hat{\pi}),\quad t\geq 0,
\end{align}
for constants $S_t \in \R^+$ with $S_{t+1} =  2L_P + 2S_t (L_P L_M +1)$, $S_0 = 0$. For $t=0$, we have from \Cref{defn:regularised MFNE-MCDM},
    \begin{align*}
        \delta_{TV}( \Psi^{\mathrm{aug}}(\pi)_0, \Psi^{\mathrm{aug}}(\hat{\pi})_0) = \delta_{TV}( \nu_0, \nu_0) = 0.
    \end{align*}
    Assume the hypothesis \eqref{eq:psi_aug_induction} holds true for some $t > 0$. By \Cref{lem:dominating_measure} one can write the total variation distance between $\Psi^{\mathrm{aug}}(\pi)_{t}$ and $\Psi^{\mathrm{aug}}(\hat{\pi})_{t} $ with respect to some dominating measure $q$. We write $\psi(\pi)_t$ and $\psi(\hat{\pi})_t$ for their respective densities. This leads to
    \begin{align*}
     \delta_{TV}\left( \Psi^{\mathrm{aug}}(\pi)_{t+1}, \Psi^{\mathrm{aug}}(\hat{\pi})_{t+1} \right)
        =\ & \frac{1}{2}\int_{\tilde{\mathcal{Y}}} \Bigg\lvert \int_{\tilde{\mathcal{Y}}} \int_U \bigg( p_y\Big(\pr{y} \mid y, u, \bm{\mu}_t^{\Psi^{\mathrm{aug}}(\pi)} \Big) \pi_t(\de u \mid y) \psi(\pi)_t( y) \\
        & \qquad  -  p_y\Big(\pr{y} \mid y, u, \bm{\mu}_t^{\Psi^{\mathrm{aug}}(\hat{\pi})} \Big) \hat{\pi}_t(\de u \mid y) \psi(\hat{\pi})_t(y) \bigg) \Bigg\rvert q(\de \pr{y})\\
        \leq\ & \frac{1}{2}( J_1 + J_2),
    \end{align*}
    where
    \begin{align*}
        J_1 & \coloneqq \int_{\tilde{\mathcal{Y}}}\bigg\vert \int_{\tilde{\mathcal{Y}}}\int_U \Big(p_y\Big(\pr{y} \mid y, u, \bm{\mu}_t^{\Psi^{\mathrm{aug}}(\pi)} \Big) \pi_t(\de u \mid y) \\
        & \quad \quad - p_y\Big( \pr{y} \mid y, u, \bm{\mu}_t^{\Psi^{\mathrm{aug}}(\hat{\pi})} \Big) \hat{\pi}_t(\de u \mid y) \Big)  \psi(\pi)_t(y)\ \bigg\vert q(\de \pr{y})  \\
        & \leq \int_{\tilde{\mathcal{Y}}\times \tilde{\mathcal{Y}}} \int_U \bigg\vert p_y\Big(\pr{y} \mid y, u, \bm{\mu}_t^{\Psi^{\mathrm{aug}}(\pi)} \Big) \pi_t(\de u \mid y) \\
        & \quad \quad - p_y\Big( \pr{y} \mid y, u, \bm{\mu}_t^{\Psi^{\mathrm{aug}}(\hat{\pi})} \Big) \hat{\pi}_t(\de u \mid y) \bigg\vert  \psi(\pi)_t(y) q(\de \pr{y}),
    \end{align*}
    and
    \begin{align*}
    J_2 &  \coloneqq  \int_{\tilde{\mathcal{Y}} \times \tilde{\mathcal{Y}}} \int_U p_y\left(\pr{y} \mid y, u, \bm{\mu}_t^{\Psi^{\mathrm{aug}}(\hat{\pi})} \right) \hat{\pi}_t(\de u \mid y) \left \lvert \psi_t(\pi)(y) - \psi(\hat{\pi})_t(y) \right\rvert q(\de\pr{y}) \\
        & \leq 2\ \delta_{TV}(\Psi^{\mathrm{aug}}(\pi)_t, \Psi^{\mathrm{aug}}(\hat{\pi})_t).
    \end{align*}
    The integral over $u \in U$ in $J_1$ can be simplified to
    \begin{align*}
        &\int_U \left\lvert p_y\Big(\pr{y} \mid y, u, \bm{\mu}_t^{\Psi^{\mathrm{aug}}(\pi)} \Big) \pi_t(\de u \mid y) - p_y\Big( \pr{y} \mid y, u, \bm{\mu}_t^{\Psi^{\mathrm{aug}}(\hat{\pi})} \Big) \hat{\pi}_t(\de u \mid y) \right\rvert \\
        \leq\ & \int_U p_y\Big(\pr{y} \mid y, u, \bm{\mu}_t^{\Psi^{\mathrm{aug}}(\pi)} \Big) \lvert\pi_t(\de u \mid y) - \hat{\pi}_t(\de u \mid y) \rvert\\
        & \quad + \int_U  \left\lvert p_y\Big(\pr{y} \mid y, u, \bm{\mu}_t^{\Psi^{\mathrm{aug}}(\pi)} \Big) - p_y\Big( \pr{y} \mid y, u, \bm{\mu}_t^{\Psi^{\mathrm{aug}}(\hat{\pi})} \Big) \right \rvert \hat{\pi}_t(\de u \mid y)\\
        \leq\ & \bigg( p_y\Big(\pr{y} \mid y, u_{\mathrm{max}}, \bm{\mu}_t^{\Psi^{\mathrm{aug}}(\pi)} \Big) - p_y\Big(\pr{y} \mid y, u_{\mathrm{min}}, \bm{\mu}_t^{\Psi^{\mathrm{aug}}(\pi)} \Big) \bigg) \delta_{TV}(\pi_t(\cdot \mid y), \hat{\pi}_t(\cdot \mid y))  \\
        & \quad + \sup_{u \in U} \left\lvert p_y\Big(\pr{y} \mid y, u, \bm{\mu}_t^{\Psi^{\mathrm{aug}}(\pi)} \Big) - p_y\Big( \pr{y} \mid y, u, \bm{\mu}_t^{\Psi^{\mathrm{aug}}(\hat{\pi})} \Big) \right \rvert,
    \end{align*}
    for some $u_{\mathrm{max}}, u_{\mathrm{min}} \in U$, so that
    \begin{align*}
        J_1 & \leq 2L_P\ \delta_{\mathcal{A}}( \pi, \hat{\pi}) + 2 \sup_{y \in \tilde{\mathcal{Y}}}  \sup_{u \in U} \delta_{TV}\left( p_y\Big(\cdot \mid y, u, \bm{\mu}_t^{\Psi^{\mathrm{aug}}(\pi)} \Big), p_y\Big( \cdot \mid y, u, \bm{\mu}_t^{\Psi^{\mathrm{aug}}(\hat{\pi})} \Big) \right)\\
        & \leq 2L_P\ \delta_{\mathcal{A}}( \pi, \hat{\pi})  + 2 L_P L_M\ \delta_{TV}(\Psi^{\mathrm{aug}}(\pi)_t, \Psi^{\mathrm{aug}}(\hat{\pi})_t ).
    \end{align*}
    Then, we can conclude the inductive step by applying the inductive hypothesis, noting
    \begin{align*}
        \delta_{TV}( \Psi^{\mathrm{aug}}(\pi)_{t+1}, \Psi^{\mathrm{aug}}(\hat{\pi})_{t+1} )
        \leq \ &  2L_P\ \delta_{\mathcal{A}}( \pi, \hat{\pi} ) + (2 L_P L_M +2)\ \delta_{TV}(\Psi^{\mathrm{aug}}(\pi)_t, \Psi^{\mathrm{aug}}(\hat{\pi})_t ) \\
        \leq \ & \big(2L_P + 2S_t (L_P L_M +1)\big) \delta_{\mathcal{A}}( \pi, \hat{\pi}).
    \end{align*}
    Next, we observe that $S_t$ satisfies a first-order linear recurrence relation, and more generally admits the explicit formula
    \begin{align*}
        S_t = \frac{2L_P}{2L_P L_M +1}(2L_P L_M +2)^t - \frac{2L_P}{2L_P L_M +1}.
    \end{align*}
    Therefore, fixing some $\zeta \in \N$ such that $\zeta > 2L_P L_M +2$, we obtain the upper bound
    \begin{align*}
        \delta_{\infty}( \Psi^{\mathrm{aug}}(\pi), \Psi^{\mathrm{aug}}(\hat{\pi})) &= \sum^{\infty}_{t=0} \zeta^{-t} \delta_{TV}( \Psi^{\mathrm{aug}}(\pi)_{t}, \Psi^{\mathrm{aug}}(\hat{\pi})_{t})\\
        & \leq \sum^{\infty}_{t=1} \frac{S_t}{\zeta^t}\ \delta_{\mathcal{A}}(\pi, \hat{\pi})\\
        &= \frac{2L_P}{2L_P L_M +1} \left( \frac{\zeta}{\zeta - 2L_P L_M -2} + \frac{1}{\zeta-1} \right) \delta_{\mathcal{A}}(\pi, \hat{\pi}).
    \end{align*}
\end{proof}

The next statement, showing the regularised fixed point operator of the MFG-MCDM to be a contraction map, now is essentially a corollary of the previous propositions.

\begin{thm}\label{thm:contraction_reg_mcdm}
    Let $\zeta$ and $\eta^{*}$ be constants such that $\zeta > 2L_P L_M +2$, and $\eta^{*} >  \frac{2M_R}{-(1-\gamma)\log \gamma}$, and recall the constants
    \begin{align*}
            q^{*} = \frac{M_R}{1-\gamma},\quad l_{\eta^{*}} = \frac{L_M(L_R + 2 \gamma q^{*}L_P)}{1- \gamma \exp\left(\frac{2q^{*}}{\eta^{*}}\right)}, \quad L_{\Psi} = \frac{2L_P}{2L_P L_M +1} \left( \frac{\zeta}{\zeta - 2L_P L_M -2} + \frac{1}{\zeta-1} \right).
    \end{align*}
    Then, for any $\eta> \sqrt{2q^{*} l_{\eta^{*}} L_{\Psi}}$, the fixed point operator $\Psi^{\mathrm{aug}} \circ \Phi_{\eta}$ is a contraction mapping on the space $(\Pcal(\tilde{\mathcal{Y}})^{\infty}, \delta_{\infty})$, where the constant $\zeta$ in the metric $\delta_{\infty}$ in \eqref{eq_flow_metric} is as chosen above. In particular, there exists a unique fixed point for $\Psi^{\mathrm{aug}} \circ \Phi_{\eta}$, which is a regularised MFNE for the MFG-MCDM problem.
\end{thm}

\begin{proof}
Let $\bm{\nu}, \hat{\bm{\nu}} \in \Pcal(\tilde{\mathcal{Y}})^{\infty}$, and let $\pi = \Phi_{\eta}(\bm{\nu})$ and $\hat{\pi}= \Phi^{\mathrm{reg}}_{\eta}(\hat{\bm{\nu}})$. Then by \Cref{cor:lip_best_response} and \Cref{prop:lip_measure_flow},
\begin{align*}
    \delta_{\infty}(\Psi^{\mathrm{aug}}(\pi), \Psi^{\mathrm{aug}}(\hat{\pi})) & \leq L_{\Psi} \delta_{\mathcal{A}}(\Phi_{\eta}(\bm{\nu}), \Phi_{\eta}(\hat{\bm{\nu}})) \\
    & \leq L_{\Psi}  K_{\eta} \delta_{\infty}(\bm{\nu}, \hat{\bm{\nu}}),
\end{align*}
where we recall from \Cref{cor:lip_best_response} that $K_{\eta} = \frac{2q^{*} l_{\eta^{*}}}{\eta^2}$. Since $\eta > \sqrt{2q^{*} l_{\eta^{*}} L_{\Psi
}}$, we have that $ L_{\Psi}  K_{\eta} < 1$, and hence the map $\Psi^{\mathrm{aug}} \circ \Phi_{\eta}$ is a contraction. As $(\Pcal(\tilde{\mathcal{Y}})^{\infty}, \delta_{\infty})$ is a complete metric space, by Banach's fixed-point theorem, there exists a unique fixed point of $\Psi^{\mathrm{aug}} \circ \Phi_{\eta}$, which by definition is also an MFNE for the regularised MFG-MCDM.  
\end{proof}

\begin{rem}

Mean field games with (relative entropy) regularisation are relevant both as approximations to non-regularised MFGs (see \cite{cui2021approximately}) and also as MFGs of research interest in their own right (e.g., \cite{saldi_regularisation}).
Introducing regularisation clearly alters the MFNE: following \cite{cui2021approximately}, we show subsequently in \Cref{thm:epsilon_nash} that the regularisation parameter $\eta$ is required to be sufficiently small to approximate the Nash equilibrium for a non-regularised game with finitely many, but large number of players. In contrast, \Cref{thm:contraction_reg_mcdm} requires $\eta$ to be sufficiently large to guarantee the uniqueness of a regularised MFNE. Despite a gap left between such assumptions in theory, it was possible in our computational experiments to find $\eta>0$ small such as to observe both convergence of iterates as well as a good approximation to the MFNE of the non-regularised MFG.
\end{rem}

\begin{rem}
Although MFNE exist for any $\eta>0$, as shown in \Cref{thm:existence}, they do not have to be unique for weakly (or even un-)regularized MFGs. In such circumstances, a Banach fixed point contraction is clearly impossible. Thus, to calculate non-unique and (all) multiple MFNE becomes a question of interest, which appears to be less studied in the literature; but note recent progress by \cite{guo2023general}.
\end{rem}

\section{Approximate Nash equilibria to the $N$-player game}\label{sec:approximate}

In this section, we show that the regularised MFNE as defined in \Cref{defn:regularised MFNE-MCDM} forms an approximate Nash equilibrium of the corresponding finite player game for a sufficiently large number of players and a sufficiently low amount of regularisation.\footnote{One cannot infer from this the computability of the equilibria. Indeed, for small $\eta$, there is no guarantee of a contractive fixed point operator, and the MFNE need not be unique.} Recall from \eqref{regularised_Q_func} that, for $\eta >0$ and $\bm{\nu} \in \Pcal(\tilde{\mathcal{Y}})^{\infty}$, $Q^{*}_{\eta, \bm{\nu}}(t,y,u)$ represents the optimal value conditioned on $y_t = y$ and $u_t =u$. For the unregularised case, i.e. for $\eta =0$, we write $Q^{*}_{\bm{\nu}} \coloneqq Q^{*}_{0, \bm{\nu}}$. For a policy $\pi\in \mathcal{A}_{\rm DM}$, we will also consider its associated $Q$-function, defined by
\begin{align*}
        Q^{\pi}_{\bm{\nu}}(s,y,u) \coloneqq  \E^{\pi}_{s,y,u}\left[ \sum^{\infty}_{t=s} \gamma^{t-s} r_y\big(y_t, u_t,  \bm{\mu}^{\nu}_t\big) \right], 
\end{align*}
where the expectation $\E^{\pi}_{s,y,u}$ is under which $\pi$ is the policy, with the initial conditions $y_s = y$ and $u_s = u$ at time $s$. For $s = 0$, we write $Q^{\pi}_{\bm{\nu}}(y,u) \coloneqq Q^{\pi}_{\bm{\nu}}(0,y,u)$. We first show that the $Q$-function associated with the softmax policy converges uniformly towards $Q^{*}_{\bm{\nu}}$ as the regularisation parameter $\eta$ decreases towards 0. In order to do so, we will use the lemma below, which shows that the following integral functional of a continuous function $f$ with respect to a measure $q \in \Pcal(U)$ with full support, converges to the supremum of $f$ over $U$ as $\eta$ tends to 0. This is due to Laplace's method \cite[p. 56]{wong2001asymptotic}, and can be seen as a special case of Varadhan's lemma \cite[Theorem 4.3.1]{dembo2009large}.

\begin{lemma}\label{varadhan_lemma}
Let $U$ be a compact Polish space, let $f:U \to \R$ be a continuous function and let $q \in \Pcal(U)$ have full support on $U$. Then
    \begin{align*}
        \lim_{\eta \to 0^{+}} \eta \log \left( \int_U \exp\left(\frac{f(u)}{\eta}\right) q(\de u) \right) = \sup_{u \in U} f(u) = \max_{u \in U} f(u).
    \end{align*}
\end{lemma}

Recall from \Cref{defn:regularised MFNE-MCDM} that $\Phi_{\eta}(\bm{\nu})$ is the softmax policy with respect to $Q^{*}_{\eta,\bm{\nu}}$. Next we show that the function $\bm{\nu} \mapsto Q^{\Phi_{\eta}(\bm{\nu})}_{\bm{\nu}}$ converges uniformly towards $\bm{\nu} \mapsto Q^{*}_{\bm{\nu}}$. The main idea follows the proof of \cite[Theorem 4]{cui2021approximately}, which is in the setting for the case $\mathcal{X}$ and $A$ being finite. We prove that the result also holds here where $\mathcal{X}$ and $A$ are compact Polish spaces, by showing that the measure induced by the softmax policy concentrates around optimal points as $\eta$ decreases towards 0.

\begin{prop}\label{lemma_uniform_convergence}
       The function $\bm{\nu} \mapsto Q^{\Phi_{\eta}(\bm{\nu})}_{\bm{\nu}}$ converges to $\bm{\nu} \mapsto Q^{*}_{\bm{\nu}}$ as $\eta \downarrow 0$, uniformly over all $y \in \tilde{\mathcal{Y}}$ and $u \in U$.
    \end{prop}
\begin{proof}
For $T>0$, we denote by $\QTstar_{\eta, \bm{\nu}}$ and $Q^{T, \pi}_{\eta, \bm{\nu}}$ the finite horizon counterparts of $Q^{*}_{\eta, \bm{\nu}}$ and $ Q^{\pi}_{\bm{\nu}}$ respectively, where the rewards after time $T$ are simply zero. We split the proof into the 5 steps below.
\begin{enumerate}
    \item For any $\bm{\nu} \in \Pcal(\tilde{\mathcal{Y}})^{\infty}$, $y \in \tilde{\mathcal{Y}}$, $u \in U$ and $t, T \in \N$ with $t \leq T$, we have $\lim_{\eta \to 0^{+}}Q^{T,*}_{\eta, \bm{\nu}}(t, y,u)=Q^{T,*}_{\bm{\nu}}(t,y,u)$. \vspace{0.5em}\\
    \textit{Proof of claim 1.} We prove the claim by induction. Fix $\bm{\nu} \in \Pcal(\tilde{\mathcal{Y}})^{\infty}$, $y \in \tilde{\mathcal{Y}}$, $u \in U$ and $t, T \in \N$ with $t \leq T$. The claim trivially holds for $t = T$. Now assume that for some $t \leq T -1$, for any $\varepsilon > 0$, there exists $\eta_0 > 0 $ such that for all $ \eta < \eta_0$,
    \begin{align*}
       \lvert Q^{T,*}_{\eta, \bm{\nu}}(t+1, y,u) - Q^{T,*}_{\bm{\nu}}(t+1,y,u)\rvert < \frac{\varepsilon}{2}.
    \end{align*}
    By \Cref{varadhan_lemma}, we can also find $\eta_1 > 0 $ such that for all $ \eta < \eta_1$,
    \begin{align*}
        \left\lvert \eta  \log \left( \int_U \exp \left(\frac{\QTstar_{\bm{\nu}}(t+1, y, \pr{u})}{\eta} \right) q(\de \pr{u})\right)  - \sup_{u\in U} \QTstar_{\bm{\nu}}(t+1, y, u) \right\rvert < \frac{\varepsilon}{2}.
    \end{align*}
    Therefore, from the dynamic programming characterisation in \Cref{prop:dpp_truncated_q}, we have for all $ \eta < \min \{\eta_0, \eta_1\}$,
    \begin{align*}
        Q^{T,*}_{\eta, \bm{\nu}}(t, y,u)
        &\leq   r_y( y, u, \bm{\mu}^{\nu}_t) + \gamma \eta \int_{\tilde{\mathcal{Y}}} \sup_{u\in U} \QTstar_{\bm{\nu}}(t+1, \pr{y}, u)  p_y(\de \pr{y} \mid y, u,  \bm{\mu}^{\nu}_t) + \varepsilon \\
        & = \QTstar_{\bm{\nu}}(t,y,u)+ \varepsilon,  \\
        Q^{T,*}_{\eta, \bm{\nu}}(t, y,u)
        &\geq  r_y( y, u, \bm{\mu}^{\nu}_t) + \gamma \eta \int_{\tilde{\mathcal{Y}}}\sup_{u\in U} \QTstar_{\bm{\nu}}(t+1, \pr{y}, u) p_y(\de \pr{y} \mid y, u,  \bm{\mu}^{\nu}_t) - \varepsilon\\
        & = \QTstar_{\bm{\nu}}(t,y,u)- \varepsilon,
    \end{align*}
    which proves claim 1.
    \item For any $\bm{\nu} \in \Pcal(\tilde{\mathcal{Y}})^{\infty}$, $y \in \tilde{\mathcal{Y}}$ and $u \in U$, we have $\lim_{\eta \to 0^{+}}Q^{*}_{\eta, \bm{\nu}}(y,u)=Q^{*}_{\bm{\nu}}(y,u)$. \vspace{0.5em}\\
    \textit{Proof of claim 2.} Fix $\bm{\nu} \in \Pcal(\tilde{\mathcal{Y}})^{\infty}$, $y \in \tilde{\mathcal{Y}}$ and $u \in U$. Let $\varepsilon > 0$ and choose $T \in \N$ such that 
    $\lvert Q^{*}_{\bm{\nu}}(y,u) - \QTstar_{\bm{\nu}}(y,u) \rvert < \varepsilon/3$ and $  \gamma^T M_R/(1-\gamma) < \varepsilon/6$. We then have 
    \begin{align*}
        \lvert Q^{*}_{\eta, \bm{\nu}} (y,u) -  \QTstar_{\eta, \bm{\nu}}(y,u) \rvert & \leq \gamma \int_{\tilde{\mathcal{Y}}} \lvert J^{*}_{\eta, \bm{\nu}}(T+1, \pr{y}) - \JTstar_{\eta, \bm{\nu}}(T+1, \pr{y}) \rvert\ p_y(\de \pr{y} \mid y, u, \bm{\mu}^{\nu}_t) \\
        & \leq \gamma \sup_{\bar{y} \in \mathcal{\tilde{Y}}}  \lvert J^{*}_{\eta, \bm{\nu}}(T+1, \bar{y}) - \JTstar_{\eta, \bm{\nu}}(T+1, \bar{y}) \rvert\\
        &\leq \gamma \sup_{\bar{y} \in \mathcal{\tilde{Y}}} \left\lvert\sup_{\pi \in \mathcal{A}_{\rm DM}}\E^{\pi}_{T+1,\bar{y}}\left[ \sum^{\infty}_{t = T+1} \gamma^{t} (r_y(y_t, u_t,  \bm{\mu}^{\nu}_t) - \eta D_{KL}(\pi_t \Vert q) )\right]  \right\rvert .
    \end{align*}
    By \Cref{thm_dpp_regularise}, the softmax policy $\pi^{\rm soft}$ attains the supremum over $\mathcal{A}_{\rm DM}$ above, and has a continuous density with respect to $q$. In addition, since $U$ is compact, there exists $M_D > 0$ such that for all $s \geq T+1$, $\lvert D_{KL}(\pi^{\rm soft}_s \Vert q)\rvert < M_D$, so that
    \begin{align*}
          \lvert Q^{*}_{\eta, \bm{\nu}} (y,u) -  \QTstar_{\eta, \bm{\nu}}(y,u) \rvert  \leq \sum^{\infty}_{t = T} \gamma^t (M_R - \eta M_D) < \frac{\varepsilon}{6} - \frac{\eta \gamma^T M_D}{1-\gamma}.
    \end{align*}
    From step 1, we can find $\eta_2 > 0 $ such that $\lvert \QTstar_{\eta, \bm{\nu}}(y,u) - \QTstar_{\bm{\nu}}(y,u) \rvert < \varepsilon/3$ and $\eta \gamma^T M_D/(1-\gamma) < \varepsilon/6$ for all $\eta < \eta_2$. Therefore, for all $\eta < \eta_2$,
    \begin{align*}
        &\lvert Q^{*}_{\eta, \bm{\nu}} (y,u) - Q^{*}_{\bm{\nu}}(y,u) \rvert \nonumber \\
        \leq\ & \lvert Q^{*}_{\eta, \bm{\nu}} (y,u) - \QTstar_{\eta, \bm{\nu}} (y,u)\rvert + \lvert \QTstar_{\eta, \bm{\nu}}(y,u) - \QTstar_{\bm{\nu}}(y,u) \rvert + \lvert \QTstar_{\bm{\nu}}(y,u) - Q^{*}_{\bm{\nu}}(y,u) \rvert < \varepsilon
    \end{align*}
    as required.
    
    \item For any $y \in \tilde{\mathcal{Y}}$ and $u \in U$, the function $\bm{\nu} \mapsto Q^{*}_{\eta, \bm{\nu}}(y,u)$ converges uniformly on $\bm{\nu} \in \Pcal(\tilde{\mathcal{Y}})^{\infty}$ to ${\bm{\nu} \mapsto Q^{*}_{\bm{\nu}}(y,u)}$ as $\eta \downarrow 0$. \vspace{0.5em}\\
    \textit{Proof of claim 3.} By \Cref{lem:qregstar_lip}, the map $\bm{\nu} \mapsto Q^{*}_{\eta, \bm{\nu}}$ is continuous with respect to the total variation metric, and hence is also weakly continuous. Moreover, the sequence of functions $\bm{\nu} \mapsto Q^{*}_{\eta, \bm{\nu}}(y,u)$ are pointwise decreasing in $\eta$, and by step 2 converges pointwise to $\bm{\nu} \mapsto Q^{*}_{\bm{\nu}}(y,u)$ as $\eta \downarrow 0$. Therefore, by Dini's theorem, we conclude that for each $y \in \tilde{\mathcal{Y}}$ and $u \in U$, $\bm{\nu} \mapsto Q^{*}_{\eta, \bm{\nu}}(y,u)$ converges uniformly to $\bm{\nu} \mapsto Q^{*}_{\bm{\nu}}(y,u)$ as $\eta \to 0^+$. 
    
    \item For any $\bm{\nu} \in \Pcal(\tilde{\mathcal{Y}})^{\infty}$, we have that for any $y \in \tilde{\mathcal{Y}}$ and $u \in U$, $\lim_{\eta \to 0^+} Q^{\Phi_{\eta}(\bm{\nu})}_{\bm{\nu}}(y,u)= Q^{*}_{\bm{\nu}}(y,u)$.\vspace{0.5em}\\
    \textit{Proof of claim 4.} We show that $\lim_{\eta \to 0^+} Q^{T,\Phi_{\eta}(\bm{\nu})}_{\bm{\nu}}(y,u)= \QTstar_{\bm{\nu}}(y,u)$ for any $T<\infty$ with $T \in \N$, the infinite horizon case then follows from the same argument as in step 2. We proceed to prove via induction. Firstly, we note that $\lim_{\eta \to 0^+} Q^{T,\Phi_{\eta}(\bm{\nu})}_{\bm{\nu}}(T, y,u)= \QTstar_{\bm{\nu}}(T,y,u)$ holds trivially. Now let $\varepsilon> 0$ and $t \leq T+1$. Since $\tilde{\mathcal{Y}}$ and $U$ are compact, $Q^{T,\Phi_{\eta}(\bm{\nu})}_{\bm{\nu}}$ and $\QTstar_{\bm{\nu}}$ are locally uniformly continuous on $\tilde{\mathcal{Y}} \times U$. Therefore, for the induction assumption, we can assume that for each $y \in \tilde{\mathcal{Y}}$ and $u \in U$, there exists a neighbourhood $U_{y,u}$ of $(y,u)$ and some $\eta_3(y,u) > 0$ such that for all $\eta < \eta_3(y,u)$ and for all $(y,u)$ in $U_{y,u}$, we have $ \lvert Q^{T,\Phi_{\eta}(\bm{\nu})}_{\bm{\nu}}(t+1,y,u) - \QTstar_{\bm{\nu}}(t + 1,y,u) \rvert < \varepsilon$. By compactness there exists a finite subcover $\{U_{y_k,u_k}\}_{1 \leq k \leq K}$ of $\tilde{\mathcal{Y}} \times U$, and by letting $\eta^{*}_3 = \min\{ \eta_3(y_k, u_k): 1 \leq k \leq K\}$, we have that  $ \lvert Q^{T,\Phi_{\eta}(\bm{\nu})}_{\bm{\nu}}(t+1,y,u) - \QTstar_{\bm{\nu}}(t + 1,y,u) \rvert < \varepsilon$ for all $y \in \tilde{\mathcal{Y}}$ and $u \in U$. \\
    
    Next, we proceed to bound from above the probability of making a suboptimal action. By applying the same argument as in step 3 for the finite horizon case, we can find $\eta_4 > 0$ such that for all $\eta < \min\{ \eta^*_3, \eta_4 \}$,
        \begin{align*}
        \lvert \QTstar_{\bm{\nu}}(t,y,u) -\QTstar_{\eta,\bm{\nu}}(t,y,u) \rvert \leq \frac{\varepsilon}{4}.
    \end{align*}
     Let $U^{*}_y \coloneqq \{ \pr{u} \in U: \QTstar_{\bm{\nu}}(t, y, \pr{u}) = \sup_{u \in U} \QTstar_{\bm{\nu}}(t, y, u) \}$, which is non-empty since $U$ is compact and $\QTstar_{\bm{\nu}}$ is continuous. Let ${L_{\varepsilon}(U^{*}_y) \coloneqq \{ \pr{u} \in U: \QTstar_{\bm{\nu}}(t, y, \pr{u}) \geq \sup_{u \in U} \QTstar_{\bm{\nu}}(t, y, u) - \varepsilon \}}$, and let $u^{*} \in U^{*}_y$. Choose $\pr{\varepsilon} \in (0, \varepsilon/2)$. Then, by letting $\Delta_u\coloneqq \QTstar_{\eta,\bm{\nu}}(t,y,u) - \QTstar_{\bm{\nu}}(t,y,u)$ and $\pr{\Delta}_u \coloneqq \QTstar_{\bm{\nu}}(t,y,u^{*}) - \QTstar_{\bm{\nu}}(t,y,u) $, we have
    \begin{align*}
        \Phi_{\eta}(\bm{\nu})_t (U \setminus L_{\varepsilon}(U^{*}_y) \mid y) &= \frac{\int_{U \setminus L_{\varepsilon}(U^{*}_y)} \exp(\QTstar_{\eta,\bm{\nu}}(t,y,u)/\eta) q(\de u)}{\int_U \exp(\QTstar_{\eta,\bm{\nu}}(t,y,u)/\eta) q(\de u)} \\
        &\leq \frac{\int_{U \setminus L_{\varepsilon}(U^{*}_y)} \exp(\QTstar_{\eta,\bm{\nu}}(t,y,u)/\eta) q(\de u)}{\int_{L_{\pr{\varepsilon}}(U^{*}_y)} \exp(\QTstar_{\eta,\bm{\nu}}(t,y,u)/\eta) q(\de u)} \\
        & = \frac{\int_{U \setminus L_{\varepsilon}(U^{*}_y)} \exp((\Delta_u - \pr{\Delta}_u)/\eta) q(\de u)}{\int_{L_{\pr{\varepsilon}}(U^{*}_y)} \exp((\Delta_u - \pr{\Delta}_u) / \eta) q(\de u)},
    \end{align*}
    where the last equality is obtained by multiplying both the numerator and denominator by $\exp(-\QTstar_{\bm{\nu}}(t,y,u^{*})/\eta)$.  Now for any $u \notin L_{\varepsilon}(U^*_y)$, we have $\pr{\Delta}_u > \varepsilon$, and for any $u \in L_{\pr{\varepsilon}}(U^*_y)$, we have $\pr{\Delta}_u \leq \pr{\varepsilon}$. Therefore,
    \begin{align}\label{estimate_softmax}
        \Phi_{\eta}(\bm{\nu})_t (U \setminus L_{\varepsilon}(U^{*}_y) \mid y) &\leq \frac{q(U \setminus L_{\varepsilon}(U^{*}_y)) \exp((\varepsilon/4 -\varepsilon)/\eta) }{q( L_{\pr{\varepsilon}}(U^{*}_y)) \exp((-\varepsilon/4 - \pr{\varepsilon}) /\eta)} \nonumber \\
        &= \frac{q(U \setminus L_{\varepsilon}(U^{*}_y)) }{q( L_{\pr{\varepsilon}}(U^{*}_y))} \exp\left(\frac{- \varepsilon/2 +\pr{\varepsilon}}{\eta}\right)
    \end{align}
    
    By a similar argument as in \Cref{lemma:q_bound}, $Q^{T,*}_{\bm{\nu}}$ and $ Q^{T,\Phi_{\eta}(\bm{\nu})}_{\bm{\nu}}$ are bounded by some constant $M_Q>0$,
    \begin{align*}
         &\lvert Q^{T,\Phi_{\eta}(\bm{\nu})}_{\bm{\nu}}(t.y,u) - \QTstar_{\bm{\nu}}(t,y,u) \rvert\\
         =\  & \left \lvert \int_{\tilde{\mathcal{Y}}}  \left(\int_U  Q^{T,\Phi_{\eta}(\bm{\nu})}_{\bm{\nu}} (t+1, \pr{y}, \pr{u})\Phi_{\eta}(\de \pr{u} \mid \pr{y})_t - \sup_{\pr{u} \in U} \QTstar_{\bm{\nu}}(t+1, \pr{y}, \pr{u}) \right) p_y (\de \pr{y} \mid y, u , \bm{\mu}^{\nu}_t)\right \rvert\\
         \leq\ &\sup_{\bar{y}\in\tilde{\mathcal{Y}}} \left\lvert \left(\int_{L_{\varepsilon}(U^{*}_{\bar{y}})} + \int_{U \setminus L_{\varepsilon}(U^{*}_{\bar{y}})}\right) Q^{T,\Phi_{\eta}(\bm{\nu})}_{\bm{\nu}} (t+1, \bar{y}, \pr{u})\Phi_{\eta}(\de \pr{u} \mid \bar{y})_t - \sup_{\pr{u} \in U} \QTstar_{\bm{\nu}}(t+1, \bar{y}, \pr{u}) \right\rvert \\
         \leq\ & I_1 + I_2 + I_3,
    \end{align*}
    where
    \begin{align*}
        I_1 \coloneqq & \sup_{\bar{y} \in \tilde{\mathcal{Y}}} \left \lvert \int_{L_{\varepsilon}(U^*_{\bar{y}})} \left( Q^{T,\Phi_{\eta}(\bm{\nu})}_{\bm{\nu}} (t+1, \bar{y}, \pr{u}) - \sup_{\bar{u} \in U} \QTstar_{\bm{\nu}} (t+1, \bar{y}, \bar{u}) \right) \Phi_{\eta}(\bm{\nu})_t (\de \pr{u} \mid \bar{y}) \right \rvert\\
        <&\ \varepsilon \\
        I_2  \coloneqq & \sup_{\bar{y}\in\tilde{\mathcal{Y}}} \left \lvert \int_{L_{\varepsilon}(U^{*}_{\bar{y}})} \sup_{\bar{u} \in U} \QTstar_{\bm{\nu}}(t+1, \bar{y}, \bar{u}) \Phi_{\eta}(\bm{\nu})_t(\de \pr{u} \mid \bar{y})- \sup_{\pr{u} \in U} \QTstar_{\bm{\nu}}(t+1, \bar{y}, \pr{u})\  \right \rvert  \\
        = &\ \Phi_{\eta}(\bm{\nu})_t(U\setminus L_{\varepsilon}(U^{*}_{\bar{y}}) \mid \bar{y}) \ \sup_{\bar{y}\in\tilde{\mathcal{Y}}}  \sup_{\bar{u} \in U} \QTstar_{\bm{\nu}}(t+1, \bar{y}, \bar{u})  \\
        < &\ M_Q \frac{q(U \setminus L_{\varepsilon}(U^{*}_{\bar{y}})) }{q(L_{\pr{\varepsilon}}(U^{*}_{\bar{y}}))} \exp\left(\frac{- \varepsilon/2 +\pr{\varepsilon}}{\eta}\right)\\
        I_3  \coloneqq &\sup_{\bar{y}\in\tilde{\mathcal{Y}}} \left \lvert \int_{U\setminus L_{\varepsilon}(U^{*}_{\bar{y}})} Q^{T,\Phi_{\eta}(\bm{\nu})}_{\bm{\nu}} (t+1, \bar{y}, \pr{u})\  \Phi_{\eta}(\de \pr{u} \mid \bar{y})  \right \rvert \\
        < &\ M_Q \frac{q(U \setminus L_{\varepsilon}(U^{*}_{\bar{y}})) }{q(L_{\pr{\varepsilon}}(U^{*}_{\bar{y}}))} \exp\left(\frac{- \varepsilon/2 +\pr{\varepsilon}}{\eta}\right),
    \end{align*}
    noting that the bound for $I_1$ follows from the induction assumption.  The claim then holds, since for any $\varepsilon_0 > 0$, one can always choose $\varepsilon > 0 $ , $\pr{\varepsilon} \in (0, \varepsilon/2)$ and $\eta < \min\{\eta^*_3, \eta_4 \}$ such that $I_1 + I_2 +I_3 < \varepsilon_0$.
    
    \item The function $(y,u,\bm{\nu}) \mapsto Q^{\Phi_{\eta}(\bm{\nu})}_{\bm{\nu}}(y,u)$ converges uniformly to $(y,u,\bm{\nu}) \mapsto Q^{*}_{\bm{\nu}}(y,u)$ as $\eta \downarrow 0$. \vspace{0.5em}\\
    \textit{Proof of claim 5.} As in step 4, we first treat the finite horizon case.  
    We show that there exists $\bar{\eta}$ such that the family of functions
    \begin{align*}
        \left\{\bm{\nu} \mapsto Q^{T,\Phi_{\eta}(\bm{\nu})}_{\bm{\nu}} (t,y,u)\right\}_{t\in [0,T], y \in \tilde{\mathcal{Y}}, u \in U, \eta < \bar{\eta}}
    \end{align*}
    are equicontinuous over $\bm{\nu} \in \Pcal(\tilde{\mathcal{Y}})^{\infty}$. In order to do so, we show that given $\varepsilon>0$, for each $(t,y,u)$, there exists $\delta_{t,y,u} >0$, such that for any $\hat{\bm{\nu}} \in \Pcal(\tilde{\mathcal{Y}})^{\infty}$ with $\delta_{\infty}(\bm{\nu}, \hat{\bm{\nu}}) < \delta_{t,y,u}$ and for all $\eta < \bar{\eta}$, we have
    \begin{align}\label{eq:equi_q}
        \left\lvert  Q^{T,\Phi_{\eta}(\bm{\nu})}_{\bm{\nu}}(
        t,y,u) -  Q^{T,\Phi_{\eta}(\bm{\hat{\nu}})}_{\bm{\hat{\nu}}}(
        t,y,u)\right\rvert < \varepsilon.
    \end{align}
    Then, by a compactness argument as in step 4, we can obtain a $\delta_{\rm min} > 0$ such that for any $\hat{\bm{\nu}}$ satisfying $\delta_{\infty}(\bm{\nu}, \hat{\bm{\nu}}) < \delta_{\rm \min}$ and for all $\eta < \bar{\eta}$, we have
    \begin{align*}
        \sup_{t,y,u}\lvert  Q^{T,\Phi_{\eta}(\bm{\nu})}_{\bm{\nu}}(t,y,u) -  Q^{T,\Phi_{\eta}(\bm{\hat{\nu}})}_{\bm{\hat{\nu}}}(t,y,u)\rvert < \varepsilon.
    \end{align*}
    We now proceed to establish \eqref{eq:equi_q} with an induction argument over $t \in [0,T]$. Starting with the base case $t = T$, since $r_y$ is Lipschitz, we can take $\delta_{T,y,u} = \varepsilon/L_R L_M$, so that by \Cref{prop:lip_cont} and \Cref{lem:nu_to_mu}, 
     \begin{align*}
        \left\lvert  Q^{T,\Phi_{\eta}(\bm{\nu})}_{\bm{\nu}}(
        T,y,u) -  Q^{T,\Phi_{\eta}(\bm{\hat{\nu}})}_{\bm{\hat{\nu}}}(
        T,y,u)\right\rvert < \lvert r_y(y,u,\bm{\mu}^{\nu}_t) - r_y(y,u,\bm{\hat{\mu}}^{\nu}_t) \rvert < L_RL_M\ \delta_{\infty}(\bm{\nu}, \hat{\bm{\nu}})<  \varepsilon.
    \end{align*}
    Now assume for the induction hypothesis that for some $t+1 \leq T$, there exists $\delta_{t+1}$ such that for any $\hat{\bm{\nu}}$ satisfying $\delta_{\infty}(\bm{\nu}, \hat{\bm{\nu}}) < \delta_{t+1}$ and for all $\eta < \bar{\eta}$, we have
    \begin{align*}
        \sup_{t,y,u}\lvert  Q^{T,\Phi_{\eta}(\bm{\nu})}_{\bm{\nu}}(t+1,y,u) -  Q^{T,\Phi_{\eta}(\bm{\hat{\nu}})}_{\bm{\hat{\nu}}}(t+1,y,u)\rvert < \frac{\varepsilon}{6}.
    \end{align*} 
    Then, for the next index $t$,
    \begin{align}\label{step 5 eq}
        &\left\lvert  Q^{T,\Phi_{\eta}(\bm{\nu})}_{\bm{\nu}}(
        t,y,u) -  Q^{T,\Phi_{\eta}(\bm{\hat{\nu}})}_{\bm{\hat{\nu}}}(
        t,y,u)\right\rvert \nonumber \\
        \leq\ & \lvert r_y(y,u,\bm{\mu}^{\nu}_t) - r_y(y,u,\bm{\hat{\mu}}^{\nu}_t) \rvert \nonumber \\
        &+ \int_{\tilde{\mathcal{Y}}}  \int_{U} \lvert Q^{T,\Phi_{\eta}(\bm{\nu})}_{\bm{\nu}}(
        t+1,\pr{y},\pr{u})\rvert \Phi_{\eta}(\bm{\nu})_t( \de \pr{u} \mid \pr{y})\ \lvert p_y(\de \pr{y} \mid y, u, \bm{\mu}^{\nu}_t) -  p_y(\de \pr{y} \mid y, u, \bm{\hat{\mu}}^{\nu}_t)\rvert \nonumber\\
        &+  \sup_{\bar{y} \in \tilde{\mathcal{Y}}} \left \lvert \int_{U} Q^{T,\Phi_{\eta}(\bm{\nu})}_{\bm{\nu}}(
        t+1,\bar{y},\pr{u}) \Phi_{\eta}(\bm{\nu})_t( \de \pr{u} \mid \bar{y})-  \int_{U} Q^{T,\Phi_{\eta}(\bm{\hat{\nu}})}_{\bm{\hat{\nu}}}(
        t+1,\bar{y},\pr{u}) \Phi_{\eta}(\bm{\hat{\nu}})_t( \de \pr{u} \mid \bar{y}) \right \rvert
    \end{align}
    For the first term in \eqref{step 5 eq}, we have $\lvert r_y(y,u,\bm{\mu}^{\nu}_t) - r_y(y,u,\bm{\hat{\mu}}^{\nu}_t) \rvert < L_R L_M \ \delta_{\infty}(\bm{\nu}, \hat{\bm{\nu}})$. For the second term, we use \Cref{prop: maxmin_ineq} to get
    \begin{align*}
        &\int_{\tilde{\mathcal{Y}}}  \int_{U} \lvert Q^{T,\Phi_{\eta}(\bm{\nu})}_{\bm{\nu}}(
        t+1,\pr{y},\pr{u})\rvert \Phi_{\eta}(\bm{\nu})_t( \de \pr{u} \mid \pr{y})\ \lvert p_y(\de \pr{y} \mid y, u, \bm{\mu}^{\nu}_t) -  p_y(\de \pr{y} \mid y, u, \bm{\hat{\mu}}^{\nu}_t)\rvert \\
        \leq\ &2 M_Q\ \delta_{TV}(p_y(\cdot \mid y, u, \bm{\mu}^{\nu}_t), p_y(\cdot \mid y, u, \bm{\hat{\mu}}^{\nu}_t) )\\
        \leq\ & 2 M_Q L_P L_M\ \delta_{\infty}(\bm{\nu}, \bm{\hat{\nu}}).
    \end{align*}
    Before considering the third term in \eqref{step 5 eq}, we note that  $\bm{\nu} \mapsto \Phi_{\eta}(\bm{\nu})_t$ is continuous (with respect to the weak topology for the space of measures),
    and consider a continuous bounded function $f: U \to \R$ with a sequence of measures $(\bm{\nu}_j)_j$ converging to $\bm{\nu}$. Let $U^{*}_y (\bm{\nu}) \coloneqq \{ \pr{u} \in U: \QTstar_{\bm{\nu}}(t, y, \pr{u}) = \sup_{u \in U} \QTstar_{\bm{\nu}}(t, y, u) \}$, and similarly let $L_{\varepsilon}(U^{*}_y (\bm{\nu})) \coloneqq \{ \pr{u} \in U: \QTstar_{\bm{\nu}}(t, y, \pr{u}) \geq  \sup_{u \in U} \QTstar_{\bm{\nu}}(t, y, u) -\varepsilon\}$. Then
    \begin{align}\label{eq:weak_step5}
        &\left \lvert \int_{L_{\varepsilon}(U^*_{y}(\bm{\nu}_j))} f(u) \Phi_{\eta}(\bm{\nu}_j)_t( \de \pr{u} \mid y) -  \int_{L_{\varepsilon}(U^*_{y}(\bm{\nu}))} f(u) \Phi_{\eta}(\bm{\nu})_t( \de \pr{u} \mid y) \right \rvert \nonumber \\
       \leq\ & \left\lvert \left(\int_{L_{\varepsilon}(U^*_{y}(\bm{\nu}_j))} - \int_{L_{\varepsilon}(U^*_{y}(\bm{\hat{\nu}}))}\right) f(u) \Phi_{\eta}(\bm{\nu})_t( \de \pr{u} \mid y) \right\rvert \nonumber  \\
       & \quad + \left \lvert   \int_{L_{\varepsilon}(U^*_{y}(\bm{\nu}))}f(u) ( \Phi_{\eta}(\bm{\nu}_j)_t( \de \pr{u} \mid y)  - \Phi_{\eta}(\bm{\nu})_t( \de \pr{u} \mid y) )\right \rvert 
    \end{align}
    Now consider the symmetric difference
    \begin{align*}
        L_{\varepsilon}(U^*_{y}(\bm{\nu}_j)) \Delta L_{\varepsilon}(U^*_{y}(\bm{\nu})) \coloneqq \big(L_{\varepsilon}(U^*_y(\bm{\nu}_j)) \setminus L_{\varepsilon}(U^*_y(\bm{\nu}))\big)\ \cup\ \big(L_{\varepsilon}(U^*_y(\bm{\nu})) \setminus L_{\varepsilon}(U^*_y(\bm{\nu}_j)) \big).
    \end{align*}
    By the continuity of $\QTstar_{\bm{\nu}}$ with respect to $\bm{\nu}$, we have $d_H(L_{\varepsilon}(U^*_y(\bm{\nu_j})), L_{\varepsilon}(U^*_y(\bm{\nu})) \to 0$ as $j \to \infty$ (cf. \Cref{hausdorff}). Therefore, there exists $\varepsilon_0 \in (0, \varepsilon)$ such that
    \begin{align*}
         L_{\varepsilon}(U^*_{y}(\bm{\nu}_j)) \Delta L_{\varepsilon}(U^*_{y}(\bm{\nu})) \subset L_{\varepsilon + \varepsilon_0}(U^*_y(\bm{\nu}))  \setminus L_{\varepsilon - \varepsilon_0}(U^*_y(\bm{\nu}))  \eqqcolon A_{\varepsilon_0},
    \end{align*}
    so that returning to \eqref{eq:weak_step5}, we have 
    \begin{align*}
    \left \lvert \int_{L_{\varepsilon}(U^*_{y}(\bm{\nu}_j))} f(u) \Phi_{\eta}(\bm{\nu}_j)_t( \de \pr{u} \mid y) -  \int_{L_{\varepsilon}(U^*_{y}(\bm{\nu}))} f(u) \Phi_{\eta}(\bm{\nu})_t( \de \pr{u} \mid y) \right \rvert \leq \sup_{u \in U} f(u) q(A_{\varepsilon_0}).
    \end{align*}
    As $\varepsilon_0 \downarrow 0$, the set $A_{\varepsilon_0}$ decreases towards the empty set $\emptyset$. Therefore, sending $\varepsilon \downarrow 0$, we have $\sup_{u \in U} f(u) q(A_{\varepsilon_0}) \to 0$, with the convergence independent of $\eta$. Next, from an analogous argument to step 4 and \eqref{estimate_softmax}, we have
    \begin{align*}
        &\ \Phi_{\eta}(\bm{\nu_j})_t (U \setminus L_{\varepsilon}(U^{*}_y(\bm{\nu}_j)) - \Phi_{\eta}(\bm{\nu})_t (U \setminus L_{\varepsilon}(U^{*}_y(\bm{\nu})) \mid y) \\
        \leq & \frac{2q(U \setminus L_{\varepsilon}(U^{*}_y(\bm{\nu}))) }{q( L_{\pr{\varepsilon}}(U^{*}_y(\bm{\nu})))} \exp\left(\frac{- \varepsilon/2 +\pr{\varepsilon}}{\eta}\right)
    \end{align*}
    for some $\pr{\varepsilon} \in(0, \varepsilon)$. Therefore, for any $\eta \in (0, \bar{\eta})$, there exists a $\delta_{\bar{\eta}} > 0$, such that for any $\bm{\hat{\nu}}$ satisfying $\delta_{\infty}(\bm{\nu}, \bm{\hat{\nu}}) < \delta_{\bar{\eta}}$, we have 
    \begin{align*}
        \left \lvert \int_{U } f(u) \Phi_{\eta}(\bm{\nu})_t( \de \pr{u} \mid y) -  \int_{U} f(u) \Phi_{\eta}(\bm{\hat{\nu}})_t( \de \pr{u} \mid y) \right \rvert < \frac{\varepsilon}{6}.
    \end{align*}
    Returning to the third term of \eqref{step 5 eq}, we conclude that for any $\bm{\hat{\nu}}$ satisfying $\delta_{\infty}(\bm{\nu}, \bm{\hat{\nu}}) < \min\{\delta_{t+1},\delta_{\bar{\eta}}\}$, we have 
    \begin{align*}
         &\sup_{\bar{y} \in \tilde{\mathcal{Y}}} \left \lvert \int_{U} Q^{T,\Phi_{\eta}(\bm{\nu})}_{\bm{\nu}}(
        t+1,\bar{y},\pr{u}) \Phi_{\eta}(\bm{\nu})_t( \de \pr{u} \mid \bar{y})-  \int_{U} Q^{T,\Phi_{\eta}(\bm{\hat{\nu}})}_{\bm{\hat{\nu}}}(
        t+1,\bar{y},\pr{u}) \Phi_{\eta}(\bm{\hat{\nu}})_t( \de \pr{u} \mid \bar{y}) \right \rvert \\
        \leq\ & \sup_{\bar{y} \in \tilde{\mathcal{Y}}} \left \lvert \int_{U} Q^{T,\Phi_{\eta}(\bm{\nu})}_{\bm{\nu}}(
        t+1,\bar{y},\pr{u}) \Phi_{\eta}(\bm{\nu})_t( \de \pr{u} \mid \bar{y})-  \int_{U} Q^{T,\Phi_{\eta}(\bm{\nu})}_{\bm{
        \nu}}(t+1,\bar{y},\pr{u}) \Phi_{\eta}(\bm{\hat{\nu}})_t( \de \pr{u} \mid \bar{y}) \right \rvert \\
        & \ + \sup_{\bar{y} \in \tilde{\mathcal{Y}}} \left \lvert \int_{U} Q^{T,\Phi_{\eta}(\bm{\nu})}_{\bm{\nu}}(
        t+1,\bar{y},\pr{u}) \Phi_{\eta}(\bm{\hat{\nu}})_t( \de \pr{u} \mid \bar{y}) - \int_{U} Q^{T,\Phi_{\eta}(\bm{\hat{\nu}})}_{\bm{
        \hat{\nu}}}(t+1,\bar{y},\pr{u}) \Phi_{\eta}(\bm{\hat{\nu}})_t( \de \pr{u} \mid \bar{y}) \right \rvert\\
        <\ & \frac{\varepsilon}{3}
    \end{align*}
    Finally, we conclude the induction proof, by taking 
    \begin{align*}
        \delta_{t,y,u} =\min\{ \varepsilon/ (3L_R L_M ), \varepsilon/(6 M_Q L_P L_M), \delta_{t+1}, \delta_{\bar{\eta}}\}.
    \end{align*}
    Claim 5 for the finite horizon case then follows by the Arzelà–Ascoli theorem, as $\Pcal(\tilde{\mathcal{Y}})^{\infty}$ is compact by Tychonoff's theorem.    For the infinite horizon case, note that since the reward function $r_y$ is bounded, given $\varepsilon >0$, we can find for some large $T$ such that
    \begin{align*}
         \lvert Q^{\Phi_{\eta}(\bm{\nu}^{*})}_{\bm{\nu}}(y,u) - Q^{\Phi_{\eta}(\bm{\hat{\nu}}^{*})}_{\bm{\hat{\nu}}}(y,u) \rvert \
        \leq \  \lvert Q^{T,\Phi_{\eta}(\bm{\nu}^{*})}_{\bm{\nu}}(y,u) - Q^{T,\Phi_{\eta}(\bm{\hat{\nu}}^{*})}_{\bm{\hat{\nu}}}(y,u) \rvert + \frac{\varepsilon}{2}.
    \end{align*}
    Therefore, for sufficiently small $\eta$ and $\delta$, with $\delta_{\infty}(\bm{\nu}, \hat{\bm{\nu}}) < \delta$,
    \begin{align*}
        \lvert Q^{T,\Phi_{\eta}(\bm{\nu})}_{\bm{\nu}}(y,u) - Q^{T,\Phi_{\eta}(\bm{\hat{\nu}})}_{\bm{\hat{\nu}}}(y,u) \rvert < \frac{\varepsilon}{2},
    \end{align*}
     so that we conclude uniform convergence by the Arzel\`a--Ascoli theorem.
\end{enumerate}
\end{proof}

This leads to the following approximate Nash equilibria for a sequence of regularised MFNE.
\begin{thm}\label{thm:epsilon_nash}
    Let $(\eta_j)_j$ be a sequence with $\eta_j \downarrow 0$. For each $j$, let $(\pi^{(j)}_*,  \bm{\nu}^{(j)}_*) \in \mathcal{A}_{\rm DM} \times \Pcal(\mathcal{Y})^{\infty}$ be an associated regularised mean field Nash equilibrium (MFNE), defined as in \Cref{defn:regularised MFNE-MCDM}, for the MCDM-MFG. Then for any $\varepsilon > 0$, there exists $\pr{j}, \pr{N} \in \N$ such that for all $j\geq\pr{j}$ and $N \geq \pr{N}$, the policy $(\pi^{(j)}_*, \ldots, \pi^{(j)}_*)$ is $\varepsilon$-Nash for the $N$-player game. That is,
    \begin{align*}
        J^N_n(\pi^{(j)}_*, \ldots, \pi^{(j)}_*) \geq \sup_{\pi \in \mathcal{A}_{\rm DM}} J^N_n(\pi, \pi^{(j), -n}_*) - \varepsilon \quad \mbox{for all $n \in \{1, \ldots, N\}$},
    \end{align*}
    where $\pi^{(j), -n}_*$ represents the policy $\pi^{(j)}_*$ being applied to all players except player $n$.
\end{thm}
    \begin{proof}
    By \Cref{lemma_uniform_convergence} we have the uniform convergence of $\bm{\nu} \mapsto Q^{\Phi_{\eta_j}(\bm{\nu})}_{\bm{\nu}}$ to $\bm{\nu} \mapsto Q^{*}_{\bm{\nu}}$, therefore the regularised policy is approximately optimal for the MFG: for any $\varepsilon >0$, there exists $\pr{j} \in \N$ such that for all $j \geq \pr{j}$,
    \begin{align*}
        J_{\bm{\nu}^{(j)}_*}(\pi^{(j)}_*) \geq \sup_{\pi \in \mathcal{A}_{\rm DM}} J_{\bm{\nu}^{(j)}_*} (\pi) - \frac{\varepsilon}{2}.
    \end{align*}
    Let $(\tau^N)_{N\in \N}$ be a sequence of policies such that for each $N$,
    \begin{align*}
        \sup_{\pi \in \mathcal{A}_{\rm DM}} J^N_1(\pi, \pi^{(j)}_*, \ldots, \pi^{(j)}_*) - \frac{\varepsilon}{6} \leq J^N_1( \tau^N, \pi^{(j)}_*, \ldots, \pi^{(j)}_*).
    \end{align*}
    By \cite[Theorem 4.10]{saldi_mfg_discrete}, there exists $\pr{N} \in \N$ such that for all $N > \pr{N}$,
    \begin{align*}
        \lvert J^N_1(\tau^N, \pi^{(j)}_*, \ldots, \pi^{(j)}_*) - J_{\bm{\nu}^{(j)}_*}(\tau^N) \rvert &< \frac{\varepsilon}{6},\\
        \lvert J^N_1(\pi^{(j)}_*, \ldots, \pi^{(j)}_*) - J_{\bm{\nu}^{(j)}_*}(\pi^{(j)}_*) \rvert &< \frac{\varepsilon}{6}.
    \end{align*}
    Combining the above we have
    \begin{align*}
        \sup_{\pi \in \mathcal{A}_{\rm DM}} J^N_1(\pi, \pi^{(j)}_*, \ldots, \pi^{(j)}_*) - \frac{\varepsilon}{6} \leq \sup_{\pi \in \mathcal{A}_{\rm DM}} J_{\bm{\nu}^{(j)}_*}(\pi)  + \frac{\varepsilon}{6}
    \end{align*}
    so that
    \begin{align*}
        \sup_{\pi \in \mathcal{A}_{\rm DM}} J^N_1(\pi, \pi^{(j)}_*, \ldots, \pi^{(j)}_*) - \varepsilon &\leq \sup_{\pi \in \mathcal{A}_{\rm DM}} J_{\bm{\nu}^{(j)}_*}(\pi)  - \frac{2\varepsilon}{3} \\
        &\leq J_{\bm{\nu}^{(j)}_*}(\pi^{(j)}_*) - \frac{\varepsilon}{6}\\
        & \leq J^N_1(\pi^{(j)}_*, \ldots, \pi^{(j)}_*)
    \end{align*}
    as required.
\end{proof}

In order to obtain an approximate Nash equilibrium, it would be desirable to take $\eta$ small and close to $0$. However, in order for the map $\Psi^{\rm aug} \circ \Phi_{\eta}$ to be a contraction, $\eta$ is required to be sufficiently large. Nonetheless, whilst adopting a version of the prior descent algorithm in \cite{cui2021approximately} to our computational example in \Cref{sec_numerics}, where previous iterates of computed policies are used as the reference measure for subsequent iterates, we observe that the algorithm generally converges well for small values of $\eta$.

\begin{rem}
The above $\varepsilon$-optimal policies for the finite player game, constructed from mean-field equilibria, only depend on the agents' observations of their own augmented state variables and on time, but not on their respective empirical distributions amongst all agents. This is due to the fact that the empirical distribution of the population in the augmented state variables is approximated by the time-varying but deterministic measure flow from the fixed point characterisation of the mean-field game equilibrium. Intuitively, the deterministic nature of this flow in the large population limit corresponds to a law of large numbers being applied to empirical measures (by propagation of chaos).

We want to emphasize, moreover, that the deterministic measure flow in the augmented state space also reveals the proportions in the large population, which are attained endogenously in mean-field equilibrium with respect to the chosen information access (i.e.\ actively controlled observation delay), and moreover the respective distributions of actions taken during the delay period by the respective subpopulations. This is illustrated in the computational example in \Cref{sec_numerics}. Analogous comments apply also to the approximate $\varepsilon$-Nash equilibria as described above.
%
\end{rem}

\section{Extension to interaction via controls}\label{sec:extended}

In this section we outline how to incorporate player interactions via the controls in the MFG-MCDM. Such models in the classical fully observable case are sometimes referred to as extended MFGs in the literature \cite{carmona2013mean, guo2019learning}. Under a similar framework to \Cref{sec:MFG,sec:regularise}, we obtain analogous results for a contraction in the fixed point interaction, under corresponding Lipschitz conditions. When interactions between players occur via the controls, the transition kernel and reward function will be dependent on a joint distribution on the state space and action space. 

\begin{defn}
    An MCDM with mean-field interaction in both  states and controls is described by a tuple $\langle \mathcal{X}, A, \mathcal{D}, \mathcal{C}, p, r \rangle$, where
    \begin{itemize}
        \item the \textit{state space} $\mathcal{X}$ and the \textit{action space} $A$ are a Polish spaces, equipped with their respective Borel sigma algebra;
        \item the set of \textit{delay values} is $\mathcal{D} = \{ d_0, \ldots, d_K\} \subset \mathbb Z_{+}$, with $ 0 \leq d_K < \ldots < d_0$;
        \item the set of \textit{cost values} is $\mathcal{C} = \{ c_0, c_1, \ldots, c_K\}\subset \R_+$, with $0 = c_0 < c_1 \ldots < c_K $;
        \item the \textit{transition kernel} is $p: \mathcal{X} \times A \times \Pcal(\mathcal{X} \times A) \to \Pcal(\mathcal{X})$;
        \item the \textit{reward function} $r: \mathcal{X} \times A \times \Pcal(\mathcal{X} \times A)\to [0, \infty)$ is bounded and measurable.
    \end{itemize}
\end{defn}

In order to define an MFNE, we require an analogous mapping to \Cref{defn:aug_to_underlying}, from $\Pcal(\tilde{\mathcal{Y}})$ to $\Pcal(\mathcal{X} \times A)^{d_0+1}$. In this instance, we consider the augmented space 
\begin{align*}
    \tilde{\mathcal{Y}} \coloneqq \bigcup_{d=d_K}^{d_0}\left( \{d\} \times \mathcal{X}^{d_0-d+1} \times A^{d_0} \right) = \bigsqcup^{d_0}_{d=d_K} \left( \mathcal{X}^{d_0-d+1} \times A^{d_0} \right) \eqqcolon \bigsqcup^{d_0}_{d=d_K}\tilde{\mathcal{Y}}_d..
\end{align*}
An element $\tilde{y}_n \in \tilde{\mathcal{Y}}$ is written as
\begin{align*}
    \tilde{y}_n = (d,\ &x_{n-d_0}, \ldots, x_{n-d},
     a_{n-d_0},\ldots, a_{n-d}, a_{n-d+1},\ldots, a_{n-1}).
\end{align*}

\begin{defn}
    Let $\bm{\nu} = (\nu_t)_t \in \Pcal(\tilde{\mathcal{Y}})^{\infty}$. For each $t \geq 0$, we write $\nu_t = \sum_d w^d_t \nu^d_t$, where $w^d_t \in [0,1]$ and $\nu^d_t \in \Pcal(\tilde{\mathcal{Y}}_d)$. For $d \leq \pr{d} \in \{0,\ldots, d_0\}$, let $\nu^{x_{t-\pr{d}}, a_{t-\pr{d}}}_t$ and $\nu^{x_{t-\pr{d}}, a_{t-\pr{d}}\vert d}_t$ be the marginals of $\nu_t$ and $\nu^{d}_t$ in the $(x_{t-\pr{d}},a_{t-\pr{d}})$--coordinates respectively. Starting with $\pr{d} = d_0$, we define $\mu_{t,d_0} = \nu^{x_{t-d_0}, a_{t-d_0}}_t \in \Pcal(\mathcal{X}\times A)$. Then, for each $0 \leq   \pr{d} < d_0$, we define
\begin{align*}
    \mu_{t,\pr{d}}& \coloneqq \sum^{d_0}_{d=d_K} w^d_t \xi^{d}_{t,\pr{d}},\quad \xi^{d}_{t,\pr{d}}\in \Pcal(\mathcal{X} \times A)\\
    \xi^{d}_{t,\pr{d}}(C) & \coloneqq \begin{cases}
        \nu^{x_{t-\pr{d}}, a_{t-\pr{d}}\mid d}_t(C), & d \leq \pr{d}; \\
       (\nu^{d,x,a}_t \star_{d - \pr{d}} p^{d, \pr{d}}_{\mu}) (C \times A^{\pr{d}}), & d > \pr{d},
    \end{cases}\\
     p_{d, \pr{d}}^{\mu} &\coloneqq (p^{\mu_{t,d}},\ldots, p^{\mu_{t, \pr{d}-1}}),\ C \in \mathcal{B}(\mathcal{X} \times A),
\end{align*}
where $\nu^{d,x,a}_t$ is the marginal of $\nu^d_t$ on the $(x_{t-d}, a_{t-d},\ldots , a_{t-1})$--coordinates, and the $\star$ notation is as in \Cref{def:star_operator}. Finally, we define $\bm{\mu}^{\nu}_t \coloneqq (\mu_{t,d})^{d_0}_{d=0}$ and $\mathcal{M}(\nu_0) \coloneqq \mu_{t,0}$.
\end{defn}

The best-response map and measure flow map of a regularised MFNE in this case is similar to that in \Cref{defn:regularised MFNE-MCDM}, the only difference being that $\mu_{t,d}$ are measures over the space $\mathcal{X} \times A$.

\begin{defn}
Let $\bm{\nu}=(\nu_t)_t \in \Pcal(\tilde{\mathcal{Y}})^{\infty}$ and $\eta > 0$. Define:
\begin{enumerate}
    \item[(i)] The best-response map $\Phi_{\eta}: \Pcal(\tilde{\mathcal{Y}})^{\infty} \to \mathcal{A}_{\rm DM}$, given by
        \begin{align*}
            \Phi_{\eta}(\bm{\nu})_t(\de u\mid y) = \frac{\exp\left(  Q^{*}_{\eta,\bm{\nu}}(t,y,u)/\eta\right)}{\int_U \exp\left( Q^{*}_{\eta,\bm{\nu}}(t,y,u)/\eta\right) q(\de u)} q(\de u).
        \end{align*}
    \item[(ii)] $\Psi^{\mathrm{aug}}: \mathcal{A}_{\rm DM} \to \Pcal(\tilde{\mathcal{Y}})^{\infty}$, the measure flow map as defined previously, where $\Psi^{\mathrm{aug}}(\pi)_0 = \nu_0$ and for $t \geq 0$,
        \begin{align*}
            \Psi^{\mathrm{aug}}(\pi)_{t+1}(\cdot) =\int_{\tilde{\mathcal{Y}}} \int_U p_y\Big(\cdot \mid y, u, \bm{\mu}_t^{\Psi^{\mathrm{aug}}(\pi)} \Big) \pi_t(\de u \mid y) \Psi^{\mathrm{aug}}(\pi)_t(\de y).
        \end{align*}
    \item[(iii)] A regularised MFNE for the MCDM problem $(\pi^{*}, \bm{\nu}^{*})\in \mathcal{A}_{\rm DM} \times \Pcal(\tilde{\mathcal{Y}})^{\infty}$ is given by a fixed point $\bm{\nu}^{*}$ of $\Psi^{\mathrm{aug}} \circ \Phi_{\eta}$, for which $\pi^{*}  = \Phi_{\eta}(\bm{\nu}^{*})$ (best response map) and $\bm{\nu}^{*} = \Psi^{\mathrm{aug}}(\pi^{*})$ (measure flow induced by policy) holds.
\end{enumerate}
\end{defn}
    
Then, we obtain the analogous statement for a unique regularised fixed point, provided that the regulariser parameter $\eta$ is large enough.

\begin{thm}
    Suppose that the transition kernel $p$ and reward function $r$ are Lipschitz continuous with constants $L_p$ and $L_r$, with bounds $M_p$ and $M_r$ respectively. Recall the constants $L_P$, $L_R$ and $L_M$ as defined in \Cref{sec:MFG}, with the appropriate replacement of state-action joint measure flow for $\bm{\mu}$. Let $\zeta$ and $\eta^{*}$ be constants such that $\zeta > 2L_P L_M +2$, and $\eta^{*} >  \frac{2M_R}{-(1-\gamma)\log \gamma}$, and recall the constants
    \begin{align*}
            q^{*} = \frac{M_R}{1-\gamma},\quad l_{\eta^{*}} = \frac{L_M(L_R + 2 \gamma q^{*}L_P)}{1- \gamma \exp\left(\frac{2q^{*}}{\eta^{*}}\right)}, \quad L_{\Psi} = \frac{2L_P}{2L_P L_M +1} \left( \frac{\zeta}{\zeta - 2L_P L_M -2} + \frac{1}{\zeta-1} \right).
    \end{align*}
    Then, for any $\eta > \sqrt{2q^{*} l_{\eta^{*}} L_{\Psi}}$, the fixed point operator $\Psi^{\mathrm{aug}} \circ \Phi_{\eta}$ is a contraction mapping on the space $(\Pcal(\tilde{\mathcal{Y}})^{\infty}, \delta_{\infty})$, where the constant $\zeta$ in the metric $\delta_{\infty}$ in \eqref{eq_flow_metric} is as chosen above. In particular, there exists a unique fixed point for $\Psi^{\mathrm{aug}} \circ \Phi_{\eta}$, which is a regularised MFNE for the MFG-MCDM problem.
\end{thm}

\begin{proof}
   The proof is essentially identical to that of \Cref{thm:contraction_reg_mcdm}, adjusted for the fact that the Lipschitz constants are now with respect to a joint state-action measure flow $\bm{\mu}$.
\end{proof}

\section{Computational example in epidemiology}\label{sec_numerics}

In this section, we apply our MFG-MCDM approach to a discrete-time SIS (susceptible-infected-susceptible) model, adapted and extended from that in \cite{cui2021approximately} to incorporate costly, delayed and controlled observations, and also an extension by which the evolution of the epidemic depends on the joint distribution of states (regarding infection) and actions (regarding social distancing) chosen by the population. The simple nature of the model allows a straightforward interpretation of the MFNE, illustrating the effects of costly information acquisition on population behaviour. Other epidemiological models include the SIR (susceptible-infected-recovered) variant, as well as models involving demographic and geographical heterogeneity, such as \cite{cont2021modelling}.

In our numerical experiments, we allow the dynamics to depend on the joint distribution of states and actions, which is crucial in capturing the basic effects from infection and distancing.
In this model, a virus circulates amongst the population, and each agent can take on two states: susceptible ($S$), or infected ($I$). At each moment, the agent can decide to go out ($U$) or socially distance ($D$). Thus we have the state space $\mathcal{X} = \{S, I\}$ and action space $A = \{U, D\}$. The probability of an agent being infected whilst going out is assumed to be proportional to the fraction of infected people who do not socially distance. Once infected, they have a constant probability of recovering at each unit in time. We use the following parameters for the transition kernel:
\begin{align*}
    p(S \mid I, U ) = p(S\mid I, D)  = p_{\mathrm{rec}}, \quad
    p(I \mid S, U) = 0.9^2  \mu_t(I, U), \quad
    p(I \mid S, D) = 0,
\end{align*}
where $p_{\mathrm{rec}}$ is the probability of recovery per day. 
There is a cost for socially distancing, and a larger cost for being infected. As we are considering a maximisation problem, we write the cost as a negative reward, which is given by

\begin{equation*}
\begin{aligned}
r(S,U) &= 0,      & r(S,D) &= -0.5, \\
r(I,U) &= -1.5,   & r(I,D) &= -1.0.
\end{aligned}
\end{equation*}

\begin{algorithm}[t]
\caption{Prior descent applied to MFG-MCDM}\label{prior_descent}
\SetKwInOut{Input}{Input}
\SetKwInOut{Output}{output}
\SetKwRepeat{Do}{do}{while}
\Input{Initial distribution $\bm{\nu}_0 \in \Pcal(\tilde{\mathcal{Y}})$, prior policy $q \in \mathcal{A}_{\rm DM}$, $tol$}
\Input{Number of iterations per loop $I$, regularisation parameter $\eta >0$, truncation time $T$.}
\While{$RelExpl> tol$}{
    \For{$i = 0, 1, \ldots, I$}{
        \For{$t = 0, 1, \ldots, T-1$}{
        Compute $\bm{\mu}^{\nu}_t = (\mu_{t, d})^{d_0}_{d=0} \in \Pcal(\mathcal{X})^{d_0+1}$.\\
        Compute the regularised $Q$-function $Q^{*}_{\eta,\bm{\nu}}(t,y,u)$ for fixed measures $\bm{\mu}^{\nu}_t$.\\
        Compute the softmax policy: $\pi^{\mathrm{soft}} \leftarrow \Phi_{\eta}(\bm{\nu})$.\\
        Compute the induced mean-field: $\bm{\nu}_{t+1} \leftarrow \Psi^{\mathrm{aug}}(\pi^{\mathrm{soft}})_{t+1}$.
        }
      }
      $ q \leftarrow \pi^{\mathrm{soft}}_{\eta}$.
      }
\end{algorithm}

In addition to the above, we introduce the notion of waiting times to receive test results. Assume that during a pandemic, the population undergoes daily testing in order to determine whether they are infected or not. Here we assume the availability of two testing options, the free option which requires a 3-day turnaround, and a paid option which offers a next-day result. We will also refer to the paid option as quick testing. We thus have for our model
\begin{align*}
    \mathcal{D} = \{d_0 = 3,\ d_1 = 1\}, \quad \mathcal{C} = \{c_0 = 0,\ c_1 >0 \},
\end{align*}
where we shall consider different values of $c_1$ in our numerical experiments.

For the computation of MFNEs for our model, we utilise the \texttt{mfglib} Python package \cite{mfglib}. We incorporate our own script for the mapping $\bm{\nu}\mapsto \bm{\mu}^{\nu}$ so that the existing library can be adapted for the MFG-MCDM, and in particular for the computation on the augmented space. We first initialise with a uniform policy as the reference measure $q$, and repeatedly apply the mapping $\Psi^{\mathrm{aug}} \circ \Phi_{\eta}$ for a range of values of the regularisation parameter $\eta$. As a benchmark to test for the convergence towards a regularised MFNE, we utilise the exploitability score, which, for a policy $\pi$, is defined by
\begin{align*}
    \mbox{Expl} (\pi) \coloneqq \max_{\tilde{\pi}} J_{\Psi^{\mathrm{aug}}(\pi)}(\tilde{\pi}) - J_{\Psi^{\mathrm{aug}}(\pi)}(\pi).
\end{align*}
The exploitability score measures the suboptimality gap for a policy $\pi$ when computed with the measure flow induced by the map $\Psi^{\mathrm{aug}}$. An exploitability score is 0 if and only if $\pi$ is an MFNE for the MFG, and a score of $\varepsilon$ indicates that $\pi$ is an $\varepsilon$-MFNE. We refer to the literature such as \cite{mf-omo, perolat2021scaling, lauriere2022learning} for a more detailed discussion. As the exploitability score depends on the rewards and the initial policy, we consider instead the relative exploitability scaled by the initial value.

\begin{figure}[t!]
    \centering
    \begin{minipage}{0.55\textwidth}
        \centering
        \includegraphics[width=0.95\textwidth]{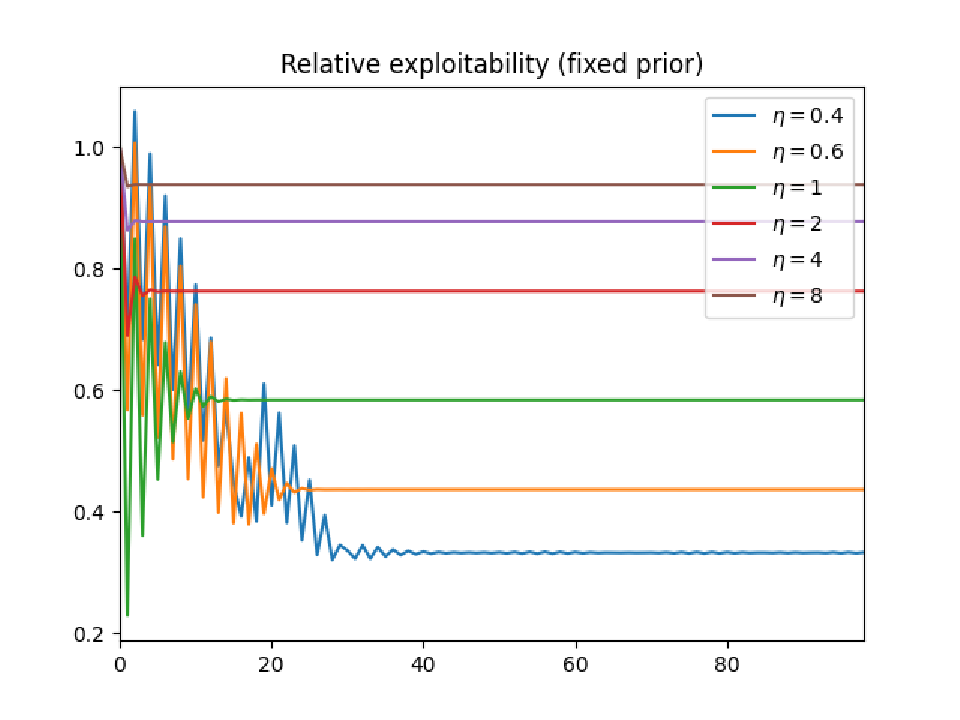}
    \end{minipage}
    \vspace{1em}
    \begin{minipage}{0.45\textwidth}
        \centering
        \includegraphics[width=1.1\textwidth]{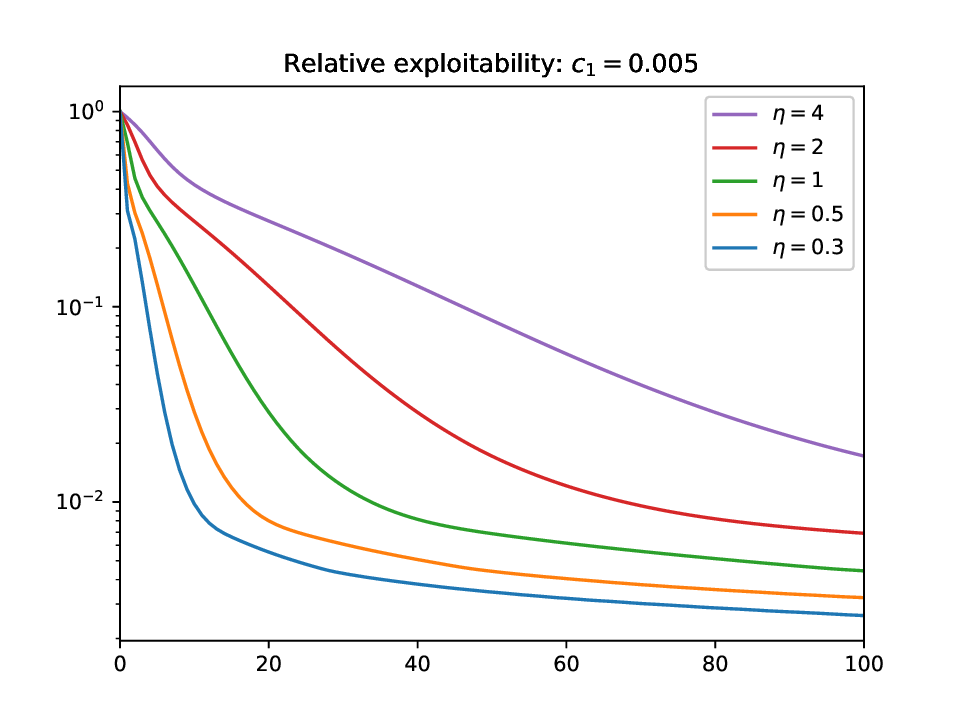}
    \end{minipage}
    \begin{minipage}{0.45\textwidth}
        \centering
        \includegraphics[width=1.1\textwidth]{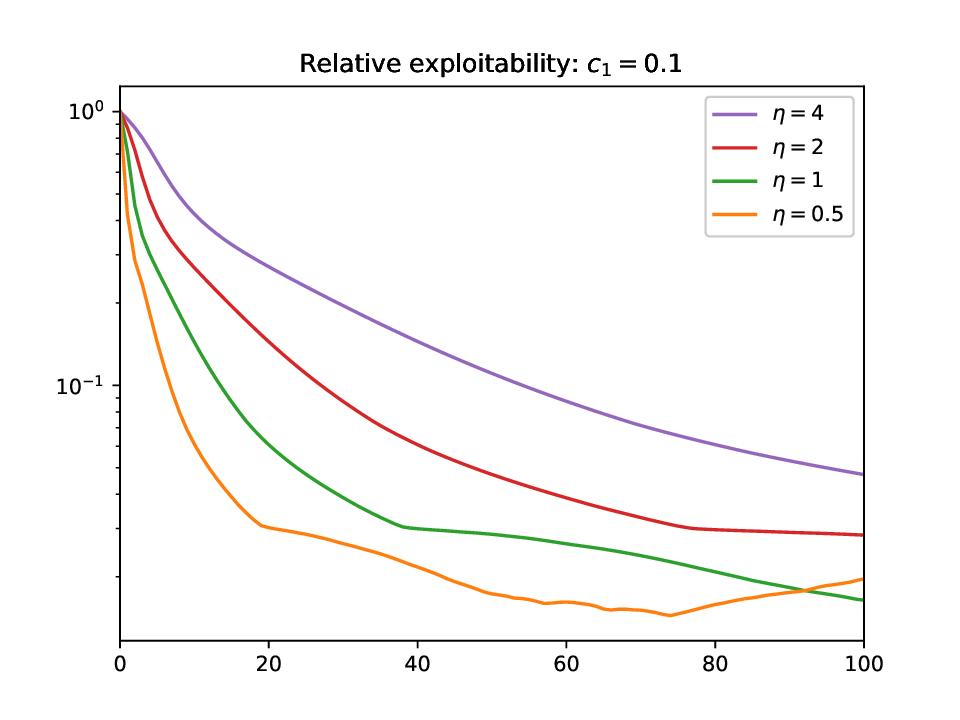}
    \end{minipage}
    \caption{
    Top: relative exploitability for $c_1 = 0.5$ when applied to a uniform policy as a reference measure, fixed across all iterations (on the horizontal axis). Bottom row: relative exploitability for the prior descent algorithm for $c_1 = 0.005$ and $c_1 = 0.1$.
    }
    \label{exploitability}
\end{figure}
The top graph of \Cref{exploitability}
shows the convergence of the relative exploitability, with the uniform policy as reference measure, fixed across all iterations. We see that for lower values of $\eta$, the algorithm converges to a lower relative exploitability value. This corresponds to the fact that the regularised MFG  approximates the non-regularised MFG more closely for $\eta$ being low. However, lower values of $\eta$ require a larger number of iterations for convergence. For $\eta$ being less than 0.2, the algorithm does not even appear to converge in our tests but explodes numerically (not plotted in the graph). This demonstrates an inherent limitation of the use of regularisation: While a sufficiently high value of $\eta$ guarantees convergence of the iteration, 
this may simultaneously lead to a MFNE of the regularized MFG which may approximate the non-regularized MFG problem poorly. Moreover, searching for a suitable value of $\eta$ is computationally expensive.

To mitigate the above issues, we utilise the prior descent algorithm \cite{cui2021approximately}. Here, the reference measure is dynamically updated, by using the policy obtained from the previous iteration to determine the reference measure for the next iteration. The reference measure can also be updated after a number of iterations instead, creating a double loop for the algorithm. We summarise the prior descent algorithm for the MFG with control of information speed in \Cref{prior_descent}. The relative exploitability score is plotted in the bottom row of \Cref{exploitability}. We see that prior descent vastly outperforms the case of using a fixed prior. In \cite{cui2021approximately}, the prior descent algorithm is further improved by using the heuristic $\eta_{j+1} = \eta_j \cdot c$ for some constant $c > 1$, gradually increasing the regularisation to aid convergence. We also applied this heuristic for our problem, but for our case we do not see significant differences compared to initialising with large fixed values of $\eta$.

\begin{figure}[t!]
    \centering
    \begin{minipage}{0.45\textwidth}
        \centering
        \includegraphics[width=1.1\textwidth]{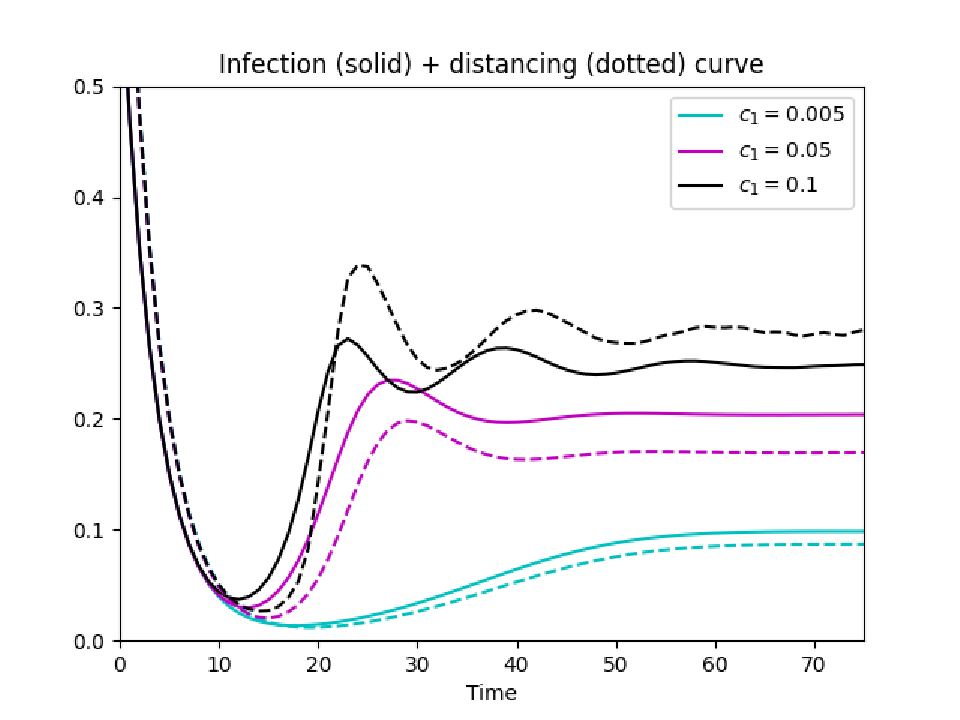}
    \end{minipage}
    \begin{minipage}{0.45\textwidth}
        \centering
        \includegraphics[width=1.1\textwidth]{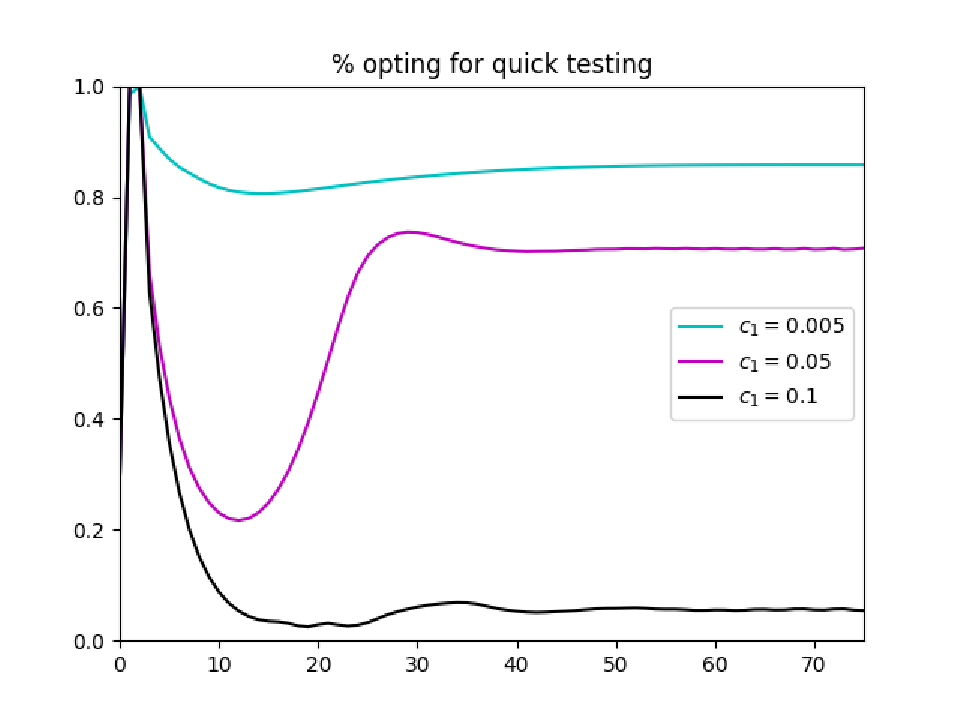} 
    \end{minipage}
    \caption{
    Left: lower testing costs are associated with lower levels of infection, with oscillations around equilibrium when the cost is high;  a delayed response is then apparent. Right: proportion opting for fast, more expensive test.}
    \label{fig:numerics}
\end{figure}

\Cref{fig:numerics} depicts the effects of the testing cost on the population behaviour at a Nash equilibrium. The value of $p_{\mathrm{rec}}$ here is $0.25$. The left graph shows the proportion of the population infected over time for $c_1 = 0.005, 0.05, 0.1$. The first observation is that a cheaper cost for quick testing corresponds to lower levels of infection, as a larger percentage of the population receives a more accurate estimate of their infectiousness, allowing a more timely optimal response. Moreover, a delayed response of social distancing leads to a higher amount of people infected and going out, which increases the probability of infection for a susceptible individual going out. This leads to the oscillations for the case of $c_1=0.1$ before it settles at an equilibrium. This delayed response can also be seen from 
the proportion of people who are socially distancing, superimposed on top as the dotted curves. For the low cost of $c_1 = 0.005$, the distancing curve follows the infection curve closely, indicating appropriate timely responses from the population. In contrast, there is a clear horizontal shift in the curves for the other two cases. In particular, for the case $c_1 = 0.1$, the peaks and troughs are separated by approximately three days, which corresponds to the fact that the optimal policy is free testing with a very high probability.

\begin{figure}[t!]
    \centering
    \includegraphics[width=0.75\textwidth]{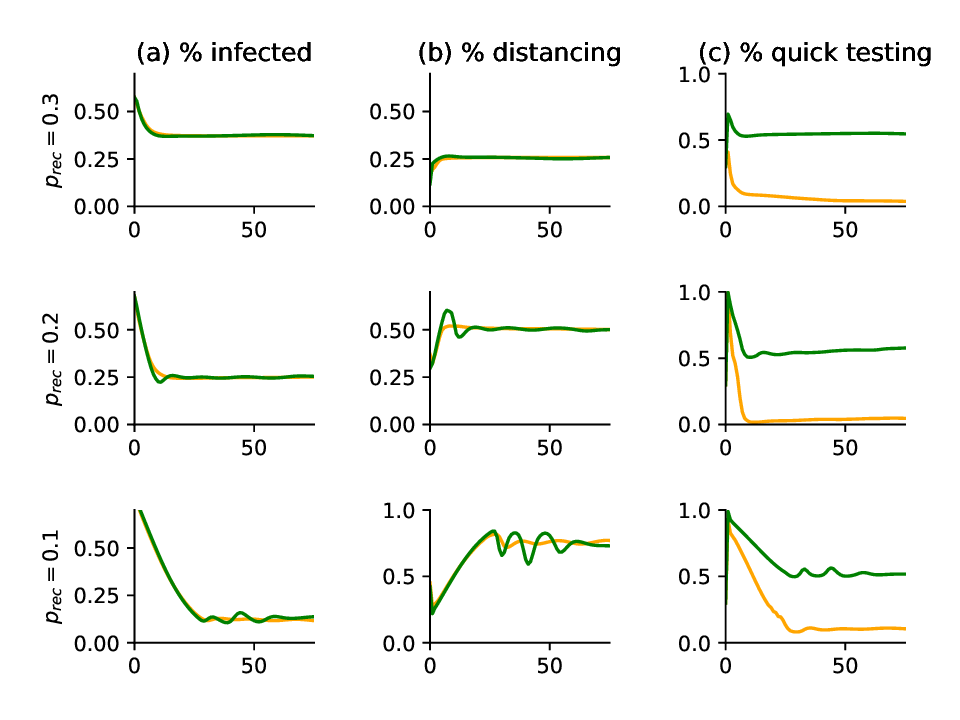}
    \caption{Population behaviour at MFNE for an inverted reward, plotted against the time duration of the problem on the horizontal axis, with $r(I,U) = -1.0$ and $r(I,D)=-1.5$ modelling selfish behaviour. Green line: $c_1 = 0.001$. Yellow line: $c_1 = 0.05$. Lower test cost here leads to higher infection as individuals are further incentivised to go out.}
    \label{fig:inverted_r}
\end{figure}

It is worth noting that the behaviour of the population at equilibrium depends on the choice of the reward function. In the above, we chose the reward such that an infected person is encouraged to socially distance, which models a population compliant with mitigation measures. If instead $r(I,U) > r(I,D)$, as is in  \cite{cui2021approximately}, individuals are incentivised to go out regardless of infection status, and distancing is preferred only when $\mu_t(I,U)$ is high. The population behaviour at equilibrium under this reward structure is shown in \Cref{fig:inverted_r}. We see higher levels of infection at equilibrium, as well as a relative insensitivity to the cost of quick testing, in contrast to the case of a compliant population. We also see more oscillatory distancing behaviour associated with low test costs. Finally, we note that a possible extension of the above model can include a heterogeneous population with proportional rewards that reflect varying preferences and attitudes within society.

\begin{rem}
Given the above example on medical testing, it appears that choices about more speedy but costly test results, for example as required for air travel during the COVID-19 pandemic, are often to be decided upon not at the time when the test result is to be \emph{received}, but at the (earlier) time when the test sample is being \emph{taken}. Although the MFG in the present paper is formulated for the former variant, let us explain briefly how our analysis indeed provides the correct solution also for the latter. The former variant, though possibly artificial for applications, is more easily formulated as a dynamic programming problem. Conversely, a decision problem with pre-commitment on the acquisition of costly information with adjustable delay periods, appears to be a more complicated formulation. To see that the optimal MFNE policy from our general framework formulated for the former variant provides simultaneously the solution for the latter variant, it suffices to note that the optimal policy at an MFNE, according to our augmented state space setting whose variables are showing in \Cref{defn:augmented_space}, is based solely on the latest state $x$ that has been observed, and the string of actions $(a)_{-d}^{-1}$ taken by the optimal policy since then, which all have been made subsequently solely based on $x$. That means that the representative agent already knows at the time of the medical test, the future date at which they will  receive the test results. In other words, the strategy can be interpreted more naturally as a decision about information acquisition being made actually at the timing of the medical test. Likewise reasoning applies beyond the concrete example of the present section. In regards to the acquisition costs $c_0,\ldots,c_K$ for timely information access, it would clearly be possible to modify the setup for payment dates of information access to be at the earlier (e.g. testing) date instead of at the later date (test results received). This would simply require a suitable discount such that the time-adjusted values of respective payments coincide, that means taking $c_i \gamma^{i}$ if the payment is to be done at the earlier time of medical testing. 

\end{rem}

\subsection*{Acknowledgments}
J. Tam is supported by the EPSRC Centre for Doctoral Training in Mathematics of Random Systems: Analysis, Modelling and Simulation (EP/S023925/1). D. Becherer is supported by German Science Foundation DFG - Berlin-Oxford IRTG 2544 "Stochastic Analysis in Interaction" - Project-ID 41020858, and acknowledges funding by DFG - CRC/TRR 388 "Rough Analysis, Stochastic Dynamics and Related Fields" - Project ID 516748464.\\

The authors would like to thank Nils Mattiss from HU Berlin for his assistance towards the computational experiments, as well as the three anonymous referees for their feedback, which led to a substantial improvement of this paper. J. Tam gratefully acknowledges the support of the Department of Economics at the University of Verona, where the majority of the revisions were completed.

\bibliographystyle{abbrvurl}
\bibliography{ref}
\end{document}